\documentclass[oneside, reqno]{amsart}

\usepackage[utf8]{inputenc}
\usepackage{microtype}
\usepackage[hyperfootnotes=false]{hyperref}
\usepackage{tikz}
	\usetikzlibrary{matrix,arrows,decorations.pathmorphing,positioning,intersections,calc}
\usepackage{tikz-cd}
\usepackage{todonotes}
\tikzset{node distance=2cm, auto}
\usepackage{calc}
\usepackage{footmisc}

	\usepackage[frenchmath]{mathastext}
	\usepackage{MnSymbol}
	\usepackage{eucal}
	\usepackage{xfrac}
	\usepackage{textcomp}

	\usepackage{amsthm}
	\usepackage{thmtools}
	\numberwithin{equation}{section}

    \theoremstyle{plain}
	\newtheorem*{lem*}{Lemma}
	\newtheorem*{thrm*}{Theorem}
    \newtheorem*{prop*}{Proposition}
    \newtheorem*{cor*}{Corollary}   

    \theoremstyle{definition}
    \newtheorem*{ex*}{Example}
	\newtheorem*{fact*}{Fact}
	\newtheorem*{not*}{Notation}
	\newtheorem*{defn*}{Definition}

    \theoremstyle{remark}
	\newtheorem*{rmk*}{Remark}

    \theoremstyle{definition}
	\newtheorem{defn}{Definition}[section]
	\newtheorem{ex}[defn]{Example}
    
    \newtheorem{ass}[defn]{Assumption}

    \theoremstyle{remark}
    
	\newtheorem{rmk}[defn]{Remark}

    \theoremstyle{plain}
	\newtheorem{prop}[defn]{Proposition}
    \newtheorem{thrm}[defn]{Theorem}
	\newtheorem{cor}[defn]{Corollary}
	\newtheorem{lem}[defn]{Lemma}

	\newcommand{\cat}{\mathtt}
	\newcommand{\st}[1]{\underline{\mathtt{#1}}}		
	\newcommand{\cal}{\mathcal}
	
		\newcommand{\A}{\cat{A}}
		\newcommand{\B}{\cat{B}}
		\newcommand{\cC}{\cat{C}}
		\newcommand{\G}{\cat{G}}
		\newcommand{\T}{\cat{T}}
		\newcommand{\W}{\cat{W}}
		\newcommand{\F}{\cat{F}}
		\newcommand{\jP}{\cat{P}}
		\newcommand{\Per}{\cat{Per}}
		\DeclareMathOperator{\Coh}{Coh}	
		\DeclareMathOperator{\Perf}{Perf}	
		\DeclareMathOperator{\St}{St}
	\newcommand{\M}{\cal M}
	
	\newcommand{\cCoh}{\underline{\Coh}}

	\newcommand{\Hilb}{\text{Hilb}}
	
	\newcommand{\Pilb}{\text{Pilb}}
	
	\newcommand{\Pair}{\cat{Pair}}
	\newcommand{\uPair}{\st{Pair}}
	\newcommand{\X}{\mathcal X}
	\newcommand{\Mum}{\mathfrak{Mum}}
	\DeclareMathOperator{\pt}{pt}
	
	\let\ss\relax
	\DeclareMathOperator{\ss}{ss}
	\newcommand{\uA}{\underline{\cat{A}}}
	
	\newcommand{\uC}{\underline{\cat{C}}}

	\newcommand{\uT}{\underline{\cat{T}}}
	\newcommand{\uF}{\underline{\cat{F}}}
    \newcommand{\uP}{\underline{\cat{P}}}
    \newcommand{\uW}{\underline{\cat{W}}}
    \newcommand{\uM}{\underline{\cal{M}}}

	\newcommand{\C}{\mathbf C}
	\newcommand{\Q}{\mathbf{Q}}
	\newcommand{\R}{\mathbf{R}}
	\newcommand{\Z}{\mathbf Z}
	\newcommand{\N}{\mathbf N}
	\newcommand{\D}{\mathbf D}
	\newcommand{\ID}{\mathbf 1}
	\newcommand{\LL}{\mathbf{L}}

	\DeclareMathOperator{\rk}{rk}
	\DeclareMathOperator{\ch}{ch}

	\DeclareMathOperator{\coker}{coker}
	\DeclareMathOperator{\cok}{coker}
    \DeclareMathOperator{\im}{im}

	\DeclareMathOperator{\Ext}{Ext}
	
	\DeclareMathOperator{\Hom}{Hom}
	
	\DeclareMathOperator{\CHom}{\!CHom}
	\DeclareMathOperator{\Seq}{Seq}
    \DeclareMathOperator{\Aut}{Aut}
	\DeclareMathOperator{\lHom}{\underline{Hom}}
	\DeclareMathOperator{\lRHom}{R\underline{Hom}}

	\DeclareMathOperator{\Hsc}{H_{\text{sc}}}
	\DeclareMathOperator{\Sym}{Sym}
	
	\newcommand{\Pic}{\text{Pic}}

	\newcommand{\hE}{\mathcal{E}}

	\newcommand{\hI}{\mathcal{I}}
	
        \newcommand{\hC}{\mathcal{C}}
        \newcommand{\hD}{\mathcal{D}}
        \newcommand{\hM}{\mathcal{M}}
        \newcommand{\hU}{\mathcal{U}}

    \DeclareMathOperator{\eff}{eff}
    \DeclareMathOperator{\mr}{mr}
	
    \DeclareMathOperator{\exc}{exc}
    \DeclareMathOperator{\nexc}{n-exc}
    \DeclareMathOperator{\reg}{reg}
    \DeclareMathOperator{\Var}{Var}
	
    \DeclareMathOperator{\gr}{gr}
    \DeclareMathOperator{\Sc}{sc}
    \DeclareMathOperator{\eex}{ex}
	
	\newcommand{\os}{\mathcal O}
	\newcommand{\hO}{\mathcal O}

	\newcommand{\Tz}{\cat T_{\zeta, (\gamma, \eta)}}
	\newcommand{\Fz}{\cat F_{\zeta, (\gamma, \eta)}}
	\newcommand{\tT}{\tilde{\T}}
	\newcommand{\fF}{\tilde{\F}}
	\newcommand{\Ti}{\cat T_{\zeta, 0}}
	\newcommand{\Fi}{\cat F_{\zeta, 0}}
	\newcommand{\Tt}{\cat T_{\theta, \gamma}}
	\newcommand{\Ft}{\cat F_{\theta, \gamma}}
	\newcommand{\Tzge}{\T_{\zeta, (\gamma,\eta)}}
	\newcommand{\Fzge}{\F_{\zeta, (\gamma,\eta)}}
	\newcommand{\zge}{\zeta,(\gamma,\eta)}
	\newcommand{\tg}{\theta}

	\DeclareMathOperator{\Quot}{Quot}
	\DeclareMathOperator{\Spec}{Spec}
	\DeclareMathOperator{\supp}{supp}

	\makeatletter\let\P\@undefined\makeatother
	\newcommand{\P}{\mathbf{P}}

	\newcommand{\into}{\hookrightarrow}
	\newcommand{\onto}{\twoheadrightarrow}

	\def\ie{i.e.}
	\def\eg{e.g.}

\begin{document}

	
\begin{abstract}
	We prove the crepant resolution conjecture for Donaldson--Thomas invariants of hard Lefschetz CY3 orbifolds, formulated by Bryan--Cadman--Young, interpreting the statement as an equality of rational functions.
	In order to do so, we show that the generating series of stable pair invariants on any CY3 orbifold is the expansion of a rational function.
	As a corollary, we deduce a symmetry of this function induced by the derived dualising functor.
	Our methods also yield a proof of the orbifold DT/PT correspondence for multi-regular curve classes on hard Lefschetz CY3 orbifolds.
\end{abstract}

\title{A proof of the Donaldson--Thomas crepant resolution conjecture}

\author{Sjoerd Viktor Beentjes}
\address{School of Mathematics and Maxwell Institute,
University of Edinburgh,
James Clerk Maxwell Building,
Peter Guthrie Tait Road, Edinburgh, EH9 3FD,
United Kingdom}
\email{sjoerd.beentjes@gmail.com}

\author{John Calabrese}
\address{Rice University MS136,
6100 Main Street, 
Houston TX 77251-1892,
United States of America} 
\email{jcalabrese@pm.me}

\author{J{\o}rgen Vold Rennemo}
\address{Department of Mathematics,
University of Oslo,
PO Box 1053 Blindern,
0316 Oslo,
Norway}
\email{jvrennemo@gmail.com}

\date{\today}
\maketitle



\section{Introduction}
Donaldson--Thomas (DT) invariants, introduced in \cite{MR1818182}, are deformation-invariant numbers that virtually enumerate stable objects in the derived category of coherent sheaves on a smooth projective threefold.
A particularly interesting case is that of DT invariants of curve-like objects, such as ideal sheaves of curves or the stable pairs of \cite{MR2545686}, on a Calabi--Yau threefold.

The crepant resolution conjecture for Donaldson--Thomas invariants, originally conjectured by Bryan, Cadman, and Young in \cite{MR2854183}, is a comparison result that predicts a relation between these curve-type DT invariants of two different Calabi--Yau threefolds.
The first threefold is an orbifold, and the second threefold is a crepant resolution of singularities of the coarse moduli space of the first.

In this paper, we prove the crepant resolution conjecture for Donaldson--Thomas invariants, interpreting the conjecture of \cite{MR2854183} as an equality of rational functions, rather than an equality of generating series.
In Appendix \ref{sec:Appendix_ExampleCRC}, we provide a simple example demonstrating the necessity of the rational function interpretation. 
\newtheorem*{thm:MainTheorem}{Theorem \ref{thm:MainTheorem}}
\begin{thm:MainTheorem}
  The crepant resolution conjecture for Donaldson--Thomas invariants holds as an equality of rational functions.
\end{thm:MainTheorem}

Let $\X$ denote the CY3 orbifold, and $D(\X)$ its \emph{bounded} coherent derived category.
Our proof goes via wall-crossing in the motivic Hall algebra.
This strategy has previously been applied to establish comparison theorems for DT invariants, see e.g.~\cite{todajams,MR2683216,MR2813335,MR3518373,MR3449219,2016arXiv160107519T,Oberdieck_ReducedStablePairInvariants}.

However, our arguments are novel in at least two ways:
\begin{enumerate}
\item \emph{The Euler pairing is non-trivial.}
  A key component of Joyce's wall-crossing formula is the Euler pairing $\chi(E,F)$ of the objects $E,F \in D(X)$ whose slopes cross.
  Previous results have made essential use of the fact that on a 3-dimensional variety $X$, the Euler pairing between two sheaves with 1-dimensional support vanishes.
  The corresponding statement fails on the orbifold $\X$, and this significantly complicates the application of the wall-crossing formula.

  This non-vanishing is related to the fact that the derived equivalence $D(Y) \cong D(\X)$ of the McKay correspondence does not preserve dimensions of supports, \eg curve-like objects on $\X$ can be sent to surface-like objects on the crepant resolution $Y$.
  From this perspective, the novelty of our argument is that we compute a wall-crossing involving surface-like objects.
\item \emph{Rationality of the generating series is crucial.}
  The comparison result between our DT-type generating functions only holds after a re-expansion of a rational generating function, analogous to the analytic continuation for the Gromov--Witten crepant resolution conjectures; see \cite{MR2483931,MR2529944}.
  This means that for a fixed numerical class $(\beta,c)$, the conjecture does not state a relation between the DT invariants of class $(\beta,c)$ on $\X$ and $Y$, but rather between the collection of all DT invariants with fixed curve class $\beta$.

  This phenomenon is analogous to the $q \leftrightarrow q^{-1}$ symmetry of the stable pairs generating function on a variety, which is a symmetry of rational functions, not of generating series.
  Again, the non-triviality of the Euler pairing makes proving these relations more complicated in our case.
\end{enumerate}

\subsection{Statement of results}
Throughout, let $\X$ be a 3-dimensional Calabi--Yau orbifold, by which we mean a smooth Deligne--Mumford stack with $\omega_{\X} \cong \hO_{\X}$, $H^{1}(\X, \hO_{\X}) = 0$, and projective coarse moduli space $g\colon \X \to X$.

\subsubsection{Rationality of stable pair invariants}
The numerical Grothendieck group $N(\X)$ is the Grothendieck group of $D(\X)$ modulo the radical of the Euler pairing:
\begin{align*}
  \chi(E,F) = \sum_i (-1)^i \dim \Ext^i_{\X}(E,F)
\end{align*}
where $E,F \in D(\X)$.
This is a free abelian group of finite rank.
We write $N_0(\X)$ and $N_{\leq 1}(\X)$ for the subgroups generated by sheaves supported in dimension 0 and $\le 1$ respectively, and write $N_{1}(\X) = N_{\le 1}(\X)/N_{0}(\X)$.
It is convenient to pick\footnote{In contrast to the case of varieties, where the holomorphic Euler characteristic gives a canonical splitting, there need not exist a canonical choice of splitting in the case of orbifolds.}
a splitting of the natural inclusion $N_0(\X) \into N_{\leq 1}(\X)$,
\begin{equation}\label{eq:SplittingX}
  N_{\leq 1}(\X) = N_1(\X) \oplus N_0(\X),
\end{equation}
and so we denote a class in $N_{\le 1}(\X)$ by $(\beta,c)$, with $\beta \in N_{1}(\X)$ and $c \in N_{0}(\X)$.\footnote{When $M$ is a smooth and irreducible variety, $N_0(M) \cong \Z$ is generated by the class of a point.
However, when $M$ is a DM stack, $N_0(M)$ has higher rank, since skyscraper sheaves supported at stacky points with different equivariant structures have in general different numerical classes.\label{foot:N0_DM_higher_rank}}

In \cite{MR2600874}, Behrend defines for any finite type scheme $M$ a constructible function $\nu \colon M \to \Z$, and shows that if $M$ is proper and carries a symmetric perfect obstruction theory, with associated virtual fundamental class $[M]^{\text{vir}}$, then
\[
  \int_{[M]^{\text{vir}}}1 = e_{B}(M) \coloneq \sum_{n\in \Z}e_{\text{top}}(\nu^{-1}(n)).
\]
Given $(\beta,c) \in N_{\leq 1}(\X)$, one defines the corresponding DT curve count as
\begin{align}\label{def:DT}
  DT(\X)_{(\beta,c)} \coloneq e_B\bigl( \Quot_\X(\hO_{\X}, (\beta,c)) \bigr) \in \Z,
\end{align}
where $\Quot_\X(\hO_{\X}, (\beta,c))$ is the projective scheme parametrising quotients $\hO_\X \onto \hO_Z$ with $[\hO_Z] = (\beta,c)$.

Given a curve class $\beta \in N_{1}(\X)$, define the generating functions
\begin{equation}\label{eq:DT_GeneratingFunctions}
  \begin{split}
    DT(\X)_\beta &= \sum_{c \in N_0(\X)} DT(\X)_{(\beta,c)} q^c, \\
    DT(\X)_0 &= \sum_{c \in N_0(\X)} DT(\X)_{(0,c)} q^c.
  \end{split}
\end{equation}
In \cite{MR2545686}, Pandharipande and Thomas introduced new curve-counting invariants by considering stable pairs.
A stable pair is a two-term complex concentrated in degrees $-1$ and $0$ of the form
\begin{equation*}
  E = (\hO_\X \stackrel{s}{\to} F) \in D(\X)
\end{equation*}
such that $F$ is pure of dimension 1 and $\coker(s) = H^{0}(E)$ is at most 0-dimensional.
There is a fine moduli space $\Pilb_\X(\beta,c)$ parametrising stable pairs $(\hO_\X \to F)$ of class $[F] = (\beta,c)$.
The corresponding PT curve count is defined as
\begin{equation}\label{def:PT}
  PT(\X)_{(\beta,c)} \coloneq e_B\left( \Pilb_\X(\beta,c) \right) \in \Z.
\end{equation}
The associated generating function for a fixed curve class $\beta \in N_1(\X)$ is
\begin{equation*}
  PT(\X)_\beta \coloneq \sum_{c \in N_0(\X)} PT(\X)_{(\beta,c)} q^c.
\end{equation*}
Note that for the empty curve class $\beta = 0$ we have $PT(\X)_0 = 1$. 
\begin{rmk}
  The generating functions above, and all others we consider in this paper, lie in certain \emph{completions} of the ring $\Q[N_{0}(\X)]$.
  We introduce some terminology to describe these completions.
  Let $L \colon N_{0}(\X) \to \R$ be a linear function.
  We say that an infinite formal sum $\sum_{c\in N_{0}(\X)} a_c q^{c}$, is \emph{Laurent} with respect to $L$ if for any $x \in \R$, the set of $c$ such that $a_c \neq 0$ and $L(c) \le x$ is finite.
  We write $\Q[N_{0}(\X)]_{L}$ for the ring of formal expressions which are Laurent with respect to $L$.
  
  Given a rational function $f \in \Q(N_{0}(\X))$, it is represented by at most one series in $\Q[N_{0}(\X)]_{L}$, which we denote by $f_{L}$ and call the \emph{Laurent expansion} of $f$ with respect to $L$.

  In particular, in Section \ref{sec:preliminaries} we define a linear function $\deg \colon N_{0}(\X) \to \Z$ such that for any effective class $c \in N_{0}(\X)$ we have $\deg(c) > 0$.
  The generating functions $DT(\X)_{\beta}$ and $PT(\X)_{\beta}$ are both Laurent with respect to $\deg$.
\end{rmk}

Our first theorem is the orbifold analogue of the rationality statement for stable pairs theory of varieties proved in \cite{todajams,MR2813335}.
\newtheorem*{thm:PTIsRational}{Theorem \ref{thm:PTIsRational}}
\begin{thm:PTIsRational}
  Let $\X$ be a CY3 orbifold, and let $\beta \in N_{1}(\X)$.
  Then $PT(\X)_\beta$ is the Laurent expansion with respect to $\deg$ of a rational function $f_\beta \in \Q(N_{0}(\X))$.

  Moreover, we may write $f_{\beta}(q) = g/h$ with $g,h \in \Z[N_{0}(\X)]$ in such a way that $h$ is of the form
  \[
    h = (1-q^{2\beta \cdot A})^{n}
  \]
  for some ample divisor $A$ on $X$ and some positive integer $n$.
\end{thm:PTIsRational}

\subsubsection{Symmetry of $PT(\X)$}
The derived dualising functor $\D(-) \coloneq \lRHom(-,\hO_{\X})[2]$ induces an involution on $N_{\le 1}(\X)$ which preserves $N_{0}(\X)$, and so induces an involution on $N_{1}(\X)$.
Note that the splitting $N_{\le 1}(\X) = N_{1}(\X) \oplus N_{0}(\X)$ cannot always be chosen compatibly with this duality, so that in general $\D(\beta,c) \not= (\D(\beta),\D(c))$.

There is an induced involution on $\Q(N_{\le 1}(\X))$, which we also denote by $\D$.
\newtheorem*{thm:dualityResultForPT}{Proposition \ref{thm:dualityResultForPT}}
\begin{thm:dualityResultForPT}
  We have an equality of rational functions
  \[
    \D(z^{\beta}f_{\beta}(q)) = z^{\D(\beta)}f_{\D(\beta)}(q).
  \]
  Equivalently, summing the above over all $\beta \in N_1(\X)$, the function
  \[
    PT(\X) = \sum_{\beta \in N_{1}(\X)} \sum_{c \in N_{0}(\X)} PT(\X)_{(\beta,c)} z^{\beta}q^{c}
  \]
  is invariant under $\D$, when considered as an element of $\Q(N_{0}(\X))[\![N_{1}^{\eff}(\X)]\!]$.
\end{thm:dualityResultForPT}
This generalises the $q \leftrightarrow q^{-1}$ symmetry of the non-orbifold PT generating function.

\subsubsection{The McKay correspondence}\label{subsubsec:McKay}
By the McKay correspondence \cite{bkr,MR2407221}, the coarse space $X$ has a distinguished crepant resolution $f \colon Y \to X$ given \'{e}tale-locally on $X$ by Nakamura's $G$-Hilbert scheme.
Moreover, $Y$ and $\X$ are derived equivalent.
Concretely, $Y = \Quot(\hO_{\X}, [\hO_{x}])$ where $x \in \X$ is a non-stacky point and $\hO_{x}$ is the skyscraper sheaf at $x$.
The universal quotient sheaf on $Y \times \X$ is the kernel of the Fourier--Mukai equivalence $\Phi\colon D(Y) \to D(\X)$.

From now on we impose the additional restriction that $\X$ be \emph{hard Lefschetz}.
This means that the fibres of $f$ are at most 1-dimensional, and it restricts the possible stabiliser groups of stacky points; see \cite[Lem.~24]{MR2551767}.

Note that the functor $\Phi$ identifies the numerical groups $\Phi \colon N(Y) \stackrel{\sim}{\to} N(\X)$, but that it \emph{does not} preserve the filtration by dimension.
This discrepancy induces a series of new subgroups of $N_{\le 1}(\X)$ and $N_{\le 1}(Y)$, which we now describe.

Let $f_{*} \colon N(Y) \to N(X)$ denote the pushforward of numerical classes, and let
\[
  N_{\exc}(Y) = (f_{*})^{-1}\bigl(N_{0}(X)\bigr) \cap N_{\le 1}(Y) \subset N_{\le 1}(Y)
\]
denote the \emph{exceptional classes}, consisting of curve classes supported on the fibres of $f$ and point classes.
The natural inclusion $\Z = N_{0}(Y) \into N_{\exc}(Y)$ is canonically split by the holomorphic Euler characteristic $N_{\exc}(Y) = N_0(Y) \oplus N_{1,\exc}(Y)$, where $N_{1,\exc}(Y)$ denotes the exceptional \emph{curve} classes.
We write $N_{\nexc}(Y) = N_{\le 1}(Y)/N_{\exc}(Y)$ for the non-exceptional classes.

The McKay equivalence induces an inclusion $\Phi \colon N_{\le 1}(Y) \to N_{\le 1}(\X)$ that identifies $\Phi(N_{\exc}(Y)) = N_{0}(\X)$.
Thus the splitting (\ref{eq:SplittingX}) induces a splitting
\begin{equation}\label{eq:SplittingY}
  N_{\leq 1}(Y) = N_{\nexc}(Y) \oplus N_{\exc}(Y).
\end{equation}

The group of \emph{multi-regular classes} is defined as $N_{\mr}(\X) = \Phi(N_{\leq 1}(Y)) \subset N_{\le 1}(\X)$.
The following diagram summarises the relations between these subgroups.
\begin{equation}\label{tikz:NumericalDiagram}
  \begin{tikzcd}
    N_0(Y) \ar[hook]{r} & N_0(Y) \oplus N_{1,\exc}(Y) \ar[equal]{d} \ar[hook]{r} & \ar[equal]{d} N_{\leq 1}(Y) & \\
    & N_0(\X) \ar[hook]{r} & N_{\mr}(\X) \ar[hook]{r} & N_{\leq 1}(\X)
  \end{tikzcd}
\end{equation}
We write $N_{1,\mr}(\X) = \Phi(N_{\nexc}(Y))$.
By \eqref{eq:SplittingY} we obtain an induced splitting 
\begin{equation}\label{eq:SplittingXmr}
  N_{\mr}(\X) = N_{1,\mr}(\X) \oplus N_0(\X).
\end{equation}
We refer to elements in $N_{1,\mr}(\X)$ as \emph{multi-regular curve classes}.

With this notation in place, our second theorem is the following.
\theoremstyle{plain}
\newtheorem*{thm:DTPT}{Theorem \ref{thm:DTPT}}

\begin{thm:DTPT}[Orbifold DT/PT correspondence]
  Let $\X$ be a CY3 orbifold satisfying the hard Lefschetz condition, and let $\beta \in N_{1,\mr}(\X)$.
  Then the equality
  \begin{equation}
    PT(\X)_\beta = \dfrac{DT(\X)_\beta}{DT(\X)_0}
  \end{equation}
  of generating series holds in the ring $\Z[N_{0}(\X)]_{\deg}$.
\end{thm:DTPT}

\subsubsection{The crepant resolution conjecture}
We now define the generating series for DT invariants appearing in the crepant resolution conjecture (CRC).

Note that we tacitly use the identification of numerical groups $N_{\exc}(Y) = N_0(\X)$ and $N_{\nexc}(Y) = N_{1,\mr}(\X)$ induced by $\Phi$.
The exceptional generating series is
\begin{equation}\label{eq:DT_Exceptional}
  DT_{\exc}(Y) \coloneq \sum_{c \in N_0(\X)} DT(Y)_{(0,c)} q^{c}.
\end{equation}

In \cite{MR3449219}, Bryan and Steinberg show that for any $\beta \in N_{\nexc}(Y)$, we have
\begin{equation}\label{eq:BS_Comparison_Theorem}
  BS(Y/X)_{\beta} = \dfrac{DT(Y)_{\beta}}{DT_{\exc}(Y)}
\end{equation}
in $\Z[N_{0}(\X)]_{\deg}$ where $BS(Y/X)_{\beta}$ denotes the generating series of $f$-relative stable pair invariants of class $\beta$.
These are stable pair invariants on $Y$ relative to the crepant resolution $f \colon Y \to X$; see Section~\ref{sec:comparisonBS} for their definition.

Pick a general ample class $\omega \in N^{1}(Y)_{\R}$, a real number $\gamma > 0$, and define the linear function $L_{\gamma} \colon N_{0}(\X) \to \R$ by
\[
  L_{\gamma}(c) = \deg(c) + \gamma^{-1}\deg_{Y}(\ch_{2}(\Psi(c)) \cdot \omega),
\]
where $\Psi \colon D(\X) \to D(Y)$ denotes the inverse to the McKay equivalence $\Phi$.

We are now in a position to state our third and main theorem.
\theoremstyle{plain}
\begin{thm:MainTheorem}[Crepant resolution conjecture]
  Let $\X$ be a CY3 orbifold satisfying the hard Lefschetz condition and let $\beta \in N_{1,\mr}(\X)$.
  Then the equality
  \[
  PT(\X)_{\beta} = BS(Y/X)_{\beta}
  \]
  holds as rational functions.

  More precisely, if $f_{\beta} \in \Q(N_0(\X))$ is the rational function of Theorem~\ref{thm:PTIsRational}, then
  \begin{enumerate}
    \item the Laurent expansion of $f_{\beta}$ with respect to $\deg$ is $PT(\X)_{\beta}$, and
    \item the Laurent expansion of $f_{\beta}$ with respect to $L_{\gamma}$ is $BS(Y/X)_{\beta}$ if $0 < \gamma \ll 1$.
  \end{enumerate}
\end{thm:MainTheorem}

\begin{rmk}
The formulation of the CRC in \cite[Conj.~1]{MR2854183} is the claim
\begin{equation}\label{eq:CRCEquality}
  \dfrac{DT(\X)_{\beta}}{DT(\X)_0} = \dfrac{DT(Y)_{\beta}}{DT_{\exc}(Y)},
\end{equation}
where the meaning of the equality sign is left unspecified.
The example given in Appendix~\ref{sec:Appendix_ExampleCRC} shows that \eqref{eq:CRCEquality} is not in general true as an equality of generating functions, and so a rational function interpretation is necessary.
\end{rmk}

\begin{rmk}
A second conjecture \cite[Conj.~2]{MR2854183} states that
\[
  DT(\X)_{0} = \frac{DT(Y)^{+}_{\exc}DT(Y)^{-}_{\exc}}{DT_{0}(Y)},
\]
where
\[
  DT^{\pm}(Y)_{\exc} = \sum_{\beta \in N_{1,\exc}(Y)} DT(Y)_{\beta}z^{\pm \beta}.  
\]
This conjecture has been proved by the second-named author in \cite[Cor.~2.8]{MR3419955}; in contrast to equation~\eqref{eq:CRCEquality}, this holds as an equality of generating functions.
\end{rmk}

\subsection{Outline of proofs}
We work in $\A = \langle \hO_{\X}[1], \Coh_{\le 1}(\X) \rangle_{\text{ex}} \subset D(\X)$, a Noetherian abelian category introduced by Toda in \cite{todajams}.
For any torsion pair $(\T, \F)$ on $\Coh_{\le 1}(\X)$, we define a $(\T,\F)$-pair to be an object $E \in \A$ of rank $-1$ such that
\begin{equation}\label{eq:Intro_TFpair_HomVanishing}
  \Hom(T,E) = 0 = \Hom(E,F)
\end{equation}
for all objects $T \in \T$ and $F \in \F$.
This notion generalises the usual curve objects we are interested in counting.
For example, $(\T,\F)$-pairs are precisely ideal sheaves of curves (shifted by one) if we choose $\T_{DT} = 0$ and $\F_{DT} = \Coh_{\le 1}(\X)$, whereas taking $\T_{PT} = \Coh_{0}(\X)$ and $\F_{PT} = \Coh_{1}(\X)$ yields stable pairs (Lem.~\ref{prop:CohomCrit}).

We write $\uPair(\T,\F)$ for the moduli stack of $(\T,\F)$-stable pairs and show that under mild assumptions on $(\T,\F)$ it is an open substack of the algebraic stack of all complexes (Prop.~\ref{thm:TFOpenImpliesPOpen}).
We produce families of torsion pairs by considering a stability condition $\mu \colon N_{\le 1}^{\eff}(\X) \to S$, where $S$ is a totally ordered set, and a varying element $s \in S$.
For any such $s$, we define a torsion pair $(\T_{\mu, s}, \F_{\mu, s})$ on $\Coh_{\le 1}(\X)$ via
\begin{align*}
  \T_{\mu,s} &= \{T \in \Coh_{\leq 1}(\X) \mid T \onto G\neq 0 \Rightarrow \mu(G) \ge s\} \\
  \F_{\mu,s} &= \{F \in \Coh_{\leq 1}(\X) \mid 0 \neq E \into F \Rightarrow \mu(E) < s\}.
\end{align*}
We then define the DT invariant $DT^{\mu,s}(\X)_{(\beta,c)} \in \Z$ as the Behrend-weighted Euler characteristic of the stack $\uPair(\T_{\mu,s}, \F_{\mu,s})_{(\beta,c)}$ parametrising pairs of class $(\beta,c)$.

For $s \in S$, let $\uM^{\ss}_{\mu}(s)$ be the stack of $\mu$-semistable objects $F$ with $\mu(F) = s$ in $\Coh_{\le 1}(\X)$.
Assuming $\uM^{\ss}_{\mu}(s)$ satisfies a certain boundedness condition which we call being \emph{log-able} (Def.~\ref{dfn:DecompositionallyFinite}), one can define the generalised DT invariant $J^{\mu}_{(\beta,c)} \in \Q$ counting $\mu$-semistable objects of class $(\beta,c)$, see Section~\ref{sec:StabDelta_DTinvariants}.

These invariants enter into a wall-crossing formula phrased in terms of the Poisson torus $\Q[N(\X)]$.
This is a Poisson algebra with basis $t^{\alpha}$ for $\alpha \in N(\X)$, commutative multiplication\footnote{In fact, the $\star$ product will play no role in our arguments; see Remark~\ref{rmk:Integration_Poisson_vs_Lie}.} $t^{\alpha_{1}} \star t^{\alpha_{2}} = (-1)^{\chi(\alpha_{1},\alpha_{2})} t^{\alpha_{1} + \alpha_{2}}$, and Poisson bracket $\{t^{\alpha_{1}}, t^{\alpha_{2}}\} = (-1)^{\chi(\alpha_{1},\alpha_{2})}\chi(\alpha_{1}, \alpha_{2})t^{\alpha_{1} + \alpha_{2}}$.

We are interested in elements of class $\alpha = (r,\beta,c) \in \Z \oplus N_{\leq 1}(\X)$, where $r$ encodes the rank of a class, and so we use the convention $t^{\alpha} = t^{r[\hO_{\X}]}z^{\beta}q^{c}$ for these.
More precisely, we consider the set of classes
\[
  S_{\beta} = \{\alpha = (r, \beta', c) \mid r \in \{0,-1\}, 0 \le \beta' \le \beta\},
\]
and define $\Q[N_{1}^{\eff}(\X)]_{\le \beta}$ to be the set of all finite sums $\sum_{\alpha \in S_{\beta}} a_{\alpha}t^{\alpha}$ with $a_{\alpha} \in \Q$.
The product and Poisson bracket on $\Q[N(\X)]$ induce a product and Poisson bracket on $\Q[N_{1}^{\eff}(\X)]_{\le \beta}$, where we let $t^{\alpha_{1}}t^{\alpha_{2}} = 0$ if $\alpha_{1} + \alpha_{2} \not\in S_{\beta}$.
We then ``complete'' and define $\Q\{N_{1}^{\eff}(\X)\}_{\le \beta}$ to be the set of possibly infinite sums $\sum_{\alpha \in S_{\beta}} a_{\alpha}t^{\alpha}$, which then inherits a partially defined product and Poisson bracket.

\subsubsection{Rationality and self-duality of $PT(\X)$}
The first stability condition we consider, is Nironi's extension of slope stability to Deligne--Mumford stacks \cite{nironi_moduli_2008}.
It is given by a slope function $\nu \colon N_{\le 1}^{\eff}(\X) \to \R \cup \{+\infty\}$ and depends on the choice of an ample class on $X$ and an auxiliary \emph{generating} vector bundle (see Section~\ref{sec2_Modified_Hilbert}).

To get a varying notion of $(\T_{\nu,\delta}, \F_{\nu,\delta})$-pair, we collapse the Harder--Narasimhan filtration of $\nu$ into a torsion pair at a varying cut-off slope $\delta \in \R$.
For a fixed class $(\beta,c)$, the notion of $(\T_{\nu,\delta},\F_{\nu,\delta})$-pair is independent of $\delta$ for $\delta \gg 0$, and in the limit as $\delta \to \infty$, it agrees with the notion of PT pair.
Applying Joyce's wall-crossing formula now gives the following identity in $\Q\{N_{1}^{\eff}(\X)\}_{\le \beta}$:
\begin{equation}\label{eq:Intro_PT_J_DT}
\begin{split}
  PT(\X)_{\leq \beta}t^{-[\hO_{\X}]} &= DT^{\nu,\infty}(\X)_{\le \beta}t^{-[\hO_{\X}]} \\
    &= \prod_{\delta \in W_{\beta} \cap [\delta_{0}, \infty)} \exp\{J^{\nu}(\delta), -\} DT^{\nu,\delta_{0}}(\X)_{\leq \beta}t^{-[\hO_{\X}]},
\end{split}
\end{equation}
where $W_{\beta} = \frac{1}{l(\beta)!}\Z$ is the set of walls for $\delta$ where the notion of $(\T_{\nu,\delta},\F_{\nu,\delta})$-pair of class $\leq \beta$ may change, where the product is taken in increasing $\delta$, and where
\[
  J^{\nu}(\delta) = \sum_{\substack{(\beta', c) \in N_{\le 1}(\X)\\ \nu(\beta',c) = \delta}} J^{\nu}_{(\beta',c)} z^{\beta'}q^{c}.
\]
If $\X$ were a \emph{variety}, the Euler pairing would vanish on $N_{\le 1}(\X)$ and equation~\eqref{eq:Intro_PT_J_DT} would simplify to a product formula, as in \cite[Thm.~7.4]{MR2813335} and \cite{MR2683216},
\[
PT(\X) = \exp\Biggl(\sum_{\delta \ge \delta_{0}}\sum_{\substack{(\beta, c) \in N_{\le 1}(\X) \\ \nu(\beta,c) = \delta}} \chi(\beta,c) J^{\nu}_{(\beta,c)}z^{\beta}q^{c}\Biggr) DT^{\nu, \delta_{0}}(\X).
\]
In particular, the rationality statement could then be deduced from this expression.

Since $\X$ is an orbifold in our setting, we instead prove rationality of $PT(\X)_{\beta}$ in the following way.
Expanding the right hand side of equation~\eqref{eq:Intro_PT_J_DT} yields an infinite sum of terms of the form
\[
  C \cdot \{J^{\nu}_{(\beta_{r},c_{r})}z^{\beta_{r}}q^{c_{r}}, -\} \circ \ldots \circ \{J^{\nu}_{(\beta_{1},c_{1})}z^{\beta_{1}}q^{c_{1}},-\} DT^{\nu, \delta_{0}}(\X)_{(\beta',c')}z^{\beta'}q^{c'}t^{-[\hO_{\X}]},
\]
where $C$ is a constant of the form $\prod (n_{k}!)^{-1}$ arising from expanding the exponential.

We then group the terms with the same curve classes $\beta_{i}$, the same inequalities between the slopes $\nu(\beta_{i},c_{i})$, and the same values for $c_{i}\!\pmod{\beta_{i} \cdot A}$.
Twisting by the ample line bundle $A$ induces an equality $J^{\nu}_{(\beta,c)} = J^{\nu}_{(\beta, c + \beta \cdot A)}$, and so we may define $J^{\nu}_{(\beta,[c])} \coloneq J^{\nu}_{(\beta,c)}$ for any $[c] \in N_{0}(\X)/\beta \cdot A$.

The sum of the terms in such a \emph{group} has the form
\[
  \sum_{k_{1}, \ldots, k_{r}} C \cdot P(k_{1}, \ldots, k_{r}) \prod_{i=1}^{r} J^{\nu}_{\beta_{i}, [c_{i}]} DT^{\nu, \delta_{0}}(\X)_{\beta',c'} z^{\beta}q^{c + \sum_{i=1}^{r} k_{i} \beta_{i} \cdot A}t^{-[\hO_{\X}]},
\]
where the sum is over $k_{i} \in \Z_{\ge 0}$ satisfying a prescribed set of relations $k_{i} < k_{i+1}$ or $k_{i} = k_{i+1}$ for each $i$.
The constant $C$ arises from the exponential as before, and the term $P$ is the coefficient arising from the formula for the Poisson bracket.

Crucially, since the Poisson bracket $\{t^{\alpha_{1}}, t^{\alpha_{2}}\} = (-1)^{\chi(\alpha_{1}, \alpha_{2})}\chi(\alpha_{1}, \alpha_{2})t^{\alpha_{1} + \alpha_{2}}$ is bilinear up to sign in the exponents, the function $P$ is a \emph{quasi-polynomial}.
It follows formally that the above sum is a rational function of the shape we claim; see Section~\ref{sec:QuasiPolynomials}.
Through various boundedness results proved in Section~\ref{sec:rationalityOfStablePairs}, we show that there are only finitely many such groups.
Thus $PT(\X)_{\beta}$ is a sum of these finitely many rational functions and hence, in particular, is rational.

Examining this rational function further, we show that the degree of
\[
DT^{\nu,\infty}(\X)_{\beta} - DT^{\nu, \delta_{0}}(\X)_{\beta}
\]
tends to $-\infty$ as $\delta_{0} \to -\infty$, and so we find $DT^{\nu, -\infty}(\X)_{\beta} = DT^{\nu, \infty}(\X)_{\beta}$ as rational functions.
This leads to the symmetry $\D (PT(\X)) = PT(\X)$ as rational functions.

\subsubsection{Comparing $PT(\X)$ and $BS(Y/X)$}
\label{Intro_PT_vs_BS}
While it will not appear in the main text, it is helpful to consider a stability condition $\zeta_{1} \colon N_{1}^{\eff}(\X) \to \R \cup \{+\infty\}$ defined as follows.
Let $A$ and $\omega$ be ample classes on $X$ and $Y$ respectively, and define
\[
  \zeta_{1}(F) = - \frac{\deg_{Y}(c_{1}(\Psi(F)) \cdot A \cdot \omega)}{\deg(F \cdot A)} \in \Q,
\]
for any $F \in \Coh_{\le 1}(\X)$.
We set $\zeta_{1}(F) = \infty$ if $F \in \Coh_{0}(\X)$.
Fixing a class $(\beta,c) \in N_{1}(\X)$, the notion of $(\T_{\zeta_{1}, \gamma}, \F_{\zeta_{1},\gamma})$-pair of class $(\beta,c)$ reduces to the notion of PT pair as $\gamma \to \infty$.
In the limit $\gamma \to 0^{+}$, assuming $\beta$ is multi-regular, we show in Section~\ref{sec:comparisonBS} that a $(\T_{\zeta_{1},\gamma}, \F_{\zeta_{1},\gamma})$-pair is identified with a BS-pair by the McKay equivalence.

Fix a curve class $\beta \in N_{1}(\X)$.
Our goal is then to show that for $0 < \gamma \ll 1$, the series $DT^{\zeta_{1},\gamma}(\X)_{\beta}$ is rational and equal to $DT^{\zeta_{1}, \infty}(\X)_{\beta}$ as rational functions.

For $0 \le \beta' \le \beta$, there are only finitely many possible values for $\zeta_{1}(\beta',c)$, and so the set of walls $V_{\beta}$ between $0$ and $\infty$ is finite.
However, the stack $\uM_{\zeta_{1}}^{\ss}(\gamma)$ is not log-able, because in general the stacks $\uM_{\zeta_{1}}^{\ss}(\gamma)_{(\beta,c)}$ are not of finite type.
In conclusion, the wall-crossing formula is not directly applicable.

We must therefore refine the stability condition $\zeta_{1}$ and introduce the stability condition $\zeta = (\zeta_{1}, \nu) \colon N_{1}^{\eff}(\X) \to \R^{2} \cup \{+\infty\}$, where $\R^{2}$ is given the lexicographical ordering: $(a,b) \ge (a',b')$ if and only if $a > a' \text{ or } (a = a', b \ge b')$.

The series\footnote{For now, we omit the orbifold $\X$ from the notation, so $DT^{\zeta,(\gamma,\eta)}(\X)_{\le \beta} = DT^{\zeta,(\gamma,\eta)}_{\le \beta}$.} $DT^{\zeta, (\gamma, \eta)}_{\le \beta}$ has a wall-crossing behaviour described as follows.
Away from the set of $\gamma$-walls $V_{\beta}$, the corresponding notion of pair is independent of $\eta \in \R$ so $(\T_{\zeta,(\gamma,\eta)},\F_{\zeta, (\gamma,\eta)}) = (\T_{\zeta_{1},\gamma}, \F_{\zeta_{1}, \gamma})$.
As a consequence, the function $DT^{\zeta, (\gamma,\eta)}_{\le \beta}$ is locally constant as a function of $(\gamma,\eta)$, and it is independent of $\eta$ when $\gamma \not\in V_{\beta}$.

Fixing a wall $\gamma \in V_{\beta}$, and varying $\eta$, the series $DT^{{\zeta, (\gamma, \eta)}}_{\le \beta}$ has the same walls as the series $DT^{\nu, \delta}_{\le \beta}$.
Moreover, taking the limit as $\eta \to \pm\infty$ makes sense, and we find
\[
DT_{\le \beta}^{\zeta_{\gamma,\pm \infty}} = DT_{\le \beta}^{\zeta_{\gamma \pm \epsilon, \eta}},
\]
for $0 < \epsilon \ll 1$ and for any $\eta \in \R$, allowing us to \emph{slide off the $\gamma$-wall}.
By an argument similar to the one showing $DT^{\nu, \infty}_{\beta} = DT^{\nu, -\infty}_{\beta}$, we deduce that $DT_{\beta}^{\zeta, (\gamma,\infty)} = DT_{\beta}^{\zeta, (\gamma,-\infty)}$ as rational functions, thus completing the re-expansion at the wall $\gamma$.

Labeling the $\gamma$-walls $V_{\beta} = \{\gamma_{1}, \ldots, \gamma_{r}\}$ with $\gamma_{i} < \gamma_{i+1}$, we prove that $PT(\X)_{\le \beta}$ and $BS(Y/X)_{\le \beta}$ are expansions of the same rational function in the following way:
\begin{align*}
  PT(\X)_{\le \beta} &= DT^{\zeta, (\gamma_{r} +\epsilon, 0)}_{\le \beta} = DT_{\le \beta}^{\zeta, (\gamma_{r}, \infty)} \leadsto DT_{\le \beta}^{\zeta, (\gamma_{r},-\infty)} = DT_{\le \beta}^{\zeta, (\gamma_{r-1},\infty)} \\
  & \leadsto DT_{\le \beta}^{\zeta, (\gamma_{r-1},-\infty)} = DT^{\zeta, (\gamma_{r-2},\infty)} \leadsto \ldots \\
  & \leadsto DT_{\le \beta}^{\zeta, (\gamma_{1},-\infty)} = DT^{\zeta, (\epsilon, 0)}_{\le \beta} = BS(Y/X)_{\le \beta},
\end{align*}
where the $\leadsto$ indicate a re-expansion of a rational function; see Figure 1.
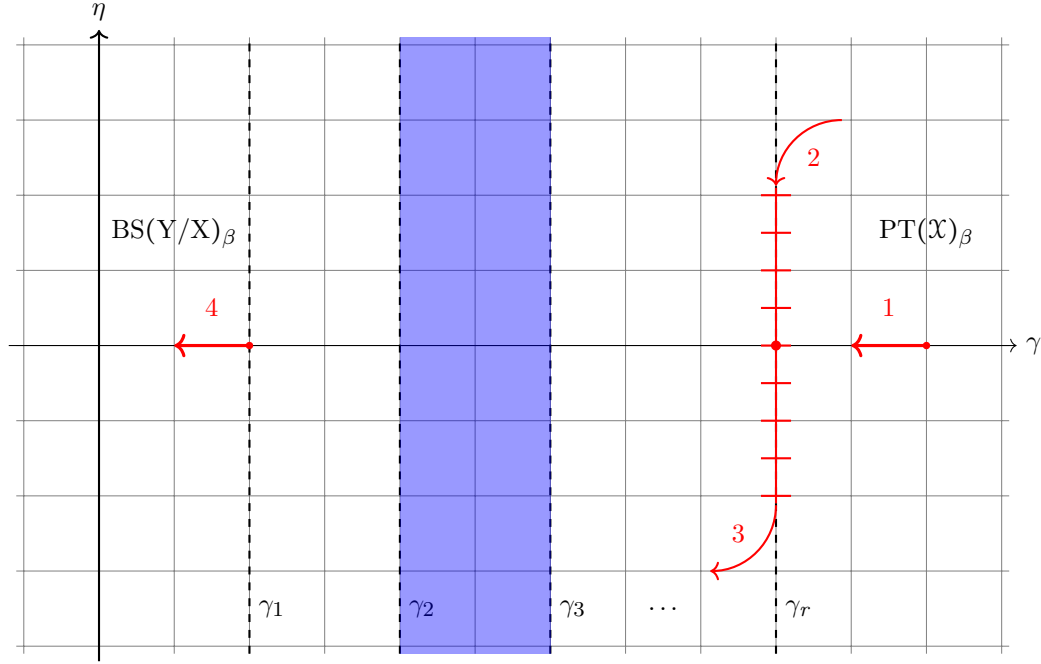
\begin{figure}[!ht]
\begin{center}
\begin{tikzpicture}
    \draw[very thin,color=gray] (-1.1,-4.1) grid (12.1,4.1);    
    \draw[->,name path=xaxis] (-1.2,0) -- (12.2,0) node[right] {$\gamma$};
    \draw[thick,->,name path=yaxis] (0,-4.2) -- (0,4.2) node[above] {$\eta$};
  
    \draw[dashed,thick,name path=line0,domain=-4.1:4.1] plot (2,\x) node[above right] {};
    \draw[dashed,thick,name path=line1,domain=-4.1:4.1] plot (4,\x) node[below right] {};
    \draw[dashed,thick,name path=line2,domain=-4.1:4.1] plot (6,\x) node[below right] {};
    \draw[dashed,thick,name path=line2,domain=-4.1:4.1] plot (9,\x) node[below right] {};
    \draw[->,very thick,name path=WCpath0,color=red] (2,0) -- (1,0) node {};
    \draw[->,very thick,name path=WCpath0,color=red] (11,0) -- (10,0) node {};
  
    \node[right] (a) at (2,-3.5) {$\gamma_1$};
    \node[right] (b) at (4,-3.5) {$\gamma_2$};
    \node[right] (c) at (6,-3.5) {$\gamma_3$};
    \node (d) at (7.5,-3.5) {$\ldots$};
    \node[right] (e) at (9,-3.5) {$\gamma_r$};
    \node (0up) at (9,2) {};
    \node (0upRight) at (10,3) {};
    \node (0down) at (9,-2) {};
    \node (0downLeft) at (8,-3) {};
    \node (PT) at (1,1.5) {$BS(Y/X)_{\beta}$};
    \node (BS) at (11,1.5) {$PT(\X)_{\beta}$};
    \node[red] at (1.5,0.5) {$4$};
    \node[red] at (9.5,2.5) {$2$};
    \node[red] at (8.5,-2.5) {$3$};
    \node[red] at (10.5,0.5) {$1$};
    
    \path[->,thick,red] (0down) edge[bend left=45] node [left] {} (0downLeft);
    \path[->,thick,red] (0upRight) edge[bend right=45] node [right] {} (0up);
    \draw[thick,red] (9,-2.1) to (9,2.1);
    \draw[thick,red] (8.8,2) to (9.2,2);
    \draw[thick,red] (8.8,1.5) to (9.2,1.5);
    \draw[thick,red] (8.8,1) to (9.2,1);
    \draw[thick,red] (8.8,0.5) to (9.2,0.5);
    \draw[thick,red] (8.8,0) to (9.2,0);
    \draw[thick,red] (8.8,-0.5) to (9.2,-0.5);
    \draw[thick,red] (8.8,-1) to (9.2,-1);
    \draw[thick,red] (8.8,-1.5) to (9.2,-1.5);
    \draw[thick,red] (8.8,-2) to (9.2,-2);

    \fill[blue,fill opacity=0.4] (4,-4.1) -- (6,-4.1) -- (6,4.1) -- (4,4.1) -- cycle;
  
    \filldraw[thick,red] (2,0) circle (1pt);
    \filldraw[thick,red] (9,0) circle (1.5pt);
    \filldraw[thick,red] (11,0) circle (1pt);
  
    \end{tikzpicture}
\end{center}
\caption{A schematic of the $(\gamma,\eta)$-wall-crossing.
The notion of $\gamma$-pair is constant between two consecutive $\gamma$-walls, for example in the blue region of $(\gamma,\eta)$ with $\gamma_2 < \gamma < \gamma_3$.}
\end{figure}

\subsection{Variants}
We discuss three variants of the crepant resolution conjecture.
\subsubsection{Euler characteristics}
The methods of this paper also prove exactly the same results for Euler--DT invariants, i.e., those defined by taking the topological Euler characteristic instead of the one weighted by the Behrend function $e_B$.

\subsubsection{The quasi-projective case}
\label{sec1_QuasiProjective_Case}
In order to connect these results with various computations for toric CY3s, let us explain how our arguments generalise to the setting of quasi-projective CY3s.
Let $\X$ be a 3-dimensional orbifold with quasi-projective coarse moduli space $X$ such that $\omega_{\X} \cong \hO_{\X}$ and $\Pic(\X)$ is finitely generated.
As before, let $Y$ denote the crepant resolution of $X$ of the McKay correspondence.

The notion of DT invariants and PT invariants of $\X$ can be defined as follows.
Choose a smooth compactification $\overline{\X}$ of $\X$ and consider those DT/PT objects on $\overline{\X}$ which restrict to $\hO_{\overline{\X}}[1]$ on $\overline{\X} \setminus \X$ (more intrinsic constructions are possible, but require a longer explanation).
As our arguments are motivic in nature and do not rely on properness of the moduli stacks involved, they transpose well to this setting.

One can argue as follows.
Define the category $\A = \langle \hO_{\overline{\X}}[1], \Coh_{\le 1}(\overline \X) \rangle_{\text{ex}}$ as before.
We say that an object in $\A$ of rank $-1$ is \emph{supported on $\X$} if its restriction to $\overline{\X} \setminus \X$ is $\hO_{\overline{\X}}[1]$.
Then the objects of rank $0$ and $-1$ in $\A$ which are supported on $\X$ define open substacks of $\uA \subset \Mum_{\overline{\X}}$, analogous to the stacks that appear in our proof in the projective case.

Note that if $E, F \in \cat A$ are such that $E$ has rank 0, $F$ has rank $0$ or $-1$, and both are supported on $\X$, then $\chi(E,F) = -\chi(F,E)$ and any extension of $E$ by $F$ or of $F$ by $E$ is supported on $\X$.
This is enough to ensure that the wall-crossing arguments of our proof go through in the Euler characteristic case.
Running our proof with minor modifications thus gives the Euler characteristic, quasi-projective versions of our main results: ``rationality of $PT(\X)$'' (Theorem~\ref{thm:PTIsRational}) and equality as rational functions ``$PT(\X) = BS(Y/X)$ if $\X$ is hard Lefschetz'' (Theorem~\ref{thm:MainTheorem}).

In the actual DT (Behrend-weighted) case, we do not know the Behrend function identities on our open moduli substacks.
Toda's argument in the projective case \cite[Thm.~2.6]{2016arXiv160107519T} relies on the existence of $(-1)$-shifted symplectic structures on the relevant moduli stacks.
While it seems reasonable to expect that $(-1)$-shifted symplectic structures also exist in the more general quasi-projective setting, cf. \cite{2014arXiv1403.2403B,Preygel}, a full exploration of this issue is beyond the scope of this paper.

\subsubsection{Beyond multi-regular curve classes}
The multi-regularity assumption on the curve class $\beta$ is used in two places: To prove the equality $PT(\X)_{\beta} = DT(\X)_{\beta}/DT_0(\X)$ and the equality $DT^{\zeta, (\epsilon,0)}(\X)_{\beta} = BS(Y/X)_{\beta}$.
In particular, the re-expansion argument does not require $\beta$ to be multi-regular, and we still have the equality of rational functions $PT(\X)_{\beta} = DT^{\zeta, (\epsilon, 0)}(\X)_{\beta}$ in the non-multi-regular case.

In the non-multi-regular case, both of the missing steps pose potentially interesting problems: that of determining the relation between $PT(\X)_{\beta}$ and $DT(\X)_{\beta}$ for general $\beta$, and that of relating $DT^{\zeta, (\epsilon,0)}(\X)_{\beta}$ to ``curve-counting-like'' invariants on $Y$.

\subsection{Previous work}
Our techniques, using wall-crossing to relate counting invariants via the motivic Hall algebra, are to a large extent refinements of those pioneered by Joyce, Bridgeland, and Toda, see~\cite{MR2354988,MR2813335,todajams,MR3021446}.

Our approach makes essential use of the work of Bryan and Steinberg \cite{MR3449219}, who introduced the notion of $f$-stable pairs associated to a crepant resolution $f$.
Their counting invariants play the role of the PT generating function on the $Y$-side and, in particular, give a geometric interpretation of the fraction $DT(Y)/DT(Y)_{\exc}$.

In \cite{MR3595887}, Ross proved the quasi-projective version of the crepant resolution conjecture for all toric CY3 orbifolds with $A_n$-singularities, by analysing the orbifold topological vertex of \cite{MR2854183} in that case.
As our results and the example in Appendix~\ref{sec:Appendix_ExampleCRC} indicate, the generating function equalities in \cite{MR3595887}, in particular \cite[Thm.~2.2]{MR3595887}, must be interpreted as equalities of rational functions. 

\subsection{Acknowledgements} We thank Kai Behrend, Tom Bridgeland, Jim Bryan, John Christian Ottem, Dustin Ross, Balázs Szendrői, Richard Thomas (who insisted that wall-crossing would work when one of us claimed it didn't), and Yukinobu Toda for useful discussions related to this problem. Special thanks to Arend Bayer; in particular, S.B. is grateful for his continued guidance over the years. 

S.B. was supported by the ERC Starting Grant no.~337039 \emph{WallXBirGeom}.
Parts of this project were carried out while J.V.R. was staying at the Mathematical Sciences Research Institute in Berkeley in connection with the program ``Enumerative Geometry Beyond Numbers'', funded by NSF Grant DMS-1440140.

\subsection{Conventions}
We work over $\C$.
All rings, schemes, and stacks will be assumed to be \emph{locally of finite type} over $\C$, unless specified otherwise.
All categories and functors will be $\C$-linear.
If $M$ is a scheme (or stack) we write $D(M)$ for the \emph{bounded coherent} derived category of $M$.
We write $\Coh_i(M)$ (resp. $\Coh_{\leq i}(M)$) for the full subcategory of coherent sheaves on $M$ pure of dimension $i$ (resp. of dimension at most $i$).
When denoting counting invariants, we omit the orbifold or variety from the notation if no confusion can occur, e.g. $DT^{\zeta,(\gamma,\eta)}_{\leq \beta} \coloneq DT^{\zeta,(\gamma,\eta)}(\X)_{\leq \beta}$.


\section{Preliminaries}
\label{sec:preliminaries}

\subsection{Numerical Grothendieck group}
\label{subsec:Numerical_Grothendieck_Group}
Let $M$ be a projective variety, or more generally a Deligne--Mumford stack with projective coarse moduli space.
We write $\Coh(M)$ for its category of coherent sheaves, and $D(M) = D^\text{b}(\Coh(M))$ for its bounded coherent derived category.
This category contains $\Perf(M)$, the subcategory of {perfect} complexes, which by definition are those locally isomorphic to a bounded complex of locally free sheaves.
When $M$ is smooth and hence satisfies the resolution property \cite{MR2108211},
$\Perf(M) = D(M)$.

For $E \in D(M)$, and $P \in \Perf(M)$, the \emph{Euler pairing}
\begin{align*}
  \chi(P,E) \coloneq \sum_i (-1)^i \dim \Hom_M(P,E[i])
\end{align*}
is well defined.
We call $E$ \emph{numerically trivial} if $\chi(P,E) = 0$ for all $P \in \Perf(M)$.

We write $K(M) = K(D(M)) = K(\Coh(M))$ for the \emph{Grothendieck group} of $M$.
We write $N(M)$ for the \emph{numerical} Grothendieck group, which is the quotient of $K(M)$ by the subgroup generated by all numerically trivial complexes.
For $E \in D(M)$, we write $[E] \in N(M)$ for its numerical class.

Assume $M$ is irreducible, with generically trivial stabiliser groups.
The rank of a sheaf $F \in \Coh(M)$ equals $(-1)^{\dim M}\chi(\hO_{p},F)$, where $p \in M$ is a non-stacky point, and the rank defines a homomorphism $\rk\colon N(M) \to \Z$.

\subsubsection{The dimensional filtration}\label{subsubfilt}
The group $N(M)$ has a filtration by dimension of support.
We write $\Coh_{\leq d}(M) \subset \Coh(M)$ for the subcategory of sheaves supported in dimension at most $d$.
We define $N_{\leq d}(M) \subset N(M)$ as the subgroup generated by classes of sheaves $F \in \Coh_{\leq d}(M)$.
We write $N_d(M) \coloneq N_{\leq d}(M)/N_{\leq d-1}(M)$ for the associated graded pieces.
The groups $N_{\leq d}(M)$ and $N_d(M)$ are free abelian of finite rank.
Note that in general $\rk N_0(M) \geq 2$ for a DM stack; see Footnote~\ref{foot:N0_DM_higher_rank}.

\subsection{Geometric setup}\label{suborbifolds}
Let now $\X$ be a \emph{CY3 orbifold}, which we take to mean that $\X$ is a smooth, irreducible, 3-dimensional Deligne--Mumford (DM) stack such that
\begin{itemize}
\item the stabilizer groups of $\X$ are generically trivial,
\item we have $\omega_\X \cong \hO_{\X}$,
\item $H^1(\X,\hO_\X) = 0$, and
\item the coarse moduli space $X$ of $\X$ is projective.
\end{itemize}
Note that the coarse moduli space $X$ and the associated canonical morphism $g \colon \X \to X$ exist by the Keel--Mori theorem.
Our assumptions imply that $X$ is Gorenstein, that $\omega_{X} = \hO_{X}$, and that $X$ has at worst quotient singularities.

\subsubsection{Splitting}
We choose a splitting of the inclusion $N_{0}(\X) \into N_{\le 1}(\X)$, so
\begin{align}\label{thesacredsplitting1}
  N_{\leq 1}(\X) \cong N_1(\X) \oplus N_0(\X).
\end{align}
Thus we will denote classes in $N_{\le 1}(\X)$ by $(\beta, c)$ with $\beta \in N_{1}(\X)$ and $c \in N_{0}(\X)$.
For $F \in \Coh_{\leq 1}(\X)$, we define $\beta_{F}$ and $c_{F}$ by $[F] = (\beta_F,c_F)$.

In contrast with the case of CY3 varieties, where $\ch_{3}$ defines a canonical splitting of $N_{0}(\X) \into N_{\le 1}(\X)$, there need not exist a canonical splitting in the orbifold case.
This seems a necessary evil.
In particular, we emphasize that the splitting cannot always be chosen compatibly with this duality, so that in general $\D(\beta,c) \not= (\D(\beta),\D(c))$ where $\D(-) = \lRHom(-,\hO_{\X})[2]$ is the derived dualising functor (shifted by two).

\subsubsection{The modified Hilbert polynomial and degree}
\label{sec2_Modified_Hilbert}
We choose a vector bundle $V$ on $\X$ which is generating in the sense of \cite{MR2007396}.
This means that every coherent sheaf on $\X$ is locally a quotient of $V^{\oplus n}$ for some $n$.
Replacing $V$ with $V \oplus V^\vee$, we may (and will) assume that $V \cong V^\vee$.
We fix an ample line bundle $A$ on $X$, and write $A$ also for $g^{*} A$.

For $F \in \Coh(\X)$, we let $F(k) = F \otimes A^{\otimes k}$.
Following Nironi \cite{nironi_moduli_2008}, we define the \emph{modified Hilbert polynomial} $p_{F}$ and the integers $l(F), \deg(F)$ by
\begin{align}\label{eq:ModifiedHilbertPolynomial}
  p_F(k) \coloneq \chi(\X, V^\vee \otimes F(k)) = l(F)k + \deg(F)
\end{align}
The polynomial $p_{F}$ depends only on the numerical class of $F$, and so $p_\alpha(k)$, $l(\alpha)$, $\deg(\alpha)$ is well defined for any class $\alpha \in N_{\le 1}(\X)$.
Moreover $l(N_{0}(\X)) = 0$, so the number $l(\beta)$ is well defined for curve classes $\beta \in N_{1}(\X)$.

Note that in general we do not have $\deg(\beta,c) = \deg(c)$.
Indeed, given a generating vector bundle $V$ there may be no way of choosing the splitting of $N_{0}(\X) \into N_{\le 1}(\X)$ in such a way that $\deg$ is compatible with the splitting.

\subsubsection{The effective cone}\label{subsubsec:EffectiveCone}
We say that $\beta \in N_1(\X)$ is \emph{effective} if $\beta = \beta_F$ for some sheaf $F \in \Coh_{\leq 1}(\X)$, and we write $\beta' \le \beta$ if $\beta - \beta'$ is effective.
We let $N^{\eff}(\X) \subset N(\X)$ be the cone spanned by effective classes, and similarly for $N_{\le d}^{\eff}(\X)$ and $N_{d}^{\eff}(\X)$.

\begin{lem}
  \label{thm:a1DeterminesBetaUpToFinitelyManyChoices}
  For every $n \in \Z$, the set $\{\beta \in N_{1}^{\eff}(\X) \mid l(\beta) = n\}$ is finite.
\end{lem}
\begin{proof}
  Let $\beta \in N_{1}^{\eff}(\X)$ with $l(\beta) = n$, so $\beta = \beta_{F}$ for some $F \in \Coh_{\le 1}(\X)$.
  After twisting by $A$, we may assume that $\deg(F) \in \{0, 1, \ldots, n-1\}$.
  Then \cite[Thm.~4.27]{nironi_moduli_2008} shows that $F$ lies in a bounded set, leaving only finitely many possibilities for $\beta_{F}$.
\end{proof}

\begin{cor}\label{lem:Effective_Cone_Convex}
  If $\beta \in N_{1}^{\eff}(\X)$, there are finitely many $\beta' \in N_{1}^{\eff}(\X)$ with $\beta' \leq \beta$.
\end{cor}

\subsection{Stability conditions}\label{sec:NironiBackground}
We recall the particular notion of stability condition we require, see \eg~\cite{MR1480783,MR2354988}. 
\begin{defn}\label{defn:Stability}
  A \emph{stability condition} on $\Coh_{\leq 1}(\X)$ consists of a \emph{slope function} $\mu \colon N_{\leq 1}(\X) \to S$ where $(S,\leq)$ is a totally ordered set, such that
  \begin{enumerate}
  \item the slope $\mu$ satisfies the \emph{see-saw property}, i.e., given an exact sequence $0 \to A \to B \to C \to 0$ 	in $\Coh_{\leq 1}(\X)$ we have either
    \begin{equation*}
      \mu(A) < \mu(B) < \mu(C) \quad \text{or} \quad
      \mu(A) = \mu(B) = \mu(C) \quad \text{or} \quad
      \mu(A) > \mu(B) > \mu(C);
    \end{equation*}
  \item the category $\Coh_{\leq 1}(\X)$ has the \emph{Harder--Narasimhan property} with respect to $\mu$, i.e., any sheaf $F \in \Coh_{\leq 1}(\X)$ admits a filtration in $\Coh_{\leq 1}(\X)$,
    \begin{equation*}
      0 = F_0 \subset F_1 \subset \ldots \subset F_{n-1} \subset F_n = F,
    \end{equation*}
    called the Harder--Narasimhan (HN) filtration, such that each factor $Q_i = F_i/F_{i-1}$ is semistable of descending slope $\mu(Q_1) > \mu(Q_2) > \ldots > \mu(Q_n)$.
    The semistable factors $Q_i$ are called the \emph{HN factors} of $F$.
  \end{enumerate}
  A sheaf $F \in \Coh_{\leq 1}(\X)$ is \emph{stable} if for all non-trivial proper subsheaves $0 \neq E \subset F$
  \begin{equation*}
    \mu(E) < \mu(F),
  \end{equation*}
  or, equivalently, $\mu(F) < \mu(F/E)$ or $\mu(E) < \mu(F/E)$ by the see-saw property.
  To obtain the notion of \emph{semistability}, replace each strict inequality $<$ by a weak one $\leq$.
\end{defn}

\begin{rmk}
  Let $\mu$ be a slope function on $\Coh_{\leq 1}(\X)$.
  Since the category $\Coh_{\le 1}(\X)$ is Noetherian, it has the Harder--Narasimhan property with respect to $\mu$ whenever $\Coh_{\leq 1}(\X)$ is $\mu$-Artinian, i.e., when any chain of subobjects $F_0 \supset F_1 \supset F_2 \supset \ldots$ in $\Coh_{\leq 1}(\X)$ such that $\mu(F_i) \geq \mu(F_{i-1})$ stabilizes; see \cite[Thm.~4.4]{MR2354988}.
\end{rmk}

\subsubsection{Nironi slope stability}
The usual notion of stability of sheaves on a variety has a useful generalisation to DM stacks.
The foundational results of this theory have been worked out by Nironi \cite{nironi_moduli_2008}.

Given $F \in \Coh_{\leq 1}(\X)$, the \emph{Nironi slope} is 
\begin{align}\label{eq:NironiSlope}
  \nu(F) \coloneq \frac{\deg(F)}{l(F)} \in \Q
\end{align}
if $F \not \in \Coh_{0}(\X)$, and it is $\nu(F) = \infty$ if $F \in \Coh_{0}(\X)$.

\begin{prop}
  \label{thm:NironiSlopeIsJoyce}
  \label{thm:NironiSlopeHasHNFiltration}
  The slope function $\nu$ defines a stability condition on $\Coh_{\leq 1}(\X)$.
\end{prop}
\begin{proof}
  For a pure 1-dimensional sheaf, the existence of Harder--Narasimhan filtrations follows from \cite[Thm.~3.22]{nironi_moduli_2008}.
  Otherwise, combine this result with the usual torsion filtration; see \cite[Cor.~3.7]{nironi_moduli_2008}.
  The see-saw property is easily verified.
\end{proof}

\subsubsection{Nironi moduli}
Let $\cCoh_\X$ denote the moduli stack of coherent sheaves on $\X$.
It is an algebraic stack that is locally of finite type by \cite[Cor.~2.27]{nironi_moduli_2008}.

We write $F_{+}$ (resp.~$F_{-}$) for the HN factor of $F$ with the biggest (resp.~smallest) slope and we write $\nu_{+}(F) = \nu(F_{+})$, $\nu_{-}(F) = \nu(F_{-})$.
Let $I \subset (-\infty,\infty]$ be an interval, and let
\begin{equation*}
  \M_\nu(I) \subset \cCoh_\X
\end{equation*}
denote the substack parametrising sheaves $F$ such that the $\nu$-slopes of all its HN factors lie in $I$, which is equivalent to $\nu_{-}(F),\nu_{+}(F) \in I$.
The moduli stack of $\nu$-semistable sheaves with $\nu(F) = s$, is the special case of $I = [s,s]$; in this case we write $\M^{\ss}_\nu(s)$ instead.
For $\beta \in N_{1}(\X)$, we write $\M_{\nu}(I, \beta)$ for the open substack of $\M_{\nu}(I)$ consisting of sheaves $F$ with $\beta_{F} = \beta$.
\begin{thrm}[Nironi]\label{thm:NironisBoundednessTheorem}
  Let $I \subseteq \R$ be an interval and $\beta \in N_{1}(\X)$.
  Then the substack $\M_{\nu}(I,\beta) \subset \cCoh_\X$ is open.
  If the interval is of finite length, the stack is of finite type.
  In particular, $\M^{\ss}_{\nu}(s, \beta)$ is of finite type for any $s \in \R$.
\end{thrm}
\begin{proof}
  These results follow by the Grothendieck lemma for stacks \cite[Lem.~4.13]{nironi_moduli_2008}, and applying the same proof as in \cite[Prop.~4.15]{nironi_moduli_2008} and \cite[Prop.~2.3.1]{MR2665168}.
\end{proof}

If $p \in \Q[x]$, let $\Quot_\X(F,p)$ denote the functor of quotients of $F \in\Coh(\X)$ with modified Hilbert polynomial $p_{F} = p$.
The following result is a combination of \cite[Thm.~4.20]{nironi_moduli_2008} and \cite[Thm.~6.1]{MR2007396}.
\begin{thrm}\label{thm:Quot_Scheme_Orbifold_Is_Projective}
  For $p \in \Q[x]$, the functor $\Quot_{\X}(F,p)$ is represented by a projective scheme, which we also denote by $\Quot_\X(F,p)$.
\end{thrm}

The following two lemmas will be used repeatedly in Section \ref{sec:rationalityOfStablePairs}.
\begin{lem}
  \label{thm:BoundOnNuMax}
  Let $F \in \Coh_{1}(\X)$ be a non-zero pure 1-dimensional sheaf.
  Then
  \[
  \begin{split}
    \nu_{+}(F) &\le \deg(F) - [l(F)-1]\nu_{-}(F) \\
    \nu_{-}(F) &\ge \deg(F) - [l(F)-1]\nu_{+}(F)
  \end{split}
  \]
\end{lem}
\begin{proof}
  Consider the exact sequence
  \[
    0 \to F_{+} \to F \to F' \to 0.
  \]
  Note that $l(F_{+}) \geq 1$ and $l(F) - l(F_{+}) \geq 0$.
  We deduce that
  \begin{align*}
    \deg(F) &= \nu(F)l(F) = \nu(F_{+})l(F_{+}) + \nu(F')l(F') \\
            &= \nu(F_{+})l(F_{+}) + \nu(F')[l(F) - l(F_{+})] \\
            &\ge \nu(F_{+})l(F_{+}) + \nu(F_{-})[l(F) - l(F_{+})] \\
            &\ge \nu_{+}(F) + \nu_{-}(F)[l(F) - 1].
  \end{align*}
  The second claim is proven similarly.
\end{proof}

As a consequence, there is the following boundedness result.
\begin{lem}
  \label{thm:BoundingDegreeAndNuMinMaxGivesFiniteType}
  Let $d, \delta \in \R$ and $\beta \in N_{1}(\X)$.
  The substacks of sheaves
  \begin{enumerate}
      \item $F \in \uM_{\nu}([\delta,\infty),\beta)$ with $\deg(F) \le d$, and
      \item $F \in \uM_{\nu}((-\infty,\delta])$ with $\deg(F) \ge d$
  \end{enumerate}
  are both of finite type.
\end{lem}
\begin{proof}
  Any pure 1-dimensional sheaf $F$ of class $\beta_{F} = \beta$ satisfying $\deg(F) \le d$ and $\nu_{-}(F) \ge \delta$, defines an element in $\M_{\nu}([\delta, d - (l(\beta)-1)\delta],\beta)$.
  The first claim now follows from Theorem~\ref{thm:NironisBoundednessTheorem}.
  The second claim is proven similarly.
\end{proof}

\subsection{The crepant resolution}
\label{subsubsec:CrepantResolutions}
Let $p \in \X$ be a non-stacky point, and let $\hO_{p}$ be the corresponding skyscraper sheaf.
We let $Y = \Hilb^1(\X) = \Quot_{\X}(\hO_{\X}, [\hO_{p}])$.
This space is \'{e}tale-locally on $X$ the moduli space of \emph{$G$-clusters} (\ie~Nakamura's $G$-Hilbert scheme \cite{MR1838978}), and comes with a map $f\colon Y \to X$.
By \cite{bkr,MR2407221}, $Y$ is a smooth projective CY3 variety, and $f$ is a crepant resolution, i.e., $f^*\omega_X = \omega_Y$.

The universal quotient sheaf $\hO_{\mathcal{Z}} \in D(Y \times \X)$ is the kernel of a Fourier--Mukai equivalence, which we refer to as the \emph{McKay correspondence}, and denote by
\begin{equation}\label{eq:McKay_Correspondence}
  \Phi\colon D(Y) \overset{\sim}{\longrightarrow} D(\X).
\end{equation}
Note that $\Phi(\os_Y) = \os_\X$, and that $g_{*} \circ \Phi = Rf_{*}$.
We denote $\Psi = \Phi^{-1}$.

\subsubsection{The hard Lefschetz condition}
\label{sec:hardLefschetz}
For where we deal with stable pair invariants on $\X$, \ie, rationality of $PT(\X)$ and the symmetry of $PT(\X)$, the variety $Y$ will not play any role, and the assumptions listed at the start of Section~\ref{suborbifolds} suffice.
When it comes to the orbifold DT/PT correspondence and the crepant resolution conjecture, we will impose the following extra condition on $\X$.
\begin{defn}
  Let $\X$ be a CY3 orbifold as defined in Section~\ref{suborbifolds}.
  We say that $\X$ is \emph{hard Lefschetz} if the fibres of the resolution $f \colon Y \to X$ have dimension $\le 1$.
\end{defn}
%
With this assumption, the map $f \colon Y \to X$ induces a t-structure on $D(Y)$, first introduced by Bridgeland in \cite{MR1893007}.
By definition, its heart $\Per(Y/X)$ is the category of \emph{perverse coherent sheaves}\footnote{Strictly speaking, there is an instance of this category ${^p}{\Per(Y/X)}$ for each $p \in \Z$. In this paper we deal with the $p=0$ version only, so we suppress this choice of \emph{perversity} throughout.},\label{perverseSheaves} consisting of those $E \in D(Y)$ such that:
\begin{itemize}
  \item $Rf_{*}(E) \in \Coh(X)$, and
  \item for any $C \in \Coh(Y)$ with $Rf_{*}C = 0$, we have $\Hom(C[1], E) = 0$.
\end{itemize}
This abelian category is a left tilt of $\Coh(Y)$ at a torsion pair \cite{MR2057015}.
Furthermore, it admits a description in terms of sheaves on $\X$.
\begin{prop}[{\cite[Thm 1.4]{MR3419955}}]\label{prop:McKay_Hearts}
  \label{cohper}
  The equivalence $\Phi \colon D(Y) \to D(\X)$ restricts to an equivalence $\Per(Y/X) \simeq \Coh(\X)$ of abelian categories.
\end{prop}

We record the following two lemmas for use in Section \ref{sec:zetaStability}.
\begin{lem}\label{thm:PushforwardInjectiveOnOneDimClasses}
  Assume that $\X$ satisfies the hard Lefschetz condition.
  Then the map $g_{*} \colon N_{1,\mr}(\X) \to N_{1}(X)$ is injective.
\end{lem}
\begin{proof}
  The McKay equivalence commutes with $g_{*}$ and $f_{*}$, and so identifies the kernels of $g_{*} \colon N_{1,\mr}(\X) \to N_{1}(X)$ and $f_* \colon \Psi(N_{1,\mr}(\X)) = N_{\le 1}(Y)/N_{\exc}(Y) \to N_{1}(X)$.
  But the latter kernel is $0$.
\end{proof}
If $E \in N(Y)$, then we define $E \cdot A$ as the class of $[E] - [E(-A)]$ in $N(Y)$. 
\begin{lem}\label{thm:DivisorsCapAAreLinearlyIndependent}
  Assume that $\X$ satisfies the hard Lefschetz condition.
  Let $D_{1}, \ldots, D_{n} \subset Y$ be the irreducible components of the exceptional locus of $f \colon Y \to X$.
  The classes $D_{1} \cdot A, \ldots, D_{n} \cdot A$ are linearly independent in $N_{1}(Y)$.
\end{lem}
\begin{proof}
  Suppose for a contradiction that $D = \sum a_{i}D_{i}$ is such that not all $a_{i} = 0$, but $D \cdot A = 0$.
  Pick a surface $S$ of class $A$ such that $f^{-1}(S)$ is smooth (replacing $A$ with some multiple if necessary).
  By the {Negativity Lemma}
  \cite[Lem.~3.40]{MR1658959}, we have $D \cdot D \cdot A = (D|_{f^{-1}(S)}, D|_{f^{-1}(S)})_{f^{-1}(S)} < 0$, which is a contradiction.
\end{proof}

\subsection{Rational functions}
The generating series we encounter, formal sums of the form $\sum_{c \in N_{0}(\X)} a(c) q^{c}$, are often expansions of rational functions.
It will be convenient to have a language for describing such expansions.

\subsubsection{Laurent expansions}
To have well-defined expansions in multiple variables, restrictions must be imposed on the sets $\{c \in N_0(\X) \,|\, a(c) \neq 0\}$ appearing.
These can be phrased in terms of various notions of boundedness of subsets of $N_{0}(\X)$.
\begin{defn}\label{def:Lbounded}
  Let $L \colon N_{0}(\X) \to \R$ be a group homomorphism.
  We say that a set $S \subset N_{0}(\X)$ is \emph{$L$-bounded} if $S \cap \{c \in N_{0}(\X) \mid L(c) \le M\}$ is finite for every $M \in \R$.
\end{defn}
\begin{lem}
  \label{thm:LBoundedTrivialities}
  Let $S$ and $T$ be $L$-bounded sets in $N_0(\X)$.
  \begin{enumerate}
    \item The union $S \cup T$ is again $L$-bounded.
    \item The sum $S+T = \{s+t \,|\, s \in S, t \in T\}$ is again $L$-bounded.
  \end{enumerate}
\end{lem}
\begin{defn}\label{def:Completed_Ring_at_L}
  Let $\Z\{N_{0}(\X)\}$ be the additive group of all infinite formal sums of terms $a(c) q^{c}$ with $a(c) \in \Z$, and $\Z[N_{0}(\X)]$ the additive group of all finite such sums.
  We define $\Z[N_{0}(\X)]_{L} \subset \Z\{N_{0}(\X)\}$ to be the subset of those formal sums for which $\{c \in N_{0}(\X) \mid a(c) \neq 0\}$ is $L$-bounded.
\end{defn}
By Lemma \ref{thm:LBoundedTrivialities}, $\Z[N_0(\X)]_L$ is a ring under the obvious operations.
\begin{defn}\label{def:Expansion_Rational_Function_wrt_L}
  Given a rational function $f = g/h$ with $g, h \in \Z[N_{0}(\X)]$, we say that a series $s \in \Z[N_{0}(\X)]_{L}$ is the expansion of $f$ in $\Z[N_{0}(\X)]_{L}$ if $s h = g$ holds in the ring $\Z[N_{0}(\X)]_{L}$.
\end{defn}
Note that such an expansion $s$ may not exist for given $f$ and $L$, but if it does, it is unique, and we will denote it by $f_{L}$.

\subsubsection{Quasi-polynomials}
\label{sec:QuasiPolynomials}
Let $s_{p} \colon \Z^{r} \to (\Z/p)^{r}$ be the standard surjection.
\begin{defn}\label{def:quasipolynomial}
  A function $a \colon \Z^r \to \C$ is said to be a \emph{quasi-polynomial} of \emph{quasi-period} $p$ if $a|_{s_{p}^{-1}(x)}$ is a polynomial function for every $x \in (\Z/p\Z)^r$.
  For $i = 1, \ldots, r$, we define the degree $\deg_{i} a$ as the supremum of the $i$-degrees of the polynomials $a|_{s_{p}^{-1}(x)}$ over all $x \in (\Z/p)^{r}$.
\end{defn}
As motivation for this definition, let us recall the connection between quasi-polynomials and rational functions in the single variable case.
\begin{lem}[{\cite[Prop.~4.4.1]{MR1442260}}]
  The following conditions on a function $a \colon \Z \to \C$ and an integer $N > 0$ are equivalent:
  \begin{enumerate}
  \item $a$ is a quasi-polynomial of quasi-period $p$, and
  \item there exist $g(q), h(q) \in \C[q]$ such that $\gcd(g,h) = 1$ and
    \begin{equation*}
      \sum_{n \geq 0} a(n)q^n = \dfrac{g(q)}{h(q)},
    \end{equation*}
    where every zero $x$ of $h(q)$ satisfies $x^p = 1$, and $\deg(g) < \deg(h)$.
  \end{enumerate}
\end{lem}

We give two generalisations of this result to the multi-variable case.
Let $q_{1}, \ldots q_{r}$ be variables, and pick a grading on $\Z[q_{1}, \ldots, q_{r}]$ such that $\deg(q_{i}) > 0$ for every $i$.\footnote{In our applications, $q_{i} = q^{c_{i}}$ where $\{c_{i}\}$ is an effective basis of $N_{0}(\X)$ and $\deg(q_{i}) \coloneq \deg(c_{i})$.}
Given a rational function $f = g/h$ with $g, h \in \Q[q_{1}, \ldots q_{r}]$, its degree is defined as $\deg(f) = \deg (g) - \deg (h)$; this is independent of the presentation of $f$ as a fraction.
Moreover, if $f$ admits an expansion with each coefficient $q_{1}^{n_{1}}\cdots q_{r}^{n_{r}}$ of degree $\le d$, then $\deg(f) \le d$.
\begin{lem}
  \label{thm:wickedLemmaSimpleVersion}
  Let $p,r \in \Z_{\ge 1}$, and let $a \colon \Z^r \to \C$ be a quasi-polynomial in $r$ variables of quasi-period $p$.
  Consider the generating series
  \begin{equation}
    f(q_1,\ldots,q_r) = \sum_{n_1, \ldots, n_r \ge 0} a(n_1,\ldots, n_r)q_1^{n_1} \cdots q_r^{n_r}.
  \end{equation}
  Then $f$ is the Laurent expansion in $\Q[q_{1}, \ldots, q_{r}]_{\deg}$ of a rational function $g/h$, where $g \in \Q[q_{1}, \ldots, q_{r}]$ and
  \begin{equation}
    h = \prod_{i = 1}^{r} \left(1-q_{i}^{p}\right)^{1 + \deg_{i} a}.
  \end{equation}
  Moreover, $\deg(g/h) < 0$.
\end{lem}
\begin{proof}
  Let $f' = (q_{r}^{-p} - 1)^{1 + \deg_{r} (a)}\cdot f$.
  Directly examining the expression shows that
  \[
    f' = \sum_{n_{r} = -p(1+\deg_{r} a)}^{-1} \sum_{n_{1}, \ldots, n_{r-1} \ge 0} a_{n_{r}}(n_{1}, \ldots, n_{r-1})q_{1}^{n_{1}}\cdots q_{r}^{n_{r}}
  \]
  where the $a_{k}$ are quasi-polynomials of period $p$ in $r-1$ variables with $\deg_{i} a_{k} \le \deg_{i} a$ for $i = 1, \ldots, r-1$.
  The claim then follows by induction on $r$.
\end{proof}

\begin{lem}
  \label{thm:SumDefinedByInequalitiesIsARationalFunction}
  \label{thm:degreeClaimForRationalFunction}
  Let $p,r \in \Z_{\ge 1}$, let $E \subset \{1, \ldots, r-1\}$, and let $a \colon \Z^r \to \C$ be a quasi-polynomial in $r$ variables of quasi-period $p$.
  Consider the generating series
  \begin{equation}
    f(q_1,\ldots,q_r) = \sum_{n_1, \ldots, n_r} a(n_1,\ldots, n_r)q_1^{n_1} \cdots q_r^{n_r},
  \end{equation}
  where the sum runs over all sequences of integers
  \begin{equation*}
    \{(n_1,\ldots,n_r) \in \Z^r \mid 0 \leq n_1 \leq \ldots \leq n_r, \,
    \text{and } n_{i} = n_{i+1} \text{ iff } i \in E\}.
  \end{equation*}
  Then $f$ is the Laurent expansion in $\Q[q_{1}, \ldots, q_{r}]_{\deg}$ of a rational function $g/h$, where $g \in \Q[q_{1}, \ldots, q_{r}]$ and
  \begin{equation}
    h = \prod_{i \in \{0,\ldots,r-1\} \setminus E} \left(1-\prod_{j=i+1}^r q_{j}^{p}\right)^{1 + \sum_{j = i + 1}^{r} \deg_{q_{i}} a}.
  \end{equation}
  Moreover, $\deg(g/h) < 0$.
\end{lem}
\begin{proof}
  Set $n_0 = 0$.
  Let $k_i = n_i - n_{i-1}$ and rewrite the claim in terms of the $k_i$.
  So
  \begin{align*}
    f = \sum_{i \notin E : k_i \ge 0} a(k_1,k_1+k_2,\ldots,k_1+\ldots+k_r)
        \prod_{j=1}^r (q_j \ldots q_r)^{k_j},
  \end{align*}
  where we read $k_{i+1} = 0$ if $i \in E$.
  Setting
  \begin{gather*}
    a'(k_{1}, \ldots, k_{r}) = a(k_{1},k_{1}+k_{2}, \ldots, k_{1} + \ldots + k_{r}), \qquad q_{i}' = \prod_{j = i}^{r} q_{i}
  \end{gather*}
  and applying Lemma \ref{thm:wickedLemmaSimpleVersion} to $a'$ completes the proof.
\end{proof}

The following result, used in Section \ref{sec:zetaStability}, allows us to detect when two generating series are different Laurent expansions of the same rational function.
\begin{lem}\label{thm:QuasipolynomialDifferenceIsANewExpansion}
  Let $c_{0} \in N_0(\X)$, and let $L_{-}, L_{+} \colon N_0(\X) \to \R$ be linear functions such that $L_{-}(c_{0}) < 0$ and $L_{+}(c_{0}) > 0$.
  Let $f$ be a rational function admitting an expansion $f_{L_{-}}$ in $\Z[N_0(\X)]_{L_{-}}$, and let $f' \in \Z[N_{0}(\X)]_{L_{+}}$ be a series such that for any $c \in N_0(\X)$, the coefficient of $q^{c + kc_{0}}$ in $f_{L_{-}} - f'$ is a quasi-polynomial in $k \in \Z$.

  Then $f' = f_{L_{+}}$ in $\Z[N_0(\X)]_{L_{+}}$.
\end{lem}
\begin{proof}
  Write $f = g/h$ for polynomials $g,h \in \Z[N_0(\X)]$.
  Consider the expression $(f_{L_{-}}-f')h$.
  This has the property that the coefficient of $q^{c + kc_{0}}$ is quasi-polynomial in $k$ for any $c \in N_0(\X)$, since $f_{L_{-}} - f'$ has this property and $h$ is a polynomial.

  On the one hand, if we let $k \to -\infty$ the coefficient of $q^{c+ kc_{0}}$ in $f'h$ is $0$ since $f' \in \Z[N_0(\X)]_{L_{+}}$ and $L_+(c_0) > 0$.
  On the other hand, the coefficient of $q^{c + kc_{0}}$ in $f_{L_{-}}h$ is $0$ since $f_{L_{-}}h = g$ and $g$ is a polynomial.
  By quasi-polynomiality of the coefficients of the difference $(f_{L_-}-f')h$, it follows that each quasi-polynomial coefficient is $0$.
  Thus $(f_{L_{-}}-f')h = 0$ in $\Z\{N_0(\X)\}$, and the claim follows.
\end{proof}

\begin{ex}
  Consider the rational function $f(q) = 1/(1-q)$, with $g = 1$ and $h = 1-q$.
  There are essentially two choices for the linear function $L$, namely $L_{\pm} \colon \Z \to \R$ where $L_{\pm}(k) = \pm k$.
  We have 
  \begin{equation*}
     f_{L_+}(q) = \sum_{n=0}^{\infty} q^n
  \end{equation*}
  since $(1-q)f_{L_{+}}(q) = 1$ in $\Z[x]_{L_+}$.
  Alternatively, we may  re-expand the function $f$ with respect to $L_{-}$.
  Our educated guess for the expansion of $f$ in $\Z[q]_{L_-}$ is
  \begin{equation*}
    f'(q) = - \sum_{n = 1}^{\infty} (q^{-1})^{n}.
  \end{equation*}
  Indeed, $f' \in \Z[q]_{L_{-}}$.
  To apply the previous lemma take $c_0 = 1$, so $L_{+}(c_0) > 0$ and $L_{-}(c_0) < 0$.
  The coefficient of the term $q^{c+kc_0}$ in the difference
  \begin{equation}\label{eq:RationalityStrangeLemma}
    f_{L_+}(q) - f'(q) = \sum_{n = -\infty}^{\infty} q^n
  \end{equation}
  is equal to $1$ for all $c, k \in \Z$.
  This is a quasi-polynomial of quasi-period $N = 1$.
  Hence
  \begin{equation*}
    f'(q) (1-q) =  \left(-\sum_{n = 1}^{\infty} (q^{-1})^{n} \right) (1-q) = 1 \quad \text{in} \quad \Z[q]_{L_-}
  \end{equation*}
  as the reader easily verifies, and we conclude that $f' = f_{L_-}$.
  Note that the choice of homomorphism $L_-$ means we are expanding the rational function $f$ around $q = \infty$.
  The root $\alpha$ of the denominator of $f$ satisfies $\alpha^N = 1$, and $\deg(f) < 0$ indeed.
\end{ex}


\section{Categories and pairs}
For the remainder of this paper, we let $\X$ be a CY3 orbifold in the sense of Section~\ref{suborbifolds}.
Here, we recall the abelian category $\A = \langle \hO_{\X}[1], \Coh_{\le 1}(\X)\rangle_{\eex}$, which contains all the objects we count.
It is the heart of the category of D$0$-D$2$-D$6$ bound states constructed in \cite{todajams}.
We discuss torsion pairs $(\T,\F)$ on $\Coh_{\le 1}(\X)$, and introduce the notion of \emph{$(\T,\F)$-pairs}. 
These are rank $-1$ objects in $\A$ which should be though of as \emph{stable} with respect to $(\T,\F)$.
Choosing $(\T,\F)$ suitably, this notion specialises to ideal sheaves, stable pairs, or Bryan--Steinberg pairs.

\subsection{Torsion pairs and tilting}
If $\B$ is an abelian category, a \emph{torsion pair} (see e.g.~\cite{MR1327209}) consists of a pair of full subcategories $(\T,\F)$ such that
\begin{enumerate}
\item $\Hom(T,F) = 0$, for all $T \in \T$, $F \in \F$,
\item every object $E \in \cat{B}$ fits into a short exact sequence
  \begin{equation*}
    0 \to T_{E} \to E \to F_{E} \to 0
  \end{equation*}
  with $T_{E} \in \T$ and $F_{E} \in \F$.
\end{enumerate}
We write $\cat B = \langle T, F \rangle$.
Note that the first condition implies that the sequence in the second condition is unique.
In our applications, $\cat B$ will be $\Coh_{\leq 1}(\X)$, and pairs will be induced by various stability conditions.

The following result provides a simply method of constructing torsion pairs.
\begin{lem}[{\cite[Lem.~2.15]{MR3021446}}]\label{lem:TodaTorsion}
  Let $\cat{B}$ be a Noetherian abelian category and let $\T$ be a full subcategory that is closed under extensions and quotients.
  Let
  \begin{equation*}
    \F = \T^{\perp} = \{F \in \cat{B} \,|\, \Hom(T,F) = 0 \text{ for all } T \in \T\}.
  \end{equation*}
  Then $(\T,\F)$ is a torsion pair on $\cat{B}$.
\end{lem}

\begin{ex}
  \label{exp:DimensionalTorsionPair}
  Consider the full subcategory $\T = \Coh_{\leq d}(\X) \subset \cat{B} \coloneq \Coh(\X)$.
  Then $\F = \Coh_{\ge d+1}(\X)$, \ie~the full subcategory of sheaves that admit no subsheaves of dimension $\le d$.
  By the previous lemma, we obtain a torsion pair
  \begin{equation}
    \Coh(\X) = \langle \Coh_{\le d}(\X), \Coh_{\ge d+1}(\X) \rangle.
  \end{equation}
\end{ex}
\begin{ex}
  All the torsion pairs on $\Coh_{\leq 1}(\X)$ that play are role, are constructed as follows.
  Let $S$ be a totally ordered set, and let $\mu \colon N^{\eff}_{\le 1}(\X) \to S$ be a stability condition on $\Coh_{\le 1}(\X)$.
  A choice of element $s \in S$ defines a torsion pair on $\Coh_{\le 1}(\X)$
  \begin{align*}
    \T_{\mu,s} &= \{T \in \Coh_{\le 1}(\X) \mid T \onto Q \neq 0 \implies \mu(Q) \ge s\}\\
    \F_{\mu,s} &= \{F \in \Coh_{\le 1}(\X) \mid 0 \neq E \into F \implies \mu(E) < s\}
  \end{align*}
  by collapsing the Harder--Narasimhan filtration of $\mu$-stability at slope $s$.
\end{ex}



The abelian category $\cat{B}$ defines the standard t-structure on $D(\cat{B})$.
In the presence of a torsion pair, one obtains a different t-structure via the process of \emph{tilting}.
\begin{prop}
  Let $(\T,\F)$ be a torsion pair on the abelian category $\cat{B}$.
  Then
  \begin{equation*}
    \cat{B}^{\flat} = \{E \in D(\cat{B}) \,|\, H^{-1}(E) \in \F,\, H^0(E) \in \T,\, H^i(E) = 0 \text{ if } i \neq -1,0\}
  \end{equation*}
  defines the heart of a bounded t-structure on $D(\cat{B})$.
  In particular, $\cat{B}^{\flat}$ is abelian and closed under extensions.
\end{prop}
\begin{ex}\label{ex:defCohFlat}
Tilting at the torsion pair of Example \ref{exp:DimensionalTorsionPair} for $d = 1$ yields the heart
  \begin{equation*}
    \Coh^{\flat}(\X) = \langle \Coh_{\ge 2}(\X)[1], \Coh_{\le 1}(\X) \rangle \subset D^{[-1,0]}(\X).
  \end{equation*} 
\end{ex}

\subsection{The category \texorpdfstring{$\A$}{A}}
The objects underlying DT invariants are elements of $\Coh^{\flat}(\X)$.
However, $\Coh^{\flat}(\X)$ is not Noetherian, and so in view of Lemma \ref{lem:TodaTorsion}, we instead work with a certain Noetherian subcategory of $\Coh^{\flat}(\X)$, cf.~\cite{todajams}.

Given full subcategories categories $\hC_{1}, \ldots, \hC_{n} \subset D(\X)$, let $\langle \hC_{1}, \ldots, \hC_{n}\rangle_{\eex} \subset D(\X)$ be the smallest extension-closed subcategory of $D(\X)$ containing each $\hC_{i}$.
\begin{defn}
  The category $\A$ is the full subcategory of $\Coh^{\flat}(\X)$ defined as
  \begin{equation}
    \A = \langle \hO_\X[1], \Coh_{\le 1}(\X) \rangle_{\text{ex}} \subset \Coh^{\flat}(\X) \subset D(\X).
  \end{equation}
\end{defn}
Note that $\Hom(\hO_\X[1],E) = 0 = \Hom(E,\hO_\X[1])$ for all $E \in \Coh_{\leq 1}(\X)$.
\begin{lem}\label{lem:PropertiesOfA}
  The category $\A$ satisfies the following properties.
  \begin{enumerate}
  \item $\A$ is a Noetherian abelian category with exact inclusion $\A \subset D(\X)$, which is to say that given $E', E, E'' \in \A$, a sequence of morphisms
    \[
      E' \to E \to E''
    \]
   is a short exact sequence in $\A$ if and only if it is an exact triangle in $D(\X)$.
  \item If $E \in \A$, then $H^{-1}(E)$ is torsion free or $0$, $H^0(E) \in \Coh_{\leq 1}(\X)$, and $H^{i}(E) = 0$ for $i \not= -1,0$.
  \item The subcategory $\Coh_{\leq 1}(\X) \subset \A$ is closed under extensions, quotients, and subobjects,
    and the inclusion $\Coh_{\leq 1}(\X) \subset \A$ is exact.
  \item $\A$ contains the shifted structure sheaf $\os_\X[1]$, the shifted ideal sheaf $I_C[1]$ of any curve $C \subset \X$, and stable pairs in the sense of Pandharipande--Thomas.
  \item $\A$ contains all Bryan--Steinberg pairs in the sense of Definition~\ref{def:BSpairs}
  \end{enumerate}
\end{lem}
\begin{proof}
  The first item is proven in \cite[Lem.~3.5, 3.8]{todajams}, and claims two to four are easily verified.
  The final claim is proven in Lemma~\ref{lem:BSpairs_in_A}.
\end{proof}

\begin{rmk}\label{rmk:splittingA}
  Let $N(\A) \subset N(\X)$ denote the subgroup generated by objects in $\A$.
  The inclusion $\Coh_{\leq 1}(\X) \subset \A$ induces an injection of abelian groups $i \colon N_{\leq 1}(\X) \into N(\A)$.
  The image of a class $\alpha \in N(\A)$ in the cokernel of $i$ equals $\rk(\alpha)$, and $n \mapsto -n[\hO_{\X}[1]]$ defines a splitting of this map, so we have a canonical splitting $N(\A) = \Z \oplus N_{\leq 1}(\X)$.
  For $E \in \A$ with $[E] = \rk(E)[\hO_{\X}] + [E']$, we write $[E] = (\rk(E),\beta_{E'},c_{E'})$.
\end{rmk}

\subsection{Pairs}
We now define the objects which we aim to count.
Note that in the following definition, $(\T,\F)$ need not be a torsion pair.
\begin{defn}\label{def:TFpair}
  Let $\T$ and $\F$ be full subcategories of $\Coh_{\leq 1}(\X)$.
  A \emph{$(\T,\F)$-pair} is an object $E \in \A$ such that
  \begin{enumerate}
  \item $\rk(E) = -1$,
  \item $\Hom(T,E) = 0$ for all $T \in \T$,
  \item $\Hom(E,F) = 0$ for all $F \in \F$.
  \end{enumerate}
  We write $\cat{Pair}(\T,\F) \subset \A$ for the corresponding full subcategory of $(\T,\F)$-pairs.
\end{defn}
\begin{rmk}\label{rmk:TFpairStandardForm}
  Assuming that $(\T,\F)$ is a \emph{torsion} pair, two things follow.
  The third condition is equivalent to $H^0(E) \in \T$.
  Moreover, if $E$ is of the form $\hO_{\X} \to G$ with $G \in \Coh_{\le 1}(\X)$, then the second condition is equivalent to $G \in \F$.
\end{rmk}
Under a cohomological criterion on $\T$, all pairs are of a standard form.
\begin{lem}\label{prop:CohomCrit}
  Let $(\T, \F)$ be a torsion pair on $\Coh_{\leq 1}(\X)$ such that every $T \in \T$ satisfies $H^i(\X,T) = 0$ for all $i \neq 0$.
  Then an object $E \in \A$ of rank $-1$ is a $(\T, \F)$-pair if and only
  if it is isomorphic to a two-term complex
  \begin{equation*}
    E = (\hO_\X \stackrel s \to G)
  \end{equation*}
  with $H^0(E) = \coker(s) \in \T$ and $G \in \F$.
\end{lem}
\begin{proof}
  The proof of \cite[Lem.~3.11(ii)]{todajams} goes through verbatim.
\end{proof}
\begin{rmk}\label{ex:PTasTF}
  This cohomological criterion implies that if $\T_{PT} = \Coh_{0}(\X)$ and $\F_{PT} = \Coh_{1}(\X)$, then a $(\T_{PT},\F_{PT})$-pair is the same thing as a stable pair in the sense of Pandharipande--Thomas \cite{MR2545686}.

  The conclusion of Lemma \ref{prop:CohomCrit} can fail if the condition on $\T$ is not satisfied.
  For example, this happens for the perverse torsion pair induced on $\Coh_{\leq 1}(\X)$ via the equivalence of Proposition~\ref{prop:McKay_Hearts}; see \cite[Lem.~3.1.1]{MR2057015}.
\end{rmk}

\subsubsection{A wall-crossing formula}
Recall the following natural generalisation of the notion of torsion pair \cite{2016arXiv160107519T}.
\begin{defn}
  Let $(\A_1,\A_2,\ldots,\A_n)$ be a sequence of full subcategories of an abelian category $\cat{B}$.
  These form a \textit{torsion $n$-tuple},
  notation $\cat{B} = \langle \A_1,\A_2, \ldots, \A_n \rangle$, if
  \begin{enumerate}
  \item we have $\Hom(E_i,E_j) = 0$ for $E_{i} \in \A_{i}$, $E_{j} \in \A_{j}$, $i < j$,
  \item for every object $E \in \cat{B}$ there is a filtration
    \begin{equation*}
      0 = E_0 \subset E_1 \subset \ldots \subset E_n = E
    \end{equation*}{}
    in $\B$ such that $Q_i = E_i/E_{i-1} \in \A_i$ for all $i = 1,2,\ldots,n$.
  \end{enumerate}
  The first condition implies that the filtration in the second condition is unique.
\end{defn}

\begin{rmk}
  Note that for any $1 \leq i \leq n-1$ we obtain a torsion pair
  \begin{equation*}
    \cat{B} = \langle \langle \A_1,\ldots,\A_i \rangle_{\eex},
    \langle \A_{i+1} \ldots, \A_n \rangle_{\eex} \rangle
  \end{equation*}
  by collapsing the filtration of $E$ to $0 \to E_i \to E \to E/E_i \to 0$ in $\cat{B}$.
\end{rmk}
Let $(\T,\F)$ be a pair of full subcategories of $\Coh_{\leq 1}(\X)$.
Define the full subcategory
\begin{equation}\label{eq:defofV}
  \cat{V}(\T,\F) = \left\{ E \in \A \mid \Hom(\T,E) = \Hom(E,\F) = 0 \right\} \subset \A.
\end{equation}
A $(\T, \F)$-pair is the same thing as a rank $-1$ object of $\cat V(\T,\F)$.
\begin{prop}\label{prop:InducedTorsionTriple}
  Let $(\T,\F)$ and $(\tilde{\T},\tilde{\F})$ be torsion pairs on $\Coh_{\leq 1}(\X)$ with $\T \subset \tilde{\T}$.
  There is an induced torsion triple on $\cat A$,
  \begin{equation}\label{eq:InducedTorsionTriple}
    \A = \langle \T, \cat{V}(\T,\tilde{\F}), \tilde{\F} \rangle.
  \end{equation}
\end{prop}
\begin{proof}
  We set $\cat{V}(\T,\tilde{\F}) = \T^{\perp} \cap ^{\perp}\!\tilde{\F}$.
  The semi-orthogonality relations are clear, so it suffices to construct a filtration of every object with factors in $\T, \cat{V}(\T,\tilde{\F})$, and $\tilde{\F}$.

  Let $E \in \cat A$.
  Since $\T$ is closed under quotients and extensions and $\cat A$ is Noetherian, Lemma \ref{lem:TodaTorsion} shows that there exist a unique exact sequence
  \[
    0 \to E_{\T} \to E \to E_{\T^{\perp}} \to 0 
  \]
  with $E_{\T} \in \T$, $E_{\T^{\perp}} \in \T^{\perp}$.
  Defining $E_{\tilde{\F}}$ as the projection of $H^{0}(E)$ to $\tilde{\F}$ induced by the torsion pair $(\tilde{\T}, \tilde{\F})$, we obtain the unique short exact sequence
  \[
    0 \to E_{^{\perp}\!\tilde{\F}} \to E \to E_{\tilde{\F}} \to 0
  \]
  with $E_{^{\perp}\!\tilde{\F}} \in {}^{\perp}\!\tilde{\F}$ and $E_{\tilde{\F}} \in \tilde{\F}$.
  The desired filtration is $0 \into E_{\T} \into E_{^{\perp}\!\tilde{\F}} \into E$.
\end{proof}

\subsubsection{Objects of small rank}\label{subsec:ObSmallRank}
The rank of an object $E \in \Coh^{\flat}(\X)$ is non-positive since $\rk(E) = -\rk H^{-1}(E)$.
In this paper we are exclusively interested in objects of rank $-1$ and $0$, which admit the following explicit description.
\begin{prop}\label{prop:ObjectsSmallRank}
  Let $E \in \Coh^{\flat}(\X)$.
  \begin{enumerate}
  \item If $\rk(E) = 0$, then $E \in \A$ if and only if $H^{-1}(E) = 0$;
  \item If $\rk(E) = -1$, then $E \in \A$ if and only if $H^{-1}(E)$ is torsion free and $\det(E) = \hO_\X$.
  \end{enumerate}
\end{prop}
\begin{proof}
  If $E \in \cat A$ there exists a filtration $0 = E_{0} \subset \cdots \subset E_{n} = E$, with each $E_{i}/E_{i-1}$ either lying in $\Coh_{\le 1}(\X)$, or else isomorphic to $\hO_{\X}[1]$.
  The rank of $E$ is the negative of the number of subquotients isomorphic to $\hO_{\X}[1]$.

  Thus if $\rk E = 0$, then each subquotient is a sheaf, and so $H^{-1}(E) = 0$.
  Conversely, if $H^{-1}(E) = 0$, then $E \in \Coh(\X) \cap \Coh^{\flat}(\X) = \Coh_{\le 1}(\X) \subset \cat A$.

  If $E \in \A$, and $F \into E' \onto E$ is a short exact sequence in $\A$ with $F \in \Coh_{\le 1}(\X)$, then $H^{-1}(E') \into H^{-1}(E)$.
  If $F \cong \hO_{\X}[1]$, then we get a short exact sequence of sheaves $\hO_{\X} \into H^{-1}(E') \onto H^{-1}(E)$.
  In both cases, $H^{-1}(E')$ is torsion free if $H^{-1}(E)$ is and $\det(E') = \det(E)$.
  By induction on the number of extensions needed to construct an object, every $E \in \cat A$ is such that $H^{-1}(E)$ torsion free and $\det(E) = \hO_{\X}$.

  Conversely, if $\rk(E) = -1$ and $\det(E) = \hO_{\X}$, then $H^{-1}(E)$ is torsion free of trivial determinant, hence equal to $I_{C}$ for some curve $C \subset \X$.
  Since $I_{C}[1] = (\hO_{\X} \to \hO_{C}) \in \A$ and $H^{0}(E) \in \Coh_{\le 1}(\X) \in \cat A$, we conclude that $E \in \cat A$.  
\end{proof}
We denote the subcategory of objects in $\A$ of rank $-1$ 
(resp.~$\ge -1$) by $\A_{\rk = -1}$ (resp.~$\A_{\rk \ge -1}$).
\begin{cor}\label{cor:Class_TFpair_Effective_HMinusOne_IdealSheaf}
  Let $E \in \A_{\rk \ge -1}$.
  Then $H^{-1}(E)$ is either $0$ or the ideal sheaf of a curve $C \subset \X$ of class $\beta_C = \beta_{\hO_{C}}$, and $\beta_{E} = \beta_{C} + \beta_{H^{0}(E)}$.
  In particular, $\beta_{E} \geq 0$.
\end{cor}

\begin{cor}
  \label{thm:subAndQuotientObjectsHaveLowerD}
  Let $E \in \cat A_{\rk \ge -1}$, and let $F \in \cat A_{\rk \ge -1}$ be a subobject or quotient of $E$.
  Then $\beta_{F} \le \beta_{E}$.
\end{cor}
\begin{proof}
  Given any exact sequence $0 \to F \to E \to F' \to 0$, we have $\beta_{F}, \beta_{F'} \ge 0$ and $\beta_{F} + \beta_{F'} = \beta_{E}$. 
\end{proof}

\section{Moduli stacks}
In this section we gather some results on various moduli stacks of objects in $D(\X)$, beginning with Lieblich's stack $\Mum_{\X}$ of gluable objects in $D(\X)$.
The main result of this section is Proposition~\ref{thm:TFOpenImpliesPOpen}, which states that the stack of $(\T,\F)$-pairs defines an open substack of $\Mum_{\X}$, provided that $\T$ and $\F$ define open substacks of $\Mum_{\X}$ and certain mild hypotheses are met.

\subsection{The mother of all moduli}
Let $Y$ be a smooth projective variety.
In \cite{MR2177199}, Lieblich constructs a stack $\Mum_Y$ parametrising objects $E \in D(Y)$ which are \emph{gluable}, \ie, such that $\Ext^{<0}(E,E) = 0$.
He shows that $\Mum_Y$ is an Artin stack, locally of finite type.
Via the McKay equivalence $D(Y) \simeq D(\X)$ of Section~\ref{subsubsec:CrepantResolutions}, which is a Fourier--Mukai transform and hence behaves well in families, one deduces the existence of the corresponding stack $\Mum_\X$ with the same properties.

The stack $\Mum_\X$ splits as a disjoint union of open and closed substacks
\begin{equation}\label{eq:MumXdecomposition}
  \Mum_\X = \coprod_{\alpha \in N(\X)} \Mum_{\X,\alpha}
\end{equation}
where $\Mum_{\X,\alpha}$ parametrises objects of class $\alpha \in N(\X)$.

Let $\cat C \subset D(\X)$ be a full subcategory whose objects have vanishing negative self-extensions, and whose objects define an open subset of $\Mum_{\X}(\C)$ in the sense that for every finite type $\C$-scheme $T$ with a morphism $f \colon T \to \Mum_{\X}$, the set $\{x \in T(\C) \mid f(x) \in \cat C\} \subset T(\C)$ is Zariski open.
In this case we say that $\cat C$ is an \emph{open subcategory} and we write $\uC \subset \Mum_{\X}$ for the corresponding open substack.

For any interval $[a,b]$, there is an open substack $\Mum^{[a,b]}_\X\subset\Mum_\X$, parametrising complexes $E \in D(\X)$ with vanishing negative self-extensions and whose cohomology is concentrated in cohomological degrees $[a,b]$; see for example \cite[App.~A]{MR3518373}.
In particular, the stack of coherent sheaves
\begin{equation*}
  \cCoh_\X = \Mum^{[0,0]}_\X \subset \Mum_\X
\end{equation*}
is an open substack.

\begin{lem}
  \label{lem:CohflatIsOpen}
  \label{lem:maxlemma}
  Let $\T$ and $\F$ be full, open subcategories of $\Coh(\X)$ such that $\T$ is closed under quotients and $\F$ is closed under subobjects.
  Each of the following conditions on an object $E \in D(\X)$ defines an open substack of $\Mum_{\X}$:
  \begin{enumerate}
  \item $H^{0}(E) \in \T$ and $H^{i}(E) = 0$ for $i \not= -1,0$.
  \item $H^{-1}(E) \in \F$ and $H^{i}(E) = 0$ for $i \not= -1,0$.
  \end{enumerate}

  In particular, if $(\T, \F)$ is an open torsion pair, then the objects of the tilt of $\Coh(\X)$ at $(\T, \F)$ form an open substack of $\Mum_{\X}$.
\end{lem}
\begin{proof}
  The claim about the tilt of $\Coh(\X)$ is shown in \cite[Thm.~A.3]{MR2998828}, and the proof given there also establishes the other two statements in this lemma.
\end{proof}

Recall the abelian category $\Coh^{\flat}(\X) = \langle \Coh_{\ge 2}(\X)[1], \Coh_{\le 1}(\X) \rangle_{\eex}$ from Example~\ref{ex:defCohFlat}.
It is an open subcategory by the above lemma.
\begin{lem}\label{lem:OpennessS}
  The category $\A_{\rk \ge -1} \subset \Coh^{\flat}(\X)$ is open.
  In particular, $\uA_{\rk \ge -1}$ is an Artin stack locally of finite type.
\end{lem}
\begin{proof}
  The set of objects in $\cat A_{\rk \ge -1}$ splits as a disjoint union according to their rank.
  The $\rk = 0$ component is $\Coh_{\leq 1}(\X)$, which is open.

  By Proposition~\ref{prop:ObjectsSmallRank}, an object $E \in \Coh^{\flat}(\X)$ of rank $-1$ lies in $\cat A$ if and only if $H^{-1}(E)$ is torsion free and $\det(E) = \hO_{\X}$.
  The former condition is open by Lemma~\ref{lem:CohflatIsOpen}, whereas the latter condition is open because $H^{1}(\X, \hO_{\X}) = 0$.
\end{proof}

\subsection{Openness of pairs}
We are now in a position to prove that the moduli stack of pairs $\uPair(\T,\F)$ is an open substack of $\Mum_\X$ under some mild assumptions on $(\T,\F)$.
In particular, this establishes the fact that $\uPair(\T,\F)$ is an Artin stack, locally of finite type.
\begin{defn}\label{def:Open_Torsion_Pair}
  A torsion pair $(\T,\F)$ on $\Coh_{\leq 1}(\X)$ is called \emph{open} if the subcategories $\T,\F \subset \Coh_{\leq 1}(\X)$ are open.
\end{defn}
Let $(\T,\F)$ be an open torsion pair on $\Coh_{\leq 1}(\X)$.
Note that an object $E \in \A_{\rk = -1}$ is a $(\T,\F)$-pair precisely when
\begin{enumerate}
\item $H^0(E) \in \T$, and 
\item $\Hom(T,E) = 0$ for all $T \in \T$.
\end{enumerate}
Condition (1) is open by Lemma \ref{lem:maxlemma}.
To show that condition (2) is also open, we reformulate it in terms of a condition on the derived dual of $E$ that is open.

By \emph{derived dual} we mean the anti-equivalence of $D(\X)$ given by
\begin{equation}\label{def:DerivedDual}
  \D(-) = \lRHom(-,\hO_\X)[2].
\end{equation}
Note that $\D(\Coh_1(\X)) = \Coh_1(\X)$ and $\D(\Coh_0(\X)) = \Coh_0(\X)[-1]$.
In particular, the abelian category $\D(\Coh_{\leq 1}(\X))$ has a torsion pair given by
\begin{equation}\label{eq:dual_torsion_pair_Coh1}
  \D(\Coh_{\leq 1}(\X)) = \langle\Coh_1(\X), \Coh_0(\X)[-1]\rangle.
\end{equation}

\begin{lem}\label{lem:Dual_Object_A}
  If $E \in \A$, then $H^{i}(\D(E)) = 0$ if $i \neq -1, 0, 1$, and $H^1(\D(E)) \in \Coh_0(\X)$.
\end{lem}
\begin{proof}
  If $E \in \Coh_{\le 1}(\X)$, the claim follows from the above discussion.
  If $E = \hO_{\X}[1]$, then $\D(E) = E$.
  Since the conclusion of the lemma is a property preserved by extensions, the claim now follows for all $E \in \A$.
\end{proof}

We wish to compare notions of pair with respect to different torsion pairs.
\begin{lem}\label{lem:TFPairs_Duals}
  Let $(\T,\F)$ and $(\tT,\fF)$ be two torsion pairs on $\Coh_{\leq 1}(\X)$.
  Assume that $\Coh_0(\X) \subset \T$, so $\F \subset \Coh_1(\X)$.
  An object $E \in \A_{\rk = -1}$ is a $(\T,\fF)$-pair if and only if the following two conditions hold:
  \begin{enumerate}
  \item $H^0(E) \in \tT$,
  \item $H^1(\D(E)) = 0$ and $H^0(\D(E)) \in \langle \Coh_0(\X), \D(\F) \rangle_{\eex}$.
  \end{enumerate}
\end{lem}
\begin{proof}
  Let $E \in \A_{\rk = -1}$.
  Then $\Hom(E,\fF) = 0$ is equivalent to $H^0(E) \in \tT$.

  Let $\T_1 = \T \cap \Coh_1(\X)$.
  Since $\Coh_0(\X) \subset \T$, we have $\T = \langle \Coh_{0}(\X), \T_{1} \rangle_{\eex}$.
  Thus $E \in \T^{\perp}$ is equivalent to $E \in \Coh_{0}(\X)^{\perp} \cap \T_{1}^{\perp}$, which is equivalent to
  \begin{equation*}
    \D(E) \in {}^{\perp}\Coh_{0}(\X)[-1] \cap {}^{\perp}\D(\T_{1}).
  \end{equation*}
  By Lemma \ref{lem:Dual_Object_A}, $\D(E) \in {}^{\perp}\Coh_{0}(\X)[-1]$ if and only if $H^{1}(\D(E)) = 0$, and if this holds then $E \in {}^{\perp}\D(\T_{1})$ if and only if $H^0(\D(E)) \in {}^{\perp}\D(\T_{1}) = \langle \Coh_{0}(\X) , \D(\F) \rangle_{\eex}$.
\end{proof}

\begin{prop}
  \label{thm:TFOpenImpliesPOpen}
  Let $(\T,\F)$ and $(\tT,\fF)$ be open torsion pairs on $\Coh_{\leq 1}(\X)$.
  Assume that $\Coh_0(\X) \subset \T$, so $\F \subset \Coh_1(\X)$.
  The substack $\uPair(\T,\fF) \subset \Mum_\X$ parametrising $(\T,\fF)$-pairs is open.
  In particular, it is an algebraic stack locally of finite type.
\end{prop}
\begin{proof}
  Note that the duality functor $\D$ induces an automorphism of the stack $\Mum_\X$.
  In particular, if $G \in \Mum_\X(S)$ is a family of universally gluable complexes over some base $S$, then so is its dual $\D(G) = \lRHom(G,\hO_{S\times \X})[2] \in \Mum_\X(S)$.
  
  Lemma \ref{lem:TFPairs_Duals} shows that $E \in \cat A_{\rk = -1}$ is a $(\T,\fF)$-pair if and only if three properties hold:
  (i) $H^0(E) \in \tT$, (ii) $H^1(\D(E)) = 0$, and (iii) $H^0(\D(E)) \in \G \coloneq \langle \Coh_0(\X), \D(\F) \rangle_{\text{ex}}$.
  By Lemma \ref{lem:maxlemma}, the first condition is open.
  As $\D(E) \in D^{[-1,1]}(\X)$ by Lemma \ref{lem:Dual_Object_A}, the second condition is open as well.    
  
  Set $\T_{1} = \T \cap \Coh_{1}(\X)$.
  Applying $\D$ to the torsion triple $\langle \Coh_{0}(\X), \T_{1}, \F\rangle$ yields a refinement to a torsion triple of the torsion pair in equation~\eqref{eq:dual_torsion_pair_Coh1}.
  Tilting at this torsion pair, we obtain the torsion triple
  \begin{equation*}
    \Coh_{\leq 1}(\X) = \langle \Coh_0(\X), \D(\F), \D(\T_1) \rangle.
  \end{equation*}
  Thus $\G$ is closed under extensions and quotients.
  We now claim that $\G$ is open, which is equivalent to $\D(\G) = \langle \Coh_{0}(\X)[-1], F \rangle_{\eex}$ being open.
  But
  \begin{equation*}
  \D(\G) = \langle \Coh_{1}(\X), \Coh_{0}(\X)[1]\rangle \cap \langle \F, \T[-1] \rangle
  \end{equation*}
  since $\Coh_0(\X) \subset \T$, and both of which are open by Lemma \ref{lem:CohflatIsOpen}.
\end{proof}



\section{Hall algebras}
In this section we define the motivic Hall algebra of the heart of a bounded t-structure $\cC \subset D(\X)$ and recall some of its properties.
For a more detailed discussion, we refer to \cite{MR2813335,MR2854172,MR3518373,1612.00372, 2016arXiv160107519T}.


\subsection{Grothendieck rings}\label{subsec:GrothendieckRings}
The Grothendieck ring $K(\St/\C)$ is the $\Q$-vector space generated by symbols $[X]$, where $X$ is a finite type Artin stack over $\C$ with affine geometric stabilisers.
These symbols are subject to the following relations.
\begin{enumerate}
  \item $[X \sqcup Y] = [X] + [Y]$.
  \item If $f \colon X \to Y$ is a geometric bijection, \ie, $f$ induces an equivalence of groupoids $X(\C) \to Y(\C)$,
    then $[X] = [Y]$.
  \item If $X_{1}, X_{2} \to Y$ are Zariski fibrations\footnote{See \cite{MR2854172} for the definition of this term, which will not be used in this paper.}  with the same fibres, then $[X_{1}] = [X_{2}]$.
\end{enumerate}
One may multiply classes by taking products: $[X]\cdot [Y] \coloneq [X\times Y]$.
This turns $K(\St/\C)$ into a commutative $\Q$-algebra, with unit given by $[\Spec \C]$.

Let $\cal S$ be an Artin stack locally of finite type with affine geometric stabilizers.
We have a relative version $K(\St/\cal{S})$, which is the $\Q$-vector space generated by symbols $[X \to \cal{S}]$, where $X$ is a \emph{finite} type Artin stack over $\C$ with affine geometric stabilisers, and where these symbols are subject to relative versions of the three relations above.
The vector space $K(\St/\cal{S})$ is a $K(\St/\C)$-module, where the module structure is given by setting $[X] \cdot [Y \to \cal{S}] = [X \times Y \to Y \to \cal{S}]$.

For the remainder of this section we fix an open substack $\uC \subset \Mum_\X$ satisfying the hypotheses of Appendix~\ref{sec:BRaxioms}.
In our applications, we take $\cC = \Coh^{\flat}(\X)$ as before.

\subsection{The motivic Hall algebra}
There exists a stack $\uC^{(2)}$ of short exact sequences in the category $\cC$.
It comes with three distinguished maps $\pi_i \colon \uC^{(2)} \to \uC$, $i = 1,2,3$.
The map $\pi_i$ corresponds to sending a short exact sequence $0 \to E_1 \to E_2 \to E_3 \to 0$ to the object $E_i$.
The following proposition is shown in Appendix \ref{sec:BRaxioms}.
\begin{prop}
  \label{thm:extensionStackIsFiniteType}
  The stack $\uC^{(2)}$ is an Artin stack, locally of finite type.
  The morphism $(\pi_1,\pi_3)\colon \uC^{(2)} \to \uC \times \uC$ is of finite type.
\end{prop}

Given two elements $[X_{1} \to \uC]$, $[X_{2} \to \uC]$ of $K(\St/\uC)$, take their fibre product
\begin{equation}\label{eq:HallDiagram}
\begin{tikzcd}
X_1 * X_2 \arrow[r] \arrow[d] \arrow[dr, phantom, "\square"] & \uC^{(2)} \arrow[r, "{\pi_{2}}"] \arrow[d, "{(\pi_{1}, \pi_{3})}"] & \uC \\
X_1 \times X_2 \arrow[r] &\uC \times \uC &
\end{tikzcd}
\end{equation}
Note that $X_1 * X_2$ is again a finite-type stack over $\C$ with affine geometric stabilizers.
Thus the class of the top horizontal line of the diagram $[X_{1} * X_{2} \to \uC]$ is again an element of $K(\St/\uC)$.
\begin{prop}
The operation $([X_{1} \to \uC], [X_{2} \to \uC]) \mapsto [X_{1} * X_{2} \to \uC]$ defines an associative product on $K(\St/\uC)$ with unit $\ID_0 = [\pt \to \uC]$ corresponding to the stack of zero objects in $\cC$.
\end{prop}
\begin{proof}
  Analogous to \cite[Thm.~4.3]{MR2854172}, see also \cite[Thm.~6.3]{Lowrey}.
\end{proof}
\begin{defn}\label{def:MotivicHallAlgebra}
  The \emph{motivic Hall algebra} of $\cC$ is $H(\cC) \coloneq (K(\St/\uC),*,\ID_0)$.
\end{defn}

Taking Cartesian products make $H(\cC)$ into an algebra over $K(\St/\C)$.
Elements of the Hall algebra are naturally graded by the numerical Grothendieck group $N(\X)$, where an element $[f \colon X \to \uC]$ is \emph{homogeneous of degree $\alpha$} if $f$ factors through the substack $\uC_{\alpha}$.
The $K(\St/\C)$-algebra structure respects this grading.

\subsection{The integration map}
\label{Section5_IntegrationMap}
Let $\LL = [\A^1_\C]$.
Consider the map of commutative rings $K(\Var/\C)[\LL^{-1}, (1 + \ldots + \LL^{n})^{-1} \colon n \ge 1] \to K(\St/\C)$, and recall that $H(\cC)$ is a $K(\St/\C)$-module.
Note that the element $(\LL - 1)^{-1} = [B\C^*]$ lies in the latter ring $K(\St/\C)$ but \emph{not} in the former.
We define the subalgebra of \emph{regular elements} $H_{\reg}(\cC) \subset H(\cC)$ as the $K(\Var/\C)[\LL^{-1}, (1 + \ldots + \LL^{n})^{-1} \colon n \ge 1]$-module generated by those elements $[Z \to \uC]$ of $H(\cC)$ for which $Z$ is a variety.
\begin{prop}
  The submodule $H_{\reg}(\cC)$ is closed under the Hall algebra product.
  
  The quotient $\Hsc(\cC) = H_{\reg}(\cC)/(\LL-1)H_{\reg}(\cC)$ is commutative and has a Poisson bracket given by
  \begin{equation}\label{eq:HallPoissonBracket}
  \{f,g\} = \frac{f*g-g*f}{\LL-1}.  
  \end{equation}
\end{prop}
The commutative Poisson algebra $\Hsc(\cC)$ is the \emph{semi-classical} Hall algebra of $\cC$.
\begin{proof}
  This analogue of \cite[Thm.~5.1]{MR2854172} holds with the same proof.
\end{proof}
We now fix a $\sigma \in \{\pm 1\}$, where $+1$ corresponds to the topological Euler characteristic $e$ and $-1$ corresponds to the Behrend-weighted Euler characteristic $e_B$.
Recall that $\chi$ is the Euler form on $N(\X)$.
The \emph{Poisson torus} $\Q[N(\X)]$ is the $\Q$-vector space with basis $\{t^{\alpha} \mid \alpha \in N(\X)\}$ and commutative product $t^{\alpha_1} \star t^{\alpha_2} = \sigma^{\chi(\alpha_1,\alpha_2)}t^{\alpha_1+\alpha_2}$.
It is a Poisson algebra when endowed with the Poisson bracket defined by
\begin{equation}\label{eq:QuantumTorusPoissonBracket}
\{t^{\alpha}, t^{\beta}\} = \sigma^{\chi(\alpha,\beta)}\chi(\alpha,\beta)t^{\alpha + \beta}.
\end{equation}

There is a homomorphism of Poisson algebras $I_{\sigma} \colon \Hsc(\cC) \to \Q[N(\X)]$ which is uniquely determined by the following condition.
Let $Y$ be a variety and let $Y \to \uC_{\alpha}$ be a morphism.
For $\sigma = 1$, we have
\begin{equation}\label{eq:IntegrationMapEuler}
    I_1 [Y \to \uC_\alpha \subset \uC] = e(Y)t^{\alpha}  \in \Q[N(\X)].
\end{equation}
For $\sigma = -1$, we have
\begin{equation}\label{eq:IntegrationMapKai}
    I_{-1} [Y \to \uC_\alpha \subset \uC] = e_B(Y \to \uC)t^{\alpha} \in \Q[N(\X)],
\end{equation}
where we define
\begin{equation}\label{eq:BehrendOfAHallAlgebraElement}
	e_{B}\left( Y \stackrel{s}{\to} \uC \right) = \sum_{k \in \Z} k \cdot e\left( (\nu \circ s)^{-1}(n) \right)
\end{equation}
and where $\nu\colon \uC \to \Z$ is Behrend's constructible function \cite{MR2600874}.
\begin{thrm}\label{thm:IntegrationMap}
  The integration map $I_{\sigma}$ is a map of Poisson algebras.
\end{thrm}
\begin{proof}
  Bridgeland proves this in the case of $\sigma = 1$, and conditionally on a certain assumption on the Behrend function in the case of $\sigma = -1$ \cite[Thm.~5.2]{MR2854172}.
  As shown by Toda in \cite[Thm.~2.8]{2016arXiv160107519T}, this assumption holds true in our setting.
\end{proof}
We work with $\sigma = -1$ for the rest of this paper; however, see Section~\ref{sec1_QuasiProjective_Case} for a discussion of the case $\sigma = 1$.
\begin{rmk}\label{rmk:Integration_Poisson_vs_Lie}
  If we equip $\Q[N(\X)]$ with the \emph{naive} product $t^{\alpha_1}t^{\alpha_2} = t^{\alpha_1+\alpha_2}$ then $\Q[N(\X)]$ is only a \emph{Lie} algebra and $I_{-1}$ is only a morphism of Lie algebras, not of Poisson algebras; this issue only occurs for the Behrend weighted case $\sigma = -1$.

  Henceforth we equip $\Q[N(\X)]$ with the \emph{naive} product.
  The wall-crossing only involves the bracket $\{-,-\}$, so this distinction is of no consequence to the validity of our arguments.
  However, it does mean that the Leibniz rule \emph{does not} hold in $\Q[N(\X)]$.
  In keeping with the literature, we opt to write our variables as $t^{\alpha} = z^{\beta}q^{c}$.
\end{rmk}

\subsection{The graded Hall algebra}
It is convenient to be able to talk about infinite graded sums in the Hall algebra, and we therefore consider the following variant.
\begin{defn}\label{def:GradedHallPreAlgebra}
  The \emph{graded Hall pre-algebra} $H_{\gr}(\cC)$ is the $\Q$-vector space generated by symbols $[X \to \uC]$, where $X$ is an Artin stack \emph{locally} of finite type over $\C$ with affine geometric stabilizers, but such that the restriction of $X$ to $\uC_{\alpha}$ is of finite type for each $\alpha \in N(\X)$.
  We impose the same relations as before.
\end{defn}
\begin{rmk}
  As the name suggests, the graded Hall pre-algebra is not quite an algebra, for the same reason that the set of all formal expressions $\sum_{n \in \Z} a_{n}q^{n}$ is not a ring.
  Indeed, the product of two elements in $H_{\gr}(\cC)$ may not lie in $H_{\gr}(\cC)$, since the product may not be of finite type over each $\uC_{\alpha}$.

  On the other hand, suppose $C = \sum_\alpha C_\alpha$, $D = \sum_\alpha D_\alpha$ are two elements of $H_{\gr}(\cC)$ with $C_\alpha$ and $D_\alpha$ homogeneous of degree $\alpha$.
  If we assume that for every $\alpha \in N(\X)$,
  \begin{equation*}
      \left\{ \alpha_1 + \alpha_2 = \alpha \mid C_{\alpha_1} \neq 0 \neq D_{\alpha_2} \right\}
  \end{equation*}
  is a finite set, then the product of $C$ and $D$ exists in $H_{\gr}(\cC)$.
\end{rmk}
One may define graded versions of the regular subalgebra $H_{\gr,\reg}(\cC) \subset H_{\gr}(\cC)$, and the semi-classical quotient $H_{\gr,\Sc}(\cC)$.
The latter comes equipped with a partially defined Lie bracket and integration morphisms $I_{\sigma} \colon H_{\gr,\Sc}(\cC) \to \Q\{N(\X)\}$, where $\Q\{N(\X)\}$ is the group of all formal expressions $\sum_{\alpha \in N(\X)} n_{\alpha}t^{\alpha}$ with $n_{\alpha} \in \Q$.

By Theorem \ref{thm:IntegrationMap}, the map $I_{\sigma}$ preserves the Lie bracket between any two elements for which it is defined.
\begin{rmk}
  We point out that in this paper, we are only concerned with objects of rank $0$ and $-1$.
  Indeed, we only make use of the (graded versions of) the Hall algebra $H(\Coh_{\le 1}(\X))$ and the Lie-bimodule structure of $H(\A_{\rk = -1}) \subset H(\Coh^{\flat}(\X))$ over $H(\Coh_{\le 1}(\X))$ given by taking Lie brackets with elements in $H(\Coh_{\leq 1}(\X))$.
\end{rmk}

\section{Wall-crossing}
In this section we establish a general wall-crossing formula for $(\T,\F)$-pairs.
As an application, we prove the DT/PT correspondence for hard Lefschetz orbifolds.

We always work in the Hall algebras associated to the abelian category $\Coh^{\flat}(\X)$, and suppress this category from the notation.
For example, $H_{\gr}$ denotes $H_{\gr}(\Coh^{\flat}(\X))$.

We call a torsion pair $(\T,\F)$ on $\Coh_{\leq 1}(\X)$ \emph{numerical} if $[T] = [F]$ in $N(\X)$, for $T \in \T, F \in \F$, implies $T = F = 0$.
Any torsion pair induced from a stability condition on $N_{\le 1}(\X)$ is numerical.
In particular, all the torsion pairs we consider are numerical.
\begin{lem}
  \label{thm:TFPairsHaveSimpleAutomorphisms}
  Let $(\T,\F)$ be a numerical torsion pair on $\Coh_{\le 1}(\X)$.
  Let $E$ be a $(\T,\F)$-pair in the sense of Definition \ref{def:TFpair}.
  Then $\Aut(E) = \C^*$.
\end{lem}
\begin{proof}
  Let $\phi \colon E \to E$ be an endomorphism of $E \in \cat{Pair}(\cat T, \cat F)$.
  If $\im \phi$ has rank 0, then by definition $\im \phi \in \cat F \cap \cat T = 0$.
  If $\im \phi$ has rank -1, then $\ker \phi \in \cat F$, and $\cok \phi \in \cat T$.
  But since $[\cok \phi] = [\ker \phi]$ in $N(\X)$, we have $\ker \phi = \cok \phi = 0$.
  Thus every non-zero endomorphism of $E$ is an automorphism, and it follows that $\Aut(E) = \C^{*}$.
\end{proof}

\begin{cor}\label{cor:PairsDefineRegularElements}
  Let $(\T,\F)$ be a numerical, open torsion pair on $\Coh_{\le 1}(\X)$, and let $\M \subset \st{Pair}(\T,\F)$ be an open and finite type substack.
  Then $(\LL - 1)[\M] \in H_{\reg}(\cC)$.
\end{cor}
\begin{proof}
  By \cite[Thm.~5.1.5]{MR2007376}, the moduli stack of pairs $\uPair(\T,\F)$ has a coarse moduli space, over which it is a $\C^*$-gerbe by Lemma \ref{thm:TFPairsHaveSimpleAutomorphisms}.
  The result follows.
\end{proof}

\begin{defn}\label{defdefel}
  We say that a full subcategory $\cC \subset \Coh^{\flat}(\X)$ \emph{defines an element} in the graded Hall pre-algebra $H_{\gr}$ if it is an open subcategory and $\uC_{\alpha}$ is of finite type for every $\alpha \in N(\X)$.
\end{defn}
\begin{rmk}
  If a subcategory $\cC \subset \Coh^{\flat}(\X)$ defines an element of $H_{\gr}$, then we omit the natural inclusion and simply write $[\uC]$ for the corresponding element.
\end{rmk}
Let $(\T_{\pm}, \F_{\pm})$ be two open, numerical torsion pairs on $\Coh_{\le 1}(\X)$ with $T_{+} \subset \T_{-}$.
We think of these torsion pairs as lying on either side of a wall, and we write $\W = \T_{-} \cap \F_{+}$ for the category of objects that change from being torsion to being free when the wall is crossed.
We obtain an open torsion triple $\Coh_{\leq 1}(\X) = \langle \T_{+}, \cat{W}, \F_{-} \rangle$.
\begin{lem}\label{thm:equivalenceOfFiniteTypeConditions}
  Assume that $\Pair(\T_{-},\F_{-})$, $\Pair(\T_{+},\F_{+})$, and $\W$ define elements of $H_{\gr}(\Coh^{\flat}(\X))$, and let $\alpha \in N(\A) \in \Z \oplus N_{\leq 1}(\X)$.
  The following are equivalent:
  \begin{enumerate}
  \item The stack $\uPair(\T_{+}, \F_{-})_{\alpha}$ is of finite type.
  \item The stack $(\uPair(\T_{+}, \F_{+}) * \uW)_{\alpha}$ is of finite type.
  \item The stack $(\uW * \uPair(\T_{-},\F_{-}))_{\alpha}$ is of finite type.
  \end{enumerate}
\end{lem}
\begin{proof}
  By Proposition \ref{prop:InducedTorsionTriple}, we have the following torsion tuple decompositions of $\A$:
  \[
    \A = \langle \T_{+}, V(\T_{+}, \F_{-}), \F_{-} \rangle = \langle \T_{-}, \cat W, V(\T_{-}, \F_{-}), \F_{-} \rangle = \langle \T_{+}, V(\T_{+}, \F_{+}), \cat W, \F_{+} \rangle.
  \]
  Restricting to $V(\T_{+}, \F_{-}) \cap \A_{\rk = -1} = \Pair(\T_{+}, \F_{-})$, we find that every $(\cat T_{+},\cat F_{-})$-pair $E$ admits unique decompositions
  \begin{gather*}
    0 \to W_{-} \to E \to E_{-} \to 0 \\
    0 \to E_{+} \to E \to W_{+} \to 0
  \end{gather*}
  in $\A$, with $W_{\pm} \in \cat W$, $E_{\pm} \in \Pair(\T_{\pm}, \F_{\pm})$.
  Conversely every exact sequence as above with $W_{\pm} \in \cat W$, $E_{\pm} \in \Pair(\T_{\pm}, \F_{\pm})$ defines an $E \in \Pair(\T_{+}, \F_{-})$.

  Thus we have natural geometrically bijective morphisms from $(\uPair(\T_{+}, \F_{+}) * \uW)_{\alpha}$ and $(\uW * \uPair(\T_{-},\F_{-}))_{\alpha}$ to $\uPair(\T_{-},\F_{+})$, proving the implications (2) $\Rightarrow$ (1) and (3) $\Rightarrow$ (1).

  Conversely, if $(\uPair(\T_{+}, \F_{+}) * \uW)_{\alpha}$ is not of finite type, then it is a countably infinite union of finite type stacks, which cannot be in geometric bijection with a finite type stack.
  Thus (1) $\Rightarrow$ (2), and in the same way (1) $\Rightarrow$ (3).
\end{proof}
\begin{lem}
  Assume that $\cat W$, $\Pair(\T_{-}, \F_{-})$, $\Pair(\T_{+},\F_{+})$ and $\Pair(\T_{+},\F_{-})$ each define an element of the graded Hall pre-algebra.
  Then we have in $H_{\gr}$
  \begin{equation}\label{eq:WCformula}
    [\uW] * [\uPair(\T_{-}, \F_{-})] = [\uPair(\T_{+},\F_{-})] = [\uPair(\T_{+}, \F_{+})] * [\uW].
  \end{equation}
\end{lem}
\begin{proof}
  See the proof of Lemma \ref{thm:equivalenceOfFiniteTypeConditions}, and argue as in \cite[Lem.~4.1]{MR2813335}.
\end{proof}

\subsection{The no-poles theorem}
We now impose a further smallness assumption on the subcategory $\cat{W}$, taken from \cite{1612.00372}, which may be thought of as saying that the torsion pairs $(\T_{\pm},\F_{\pm})$ are close enough for the wall $\cat{W} = \T_{-} \cap \F_{+}$ to be crossed.
\begin{defn}
  \label{dfn:DecompositionallyFinite}
  A full subcategory $\cat{W} \subset \Coh_{\leq 1}(\X)$ is \emph{log-able} if:
  \begin{itemize}
  \item $\cat{W}$ is closed under direct sums and summands.
  \item $\W$ defines an element of $H_{\gr}(\cC)$.
  \item if $\alpha \in N(\X)$, there are only finitely many ways of writing $\alpha = \alpha_{1} + \cdots + \alpha_{n}$, with each $\alpha_{i}$ the class of a non-zero element in $\cat{W}$.
  \end{itemize}
\end{defn}
\begin{thrm}
  \label{thm:NoPoles}
  If $\cat{W}$ is log-able, then
  \begin{equation*}
    (\LL - 1)\log ([\uW]) \in H_{\gr,\reg}(\Coh^{\flat}(\X)).
  \end{equation*}
\end{thrm}
\begin{proof}
  This follows as in \cite[Sec.~6]{MR2813335}.
  We use the fact that the Behrend function identities hold in $\uA$, by \cite[Thm.~2.6]{2016arXiv160107519T}, as well as the fact that the ``no poles'' statement analogous to that of \cite[Thm.~8.7]{MR2354988} holds, by \cite[Thms.~4 \& 5]{1612.00372}.
  The fact that $\uA$ satisfies the axioms used in \cite{1612.00372} is shown in Appendix \ref{sec:BRaxioms}.
\end{proof}

\subsection{The numerical wall-crossing formula}
We apply the integration map.
\begin{defn}
  \label{defn:wallCrossingMaterial}
  Let $(\T_{\pm}, \F_{\pm})$ be open torsion pairs on $\Coh_{\le 1}(\X)$ with $\T_{+} \subset \T_{-}$, and let $\W = \T_{-} \cap \F_{+}$.
  We say that these torsion pairs are \emph{wall-crossing material} if
  \begin{enumerate}
  \item $\cat{W}$ is log-able
  \item the categories $\Pair(\T_+,\F_+)$, $\Pair(\T_-,\F_-)$ and $\Pair(\T_{+}, \F_{-})$ define elements of the graded Hall algebra.
  \end{enumerate}
\end{defn}
\begin{thrm}\label{thm:WallCrossingFormula}
  Let $(\T_\pm,\F_\pm)$ be open torsion pairs on $\Coh_{\leq 1}(\X)$, with $\T_{+} \subset \T_{-}$, which are wall-crossing material.
  Then $w \coloneq I \left( (\LL-1) \log [\uW] \right)$ is well defined, and
  \begin{equation}\label{eq:NumericalWallCrossingFormula}
    I \left( (\LL-1) [\uPair(\T_{+},\F_{+})]\right) =  \exp\left( \left\{ w, - \right\} \right) I \left( (\LL-1) [\uPair(\T_{-},\F_{-})] \right).
  \end{equation}
\end{thrm}
\begin{proof}
  Since $\cat{W}$ is log-able $(\LL-1)\log([\uW]) \in H_{\gr,\reg}(\Coh^{\flat}(\X))$ by Theorem \ref{thm:NoPoles}.
  By Corollary \ref{cor:PairsDefineRegularElements}, we have $(\LL-1)\uPair(\T_{\pm},\F_{\pm}) \in H_{\gr,\reg}(\Coh^{\flat}(\X))$ as well.
  The result then follows by the arguments of \cite[Cor.~6.4]{MR2813335} and equation \eqref{eq:WCformula}.
\end{proof}

\subsection{The DT/PT correspondence}
As a first application of the wall-crossing formula, we prove the orbifold DT/PT correspondence.
Recall that stable pairs on $\X$ are precisely $(\T_{PT},\F_{PT})$-pairs, where $\T_{PT} = \Coh_0(\X)$ and $\F_{PT} = \Coh_1(\X)$.

\begin{lem}\label{lem:Hygiene_for_DTPT_correspondence}
  Let $\T_{DT} = 0, \W = \Coh_0(\X)$, and $\F_{PT} = \Coh_1(\X)$.
  Then $(\T_{DT},\W,\F_{PT})$ is an open numerical torsion triple on $\Coh_{\leq 1}(\X)$ that is wall-crossing material.
\end{lem}
\begin{proof}
  Clearly $\W$ is log-able.

  A $(\T_{DT},\F_{DT})$-pair is an object of the form $I[1]$, where $I$ is a torsion free sheaf of rank $1$.
  Fixing the numerical class of the pair, the stack of such is an open substack of $\Mum_{\X}$, which is moreover of finite type.
  Hence, $\Pair(\T_{DT}, \F_{DT})$ defines an element of $H_{\gr}(\Coh^{\flat}(\X))$.

  If $\X$ is a variety, it is well-known that $\uPair(\T_{PT},\F_{PT})_{\alpha}$ is an open substack of finite type for every $\alpha \in N_{\le 1}(\X)$; for a proof in our setting, combine part~\ref{enum:PNuOpenFiniteType} of Proposition~\ref{thm:BoundednessOfRank1SigmaSemistables} and Lemma~\ref{thm:sigmaInfinityIsPT}.
  Hence, $\Pair(\T_{PT},\F_{PT})$ defines an element of $H_{\gr}(\Coh^{\flat}(\X))$.
  Moreover, it is shown there that for each $\beta \in N_{1}(\X)$, the set $\{c \in N_{0}(\X) \mid \uPair(\T_{PT}, \F_{PT})_{(\beta,c)} \not= \varnothing\}$ is $\deg$-bounded.
  Consequently, it follows that $[\W] * [\uPair(\T_{PT}, \F_{PT})]$ defines an element of $H_{\gr}(\Coh^{\flat}(\X))$.
\end{proof}

For a class $(\beta,c) \in N_{\le 1}(\X)$, we write $PT(\X)_{(\beta,c)}$ for the Behrend-weighted Euler characteristic of the corresponding coarse moduli space.
In terms of the integration morphism of \eqref{eq:IntegrationMapKai}, we have
\begin{equation}
  I\bigl((\LL - 1) [\uPair(\T_{PT},\F_{PT})]_{(\beta,c)}\bigr) = PT(\X)_{(\beta,c)}t^{(-1,\beta,c)}
\end{equation}
in the Poisson torus $\Q[N(\X)]$.
We collect these invariants in a generating function
\begin{equation}\label{eq:PT_Generating_Function_PTX}
  PT(\X)_\beta = \sum_{c \in N_0(\X)} PT(\X)_{(\beta,c)}q^{c}
\end{equation}
Similarly, there is a generating function for the Donaldson--Thomas invariants.

We prove the orbifold DT/PT correspondence for multi-regular curve classes.
\begin{thrm}\label{thm:DTPT}
  Let $\X$ be a CY3 orbifold satisfying the hard Lefschetz condition, and let $\beta \in N_{1,\mr}(\X)$.
  Then we have
  \begin{equation}\label{eq:DTPT_Orbifold_Formula}
    PT(\X)_\beta = \frac{DT(\X)_\beta}{DT(\X)_0}
  \end{equation}
  as generating series in $\Z[N_0(\X)]_{\deg}$.
\end{thrm}
\begin{proof}
  We apply the numerical wall-crossing formula of Theorem \ref{thm:WallCrossingFormula} to the open numerical torsion triple $\T_+ = \T_{DT} = 0$, $\W = \Coh_0(\X)$, and $\F_- = \F_{PT} = \Coh_1(\X)$.
  The triple $(\T_{+}, \W, \F_{-})$ is wall-crossing material by Lemma \ref{lem:Hygiene_for_DTPT_correspondence}.

  We compute both sides of the wall-crossing formula.
  The left hand side of \eqref{eq:NumericalWallCrossingFormula} yields $I\left((\LL-1)[\uPair(\T_{DT},\F_{DT}]_{\beta}\right) = DT(\X)_\beta z^\beta t^{-[\hO_{\X}]}$.
  Define the element $w = I((\LL-1)\log [\uW])$, which lies in $H_{\gr,\reg}(\cC)$.
  The right hand side of \eqref{eq:NumericalWallCrossingFormula} yields
  \begin{equation}
    \exp(\{w,-\})I\left((\LL-1)[\uPair(\T_{PT},\F_{PT}]_{\beta}\right) = \exp(\{w,-\}) PT(\X)_\beta z^\beta t^{-[\hO_{\X}]}.
  \end{equation}

  Let now $c \in N_{0}(\X)$.
  Applying the McKay homomorphism, we get $\Psi(c) \in N_{\le 1}(Y)$.
  Since $\beta$ is multi-regular, we also have $\Psi(\beta,c') \in N_{\le 1}(Y)$ for every $c' \in N_{0}(\X)$.
  The Euler pairing on $Y$ is trivial on the subspace $N_{\le 1}(Y)$, and so
  \begin{equation*}
    \chi(c,(\beta,c')) = \chi(\Psi(c),\Psi(\beta,c')) = 0.
  \end{equation*}
  We can write $w = \sum_{c \in N_{0}(\X)} w_{c}q^{c}$, and it follows that $\{w, z^{\beta}q^{c'}\} = 0$.
  But then
  \[
    \exp(\{w,-\})PT(\X)_{\beta}z^{\beta}t^{-[\hO_{\X}]} = PT(\X)_{\beta}z^{\beta}\exp(\{w,-\})t^{-[\hO_{\X}]}.
  \]
  Combining the left and right hand sides of equation~\eqref{eq:NumericalWallCrossingFormula} now yields
  \[
    \frac{DT(\X)_{\beta}}{PT(\X)_{\beta}} = t^{[\hO_{\X}]}\exp(\{w,-\})t^{-[\hO_{\X}]}
  \]
  for all $\beta \in N_{1,\mr}(\X)$.
  Choosing $\beta = 0$, we recall that $PT(\X)_0 = 1$, and hence
  \[
    \frac{DT(\X)_{\beta}}{PT(\X)_{\beta}} = \frac{DT(\X)_{0}}{PT(\X)_{0}} = DT(\X)_{0},
  \]
  because $E = \hO_{\X}[1]$ is the only stable pair with $\beta_{E} = 0$.
\end{proof}



\section{Rationality of stable pair invariants}
\label{sec:rationalityOfStablePairs}
In this section we prove the rationality of the series $PT(\X)_\beta$ in Theorem~\ref{thm:PTIsRational}, and a certain symmetry of $PT(\X)$ in Proposition~\ref{thm:dualityResultForPT}.

Let $\delta \in \R$.
Define a torsion pair $(\cat T_{\nu,\delta}, \cat F_{\nu,\delta})$ on $\Coh_{\le 1}(\X)$ by
\begin{equation*}
  \begin{split}
    \T_{\nu,\delta} &\coloneq \left\{ T \in \Coh_{\le 1}(\X) \mid \nu_{-}(T) \ge \delta \right\} = \{ T \in \Coh_{\le 1}(\X) \mid T \onto Q \not= 0 \Rightarrow \nu(Q) \ge \delta\} \\
    \F_{\nu,\delta} &\coloneq \left\{ F \in \Coh_{\le 1}(\X) \mid \nu_{+}(F) < \delta \right\} = \{ F \in \Coh_{\le 1}(\X) \mid 0 \not= S \into F \Rightarrow \nu(S) < \delta \}
  \end{split} 
\end{equation*}
We write $\cat{P}_{\nu,\delta}$ for the category of $(\T_{\nu,\delta},\F_{\nu,\delta})$-pairs in the sense of Definition~\ref{def:TFpair}.

Recall that we have fixed a self-dual generating bundle $V$ on $\X$ and an ample line bundle $A$ on $X$, that $p_F(k) = l(\beta_F)k + \deg(F)$ denotes the modified Hilbert polynomial of $F \in \Coh_{\leq 1}(\X)$, and that $F(k) = F \otimes A^{\otimes k}$; see Section~\ref{sec2_Modified_Hilbert}.

We write $\cC \coloneq \Coh^{\flat}(\X)$ and work in the associated Hall algebra $H_{\gr}(\cC)$ throughout.

\subsection{Openness and finiteness results of Nironi-semistable sheaves}
Let $I \subset \R \cup \{+\infty\}$ be an interval.
Recall that $\M_{\nu}(I)$ denotes the full subcategory of sheaves $F \in \Coh_{\leq 1}(\X)$ such that the slopes of all semistable factors in the Harder--Narasimhan filtration of $F$ lie in $I$.
The corresponding moduli stack is denoted by $\uM_{\nu}(I)$.
\begin{prop}\label{thm:SigmaDeltaHygiene}  
  Let $\delta \in \R$.
  \begin{enumerate}
  \item \label{thm:sigmaTorsionPairIsOpen}
    The torsion pair $(\T_{\nu,\delta},\F_{\nu,\delta})$ is open.
  \item \label{enum:DecompositionallyFinite}
    For any bounded interval $I \subset \R$, the stack $\uM_{\nu}(I)$ is open in $\cCoh(\X)$, and the category $\M_{\nu}(I)$ is log-able.
  \item \label{enum:NuStableFinitelyManyC}
    For $\beta \in N_{1}(\X)$, let
    \begin{equation*}
      {\cal L}_\beta = \left\{c \in N_0(\X) \bigm| \uM^{\ss}_{\nu}(\beta,c) \neq \emptyset \right\} \subset N_{0}(\X).
    \end{equation*}
    The image of ${\cal L}_\beta$ in $N_0(\X)/\Z(A \cdot \beta)$ is finite.
  \end{enumerate}
\end{prop}
\begin{proof}
  The openness statements and the fact that $\M_{\nu}(I)$ defines an element of $H_{\gr}(\cC)$ follow from Theorem~\ref{thm:NironisBoundednessTheorem}.
  Let $(\beta,c) \in N_{\leq 1}(\X)$ and suppose we can decompose $(\beta, c) = (\beta',c') + (\beta'',c'')$ such that
  \begin{equation*}
    \M^{\ss}_{\nu}(\beta',c') \not= \varnothing \not= \M^{\ss}_{\nu}(\beta'',c'')
  \end{equation*}
  and $\nu(\beta',c') = \delta = \nu(\beta'',c'')$.
  By Lemma \ref{lem:Effective_Cone_Convex}, there are only finitely many effective classes $\beta', \beta'' \geq 0 $ such that $\beta = \beta'+\beta''$.
  And given $\beta'$, Theorem \ref{thm:NironisBoundednessTheorem} shows that there are only finitely many choices for $c'$ such that $\uM^{\ss}_{\nu}(\beta',c') \not= \varnothing$ and $\nu(\beta',c') = \delta$.
  Furthermore, $\M_{\nu}^{\ss}(\delta)$ is closed under direct sums and summands, hence is log-able.

  For the third claim, note that a sheaf $F \in \Coh_{\leq 1}(\X)$ such that $\beta_{F} = \beta$ satisfies $\nu(F) = \deg(F)/l(\beta) \in \frac{1}{l(\beta)}\Z$.
  Replacing $F$ by $F(-\lfloor \nu(F) \rfloor)$ if necessary, which does not change the image of $c_{F}$ in $N_{0}(\X)/\Z(A \cdot \beta)$, we may assume that $\nu(F) \in [0, 1)$.
  Thus the image of $c_{F}$ lies in the set.
  \begin{equation*}
    \bigcup_{a = 0}^{l(\beta)-1} \left\{c + \Z(A \cdot \beta) \bigm| \M^{\ss}_{\nu}(\beta,c) \neq \emptyset \text{ and } \nu(\beta,c) = \frac{a}{l(\beta)}\right\} \subset N_0(\X) / \Z(A \cdot \beta).
  \end{equation*}
  But each of the sets in this union is finite since the stack $\uM^{\ss}_{\nu}(a/l(\beta),\beta)$ is of finite type for every $a \in \Z$.
  This completes the proof.
\end{proof}

\subsection{Openness and finiteness for \texorpdfstring{$(\T_{\nu,\delta},\F_{\nu,\delta})$}{Nu-Delta}-pairs}
In this section we show a similar openness and boundedness result for the moduli stacks of $(\T_{\nu,\delta},\F_{\nu,\delta})$-pairs; this is Proposition~\ref{thm:BoundednessOfRank1SigmaSemistables}.
We first collect a number of lemmas.

Recall that given a polynomial $p \in \Z[k]$, there is a projective moduli scheme $\Quot_\X(\hO_{\X},p)$ parametrising quotients $\hO_{\X} \onto F$ with $p_{F} = p$.
If $(\beta,c) \in N_{\le 1}(\X)$,
\begin{equation*}
  \Quot_{\X}(\hO_{\X})_{(\beta,c)} \subset \Quot_{\X}(\hO_{\X}, p_{(\beta,c)})
\end{equation*}
denotes the component parametrising quotients of numerical class $[F] = (\beta,c)$.
\begin{lem}
  \label{thm:LinearTermInHilbertPolynomialOfQuotSchemeIsBoundedBelow}
  Let $\beta \in N_{1}(\X)$.
  The set
  \begin{equation*}
    \bigcup_{\beta' \leq \beta}\{c \in N_0(\X) \mid \Quot_\X(\hO_{\X})_{(\beta',c)} \neq \varnothing\}
  \end{equation*}
  is $\deg$-bounded.
\end{lem}
\begin{proof}
  By Lemma \ref{lem:Effective_Cone_Convex}, it is enough to prove that the set
  \[
    Q_{\beta} := \{c \in N_0(\X) \mid \Quot_\X(\hO_{\X})_{(\beta,c)} \neq \varnothing\}.
  \]
  is $\deg$-bounded for every $\beta \in N_1(\X)$.
  Given $r \in \R$, the projectivity of the Quot scheme implies that the subscheme
  \begin{equation*}
    \bigcup_{c \in \deg^{-1}(r)} \Quot_\X(\hO_{\X})_{(\beta,c)} \subseteq \Quot_\X(\hO_{\X}, l(\beta)k + \deg(\beta,0) + r)
  \end{equation*}
  is projective, and so $Q_\beta \cap \deg^{-1}(r)$ is a finite set.

  Let $d = \deg([\hO_{\X,x}])$, where $x \in \X$ is a non-stacky point.
  Since $\deg(Q_\beta) \subseteq \Z$, and so is discrete, we have for any $e \in \Z$ that the scheme
  \begin{equation*}
    H_e = \bigcup_{c \in \deg^{-1}([e,e+d])} \Quot_\X(\hO_{\X})_{(\beta, c)}
  \end{equation*}
  is projective, which means that $Q_\beta \cap \deg^{-1}([e,e+d])$ is finite.
  By adding on floating points, see \eg~\cite[Lem.~3.10]{toda_limit_2009}, we deduce that $\dim H_{e-d} + 3 \le \dim H_e$ for any $e \in \Z$.
  We conclude that $H_e = \varnothing$ for $e \ll 0$, and thus that $Q_{\beta}$ is $\deg$-bounded.
\end{proof}

Recall the shifted derived dualising functor $\D(-) = \lRHom(-,\hO_\X)[2]$.
\begin{lem}\label{lem:Modified_Hilbert_Polynomial_Dualising}
  If $F \in \Coh_{\leq 1}(\X)$, then $p_{\D(F)}(k) = -p_{F}(-k)$ and $\nu(\D(F)) = -\nu(F)$.
\end{lem}
\begin{proof}
  This is straightforward, but requires the assumption that the generating vector bundle $V$ in the definition of the modified Hilbert polynomial be self-dual.
\end{proof}
The following is a duality result for the moduli stack $\jP_{\nu,\delta}$ of $(\T_{\nu,\delta},\F_{\nu,\delta})$-pairs.
\begin{lem}
  \label{thm:DualisingPreservesSigmaStables}
  Let $\delta \in \R \setminus \Q$.
  Then $\D(\cat P_{\nu,\delta}) = \cat P_{\nu, -\delta}$.
\end{lem}
\begin{proof}
  As $\delta \not\in \Q$ and $\nu(F) \in \Q$ for any $F \in \Coh_{\le 1}(\X)$, the condition $\nu(F) \geq \delta$ holds if and only if $\nu(F) > \delta$.
  It follows that $\D(\F_{\nu,\delta}) = \T_{\nu,-\delta} \cap \Coh_{1}(\X)$.
  
  Let $E \in \cat P_{\nu,\delta}$ be a pair.
  Since $\cat P_{\nu,\delta} \subset \cat A = \langle \hO_{\X}[1], \Coh_0(\X), \Coh_1(\X) \rangle_{\text{ex}}$, we have
  \[
    \D(E) \in \D(\cat A) = \langle \hO_{\X}[1], \Coh_1(\X), \Coh_0(\X)[-1] \rangle_{\text{ex}}.
  \]
  Since $\Coh_0(\X) \subset \T_{\nu,\delta}$, we have $E \in \Coh_{0}(\X)^{\perp}$ and so $\D(E) \in ^{\perp}\!\!\Coh_{0}(\X)[-1]$
  This implies $\D(E) \in \langle \hO_\X[1], \Coh_1(\X) \rangle_{\text{ex}}$ and hence $\D(E) \in \A$.

  Since $E \in (\T_{\nu,\delta} \cap \Coh_{1}(\X))^{\perp}$ we have $\D(E) \in ^{\perp}\!\F_{\nu,-\delta}$, and $E \in {^{\perp}\F_{\nu,\delta} \cap ^{\perp}\!\!\Coh_0(\X)[-1]}$ implies $\D(E) \in (\T_{\nu,-\delta} \cap \Coh_1(\X))^{\perp} \cap \Coh_0(\X)^{\perp} = \T_{\nu,-\delta}^{\perp}$.
  Thus $\D(E) \in \cat P_{\nu,-\delta}$.
\end{proof}

\begin{lem}\label{thm:SetOfBoundedBelowGuysIsBounded}
  Let $\beta \in N_{1}(\X)$ and $\delta \in \R$.
  The set
  \begin{equation}
    \{c \in N_{0}(\X) \mid \M_\nu([\delta,\infty))_{(\beta,c)} \neq \varnothing\}.
  \end{equation}
  is $\deg$-bounded.
\end{lem}
\begin{proof}
  Let $r \in \R$. We have to show that the set
  \[
    \{c \in N_{0}(\X) \mid \deg(c) \le r \text{ and }\M_\nu([\delta,\infty))_{(\beta,c)} \neq \varnothing\}.
  \]
  is finite.
  So let $F$ be a pure 1-dimensional sheaf with $\beta_{F} = \beta$, $\deg(c_{F}) \le r$, and $\nu_{-}(F) \ge \delta$.
  By Lemma \ref{thm:BoundOnNuMax}, we have
  \begin{align*}
    \nu_{+}(F) &\le \deg(F) - [l(\beta)-1]\nu_{-}(F) \\
                  &= \deg(\beta,0) + \deg(c_{F}) - [l(\beta)-1]\nu_{-}(F) \\
                  &\le \deg(\beta,0) + r + [l(\beta)-1]\delta.
  \end{align*}
  By Theorem \ref{thm:NironisBoundednessTheorem}, there are then only finitely many possible values for $c_{F}$.
\end{proof}

\begin{prop}\label{thm:BoundednessOfRank1SigmaSemistables}
  Let $\delta \in \R$.
  \begin{enumerate}
  \item \label{enum:PNuFinitelyManyC}
    For any class $\beta \in N_1(\X)$, the set $\{c \in N_{0}(\X) \mid \uP_{\nu,\delta}(\beta,c) \neq \varnothing\}$ is finite.
  \item \label{enum:PNuOpenFiniteType}
    For any class $(\beta,c) \in N_{\leq 1}(\X)$, the moduli stack $\uP_{\nu,\delta}(\beta,c)$ is an open and finite type substack of $\Mum_{\X}$.
  \item \label{enum:NuFinitelyManyDecompositions}
    There are only finitely many ways of decomposing a class $(\beta,c) \in N_{\leq 1}(\X)$ as $(\beta,c) = (\beta',c') + (\beta'',c'')$ with both $\uP_{\nu,\delta}(\beta',c') \neq \varnothing$ and $\beta'' \in N_{1}^{\eff}(\X)$.
  \end{enumerate}
\end{prop}
\begin{proof}
  We may assume that $\delta \not\in \Q$ by replacing $\delta$ with $\delta - \epsilon$ for $0 < \epsilon \ll 1$ if necessary, since for a fixed $\beta$, this does not change the notion of $(\T_{\nu,\delta}, \F_{\nu,\delta})$-pairs of class $\le \beta$.

  For the first part, let $E$ be a $(\T_{\nu,\delta},\F_{\nu,\delta})$-pair of class $(-1,\beta,c)$.
  Note that $H^{-1}(E) = I_C$ is the ideal sheaf of an at most 1-dimensional closed substack $C \subset \X$.
  Let $H^0(E) = T$, so $T \in \T_{\nu,\delta}$.
  We have $\deg(E) = \deg(\hO_C) + \deg(T)$.
  The set of possible values for $c_{\hO_{C}}$ and $c_{T}$ are both $\deg$-bounded, by Lemmas~\ref{thm:LinearTermInHilbertPolynomialOfQuotSchemeIsBoundedBelow} and \ref{thm:SetOfBoundedBelowGuysIsBounded}, hence the set of possible values for $c_{E}$ is $\deg$-bounded.

  By Lemma \ref{thm:DualisingPreservesSigmaStables}, the set of possible values for $c_{\D(E)} = -c_{E}$ is $\deg$-bounded as well.
  It follows that the set of possible values for $c_{E}$ is in fact finite.

  For part (2), openness follows from Propositions \ref{thm:TFOpenImpliesPOpen} and \ref{thm:SigmaDeltaHygiene}.
  For the finite type claim, note that the above shows that there are finitely many choices for $[\hO_{C}]$ and $[T]$.
  For each such choice, the relevant moduli stacks (\ie~the stack of ideal sheaves $I_{C}$ of a given class and $\hM_{\nu}([\delta,\infty),[T])$) are of finite type.
  The stack of extensions of objects in $\hM_{\nu}([\delta,\infty),[T])$ by an object in $\Quot(\hO_{\X},[\hO_{C}])$ is of finite type, by Proposition~\ref{thm:extensionStackIsFiniteType}.
  This proves that the stack $\uP_{\nu,\delta}(\beta,c)$ is of finite type.

  The third claim follows from the first claim and Lemma~\ref{lem:Effective_Cone_Convex}.
\end{proof}

\subsection{\texorpdfstring{$\delta$}{Delta}-walls}
Let $\beta \in N_1(\X)$.
We now study the set of $\delta \in \R$ where the notion of $(\cat T_{\nu,\delta}, \cat F_{\nu,\delta})$-pair may change for objects of class $(-1,\beta',c')$ with $\beta' \leq \beta$.

Let $W_{\beta} = \frac{1}{l(\beta)!}\Z \subset \R$ be the set of \emph{walls for $\beta$}.
\begin{lem}\label{lem:Walls_for_Delta_Pairs}
  The notion of $(\T_{\nu,\delta}, \F_{\nu,\delta})$-pair of class $\beta' \leq \beta$ is locally constant for $\delta \in \R \setminus W_\beta$.
\end{lem}
\begin{proof}
  Let $E \in \A$ be of class $(-1,\beta,c)$.
  The object $E$ is a $\delta$-pair if and only if there are no surjections $E \onto F$ with $F \in \F_{\nu,\delta}$ and no injections $T \into E$ with $T \in \T_{\nu,\delta}$.
  By Lemma~\ref{thm:subAndQuotientObjectsHaveLowerD}, for any quotient $F$ and any subobject $T$ we must have $l(F), l(T) \le l(E) = l(\beta)$, and so $\nu(F), \nu(T) \in \frac{1}{l(\beta)!}\Z$.
  This proves the claim.
\end{proof}

\begin{prop}\label{cor:WallCrossingMaterial_Delta_Pairs}
  Let $\beta \in N_1(\X)$, let $\delta \in W_\beta$, and let $0 < \epsilon < \frac{1}{l(\beta)!}$.
  The triple
  \begin{equation*}
    (\T_{\nu, \delta+\epsilon}, \M_{\nu}([\delta-\epsilon,\delta + \epsilon)), \F_{\nu, \delta-\epsilon})
  \end{equation*}
  is a torsion triple that is wall-crossing material in the sense of Definition~\ref{defn:wallCrossingMaterial}.
\end{prop}
\begin{proof}
  It is a torsion triple by Lemma~\ref{lem:Walls_for_Delta_Pairs}, which is open by Proposition~\ref{thm:SigmaDeltaHygiene}.
  The category $\M_{\nu}([\delta-\epsilon, \delta + \epsilon))$ is log-able by part~\eqref{enum:DecompositionallyFinite} of Proposition~\ref{thm:SigmaDeltaHygiene}.
  
  Part (\ref*{enum:PNuOpenFiniteType}) of Proposition~\ref{thm:BoundednessOfRank1SigmaSemistables} states that the subcategories $\jP_{\nu, \delta \pm \epsilon}$ define elements of $H_{\gr}(\cC)$.
  Part (\ref*{enum:NuFinitelyManyDecompositions}) of Proposition~\ref{thm:BoundednessOfRank1SigmaSemistables} now proves that $(\T_{\nu, \delta+\epsilon}, \M^{\ss}_{\nu}(\delta), \F_{\nu, \delta-\epsilon})$ is wall-crossing material.
\end{proof}

\subsection{DT invariants}
\label{sec:StabDelta_DTinvariants}
We are now in a position to apply the integration map to define DT-type invariants counting Nironi-semistable sheaves and $(\cat T_{\nu, \delta}, \cat F_{\nu, \delta})$-pairs.

\subsubsection{Rank $0$}
Let $a \in \R$.
By Lemma~\ref{thm:SigmaDeltaHygiene}, the stack $\uM^{\ss}_{\nu}(a)$ defines an element $[\uM^{\ss}_{\nu}(a)] \in H_{\gr}(\cC)$, which is moreover log-able.
Thus we obtain an element
\begin{equation*}
  \eta_{\nu, a} \coloneq (\LL-1)\log ([\uM^{\ss}_{\nu}(a)]) \in H_{\gr,\reg}(\cC)
\end{equation*}
by Theorem~\ref{thm:NoPoles}.
Projecting this element to the semi-classical quotient $H_{\gr,\Sc}(\cC)$, we define DT-type invariants $J^{\nu}_{(\beta,c)} \in \Q$ by the formula
\begin{align}\label{eq:DTRank0Definition}
  \sum_{\nu(\beta,c) = a} J^{\nu}_{(\beta,c)} z^\beta q^c \coloneq I\left(\eta_{\nu,a} \right) \in \Q\{N(\X)\}.
\end{align}
These are the Joyce--Song orbifold-analogues of Toda's $N$-invariants; see \cite{todajams}.
These invariants count Nironi-semistable objects of slope $a$.

\subsubsection{Rank $-1$}
Let $(\beta, c) \in N_{\le 1}(\X)$, and let $\delta \in \R$.
By Corollary \ref{cor:PairsDefineRegularElements}, we obtain an element $(\LL - 1)[\uP_{\nu,\delta}(\beta, c)] \in H_{\reg}(\cC)$.
Projecting this element to the semi-classical quotient and applying the integration morphism, we define DT-type invariants
\begin{equation}\label{eq:DTRank1Definition}
  DT^{\nu,\delta}_{(\beta,c)}z^\beta q^ct^{-[\hO_{\X}]} \coloneq I\bigl((\LL-1)[\uP_{\nu,\delta}(\beta,c)] \bigr).
\end{equation}
Crucially, the $J$-invariants do not depend on $\delta$, whereas the invariants $DT^\delta$ \emph{do}.

\subsection{The limit as \texorpdfstring{$\delta \to \infty$}{delta goes to infinity}}
Fix a class $(\beta,c) \in N_{\leq 1}(\X)$.
We now show that the invariant $DT^{\nu, \delta}_{(\beta,c)}$ stabilises as $\delta$ tends to infinity, and that its limit equals the stable pair invariant $PT(\X)_{(\beta,c)}$.
By Proposition \ref{thm:BoundednessOfRank1SigmaSemistables}, we may define numbers
\begin{align*}
  M^{+}_{\le \beta} &= \max_{\substack{0 \le \beta' \le \beta \\ c \in N_{0}(\X)}}\{\deg(\beta',c) \mid \uP_{\nu,0}(\beta',c) \not= \varnothing\}, \\
  M^{-}_{\le \beta} &= \min_{\substack{0 \le \beta' \le \beta \\ c \in N_{0}(\X)}}\{\deg(\beta',c) \mid \uP_{\nu,0}(\beta',c) \not= \varnothing\}.
\end{align*}
\begin{lem}
  \label{thm:vanishingOfUP}
  Let $0 \le \beta' \le \beta$, and let $c \in N_{0}(\X)$.
  If $\delta \le 0$ and $\deg(\beta',c) > M^{+}_{\le \beta}$ then $\uP_{\nu,\delta}(\beta',c) = \varnothing$.
  If $\delta \ge 0$ and $\deg(\beta',c) < M^{-}_{\le \beta}$, then $\uP_{\nu,\delta}(\beta',c) = \varnothing$.
\end{lem}
\begin{proof}
  We only treat the claim for $\delta \le 0$, the other one is similar.
  
  The claim is clear for $\delta = 0$.
  We argue by contradiction, and assume the claim false.
  By Lemma~\ref{lem:Walls_for_Delta_Pairs}, there then exists a maximal $\delta \in W_{\beta}$ such that the claim holds for $\delta + \epsilon$ and fails for $\delta - \epsilon$.

  Thus there is a class $(\beta', c)$ with $\deg(\beta',c) > M^{+}_{\le \beta}$ such that $\uP_{\nu,\delta - \epsilon}(\beta',c) \not= \varnothing = \uP_{\nu,\delta + \epsilon}(\beta',c)$.
  This implies that we can find $(\beta'', c'')$ such that $\uP_{\nu,\delta + \epsilon}(\beta'', c'') \not= \varnothing$, and a $\nu$-semistable object $F$ of class $(\beta' - \beta'', c - c'')$ such that $\nu(F) = \delta \le 0$.
  Then
  \begin{gather*}
    \deg(\beta'',c'') \le M^{+}_{\le \beta}\\
    \deg(\beta' - \beta'', c-c'') = \nu(F)l(F) \le 0
  \end{gather*}
  and so $\deg(\beta',c) \le M^{+}_{\le \beta}$.
  This is a contradiction.
\end{proof}

\begin{lem}
  \label{lem:Delta_Pairs_are_Stable_Pairs_at_Infinity}
  \label{thm:sigmaInfinityIsPT}
  Let $E \in \A$ be an object of class $(-1,\beta, c)$, let
  \[
    \delta_{(\beta,c)} = \max\{0, \deg(\beta,c) - M^{-}_{\le \beta}\},
  \]
  and let $\delta > \delta_{(\beta,c)}$.
  Then the following are equivalent:
  \begin{enumerate}
  \item\label{enum:delta} $E$ is a $(\cat T_{\nu,\delta},\cat F_{\nu,\delta})$-pair.
  \item\label{enum:PT} $E$ is a stable pair.
  \end{enumerate}
\end{lem}
\begin{proof}
  We first show that condition (\ref*{enum:delta}) is independent of the precise value of $\delta > \delta_{(\beta,c)}$.
  By Lemma \ref{lem:Walls_for_Delta_Pairs}, it is enough to show that for any $\delta \in W_{\beta} \cap (\delta_{(\beta,c)},\infty)$ we have $\uP_{\nu,\delta-\epsilon}(\beta,c) = \uP_{\nu,\delta + \epsilon}(\beta,c)$ for $0 < \epsilon \ll 1$, \ie, every such wall-crossing is trivial.

  By the wall-crossing formula, the moduli stack $\uP_{\nu,\delta+\epsilon}(\beta,c)$ can only differ from $\uP_{\nu,\delta}(\beta,c) = \uP_{\nu,\delta-\epsilon}(\beta,c)$ if there exists an object in $\jP_{\nu,\delta + \epsilon}(\beta,c)$ that is destabilised when the wall $\delta$ is crossed.
  This happens precisely if we can decompose the class $(\beta,c) = (\beta',c') + (\beta'', c'')$ with $\nu(\beta'',c'') = \delta$ and
  \begin{equation*}
    \uP_{\nu,\delta}(\beta',c') \neq \varnothing \neq \M^{\ss}_{\nu}(\beta'',c'').
  \end{equation*}
  Suppose for a contradiction that there exists such a decomposition.
  Since $\delta \geq 0$, we have $\deg(\beta',c') \ge M_{\le \beta}^{-}$.
  Applying Lemma~\ref{thm:vanishingOfUP}, we find
  \begin{align*}
    \deg(\beta,c) &= \deg(\beta',c') + \deg(\beta'',c'') = \deg(\beta',c') + \delta l(\beta'') \\
       &> M_{\le \beta}^{-} + \delta_{\beta,c} > \deg(\beta,c),
  \end{align*}
  which is a contradiction.
  
  Suppose that $E$ is a PT pair, so $E = (\hO_{\X} \xrightarrow{s} F)$ with $\coker(s) \in \Coh_0(\X)$ and $F \in \Coh_1(\X)$.
  If $S \in \Coh_{\leq 1}(\X)$ is a subobject of $E$, then the inclusion factors through an inclusion $S \into F$ by Lemma~\ref{lem:SubobjectFactors}.
  Hence $\nu_{+}(S) \leq \nu_{+}(F)$.
  Taking $\delta \ge \nu_{+}(F)$, we find $S \in \F_{\nu,\delta}$.
  Furthermore, if $Q \in \Coh_{\le 1}(\X)$ is a quotient object of $E$, it is a quotient of $\coker(s)$.
  Hence, $Q \in \Coh_{0}(\X) \subset \T_{\nu,\delta}$ and $E$ is a $(\T_{\nu,\delta},\F_{\nu,\delta})$-pair.

  Conversely, suppose that $E$ is a $(\T_{\nu,\delta}, \F_{\nu,\delta})$-pair.
  Set $\T_{PT} = \Coh_{0}(\X)$ and $\F_{PT} = \Coh_{1}(\X)$.
  By Lemma~\ref{prop:CohomCrit}, it suffices to show that $E$ is a $(\T_{PT},\F_{PT})$-pair.

  If $S \in \Coh_{\leq 1}(\X)$ is a subobject of $E$, then $S$ is a pure 1-dimensional sheaf.
  Hence $S \in \F_{PT}$.
  Furthermore, let $G$ be the pure 1-dimensional part of $H^0(E)$, which is a quotient of $E$ in $\A$.
  If $G \neq 0$, then taking $\delta > \nu(G)$ implies $G \not\in \T_{\nu,\delta}$, which contradicts $E$ being a $(\T_{\nu,\delta}, \F_{\nu,\delta})$-pair.
  Hence $G = 0$, and so $H^0(E) \in \Coh_{0}(\X) = \T_{PT}$.
\end{proof}
\begin{lem}\label{lem:SubobjectFactors}
  Let $E \in D(\X)$ be an object of the form $E = (\hO_\X \to G)$, and let $C \in \Coh_{\leq 1}(\X)$.
  Any morphism $C \to E$ factors through $G \to E$.
\end{lem}
\begin{proof}
  We have an exact triangle $G \to E \to \hO_\X[1]$, and for reasons of dimension and Serre duality $\Hom(C,\hO_\X[1]) \cong H^2(\X,C)^{\vee} = 0$.
\end{proof}

\begin{cor}\label{cor:Delta_Pairs_are_Stable_Pairs_at_Infinity}
  Let $(\beta,c) \in N_{\leq 1}(\X)$.
  If $\delta > \max(0, \deg(\beta,c) - M^{-}_{\le \beta})$, then
  \begin{equation*}
    DT^{\nu,\delta}_{(\beta,c)} = PT(\X)_{(\beta,c)}.
  \end{equation*}
\end{cor}

\subsection{The proof of rationality}
\label{subsec:Proof_of_Rationality}
With all the boundedness results in place, we now apply the numerical wall-crossing formula to prove the rationality of the generating series of stable pair invariants on a general CY3 orbifold $\X$.
Our approach is to compare the full series $DT^{\nu,\delta}(\X)$ with $DT^{\nu,\infty}(\X)= PT(\X)$.

Applying the wall-crossing formula directly is problematic, because the $\delta$-walls are dense in $\R$.
However, fixing a $\beta \in N_{1}(\X)$, we may instead focus on the sub-series $DT^{\nu,\delta}(\X)_{\beta}$, and more generally the series
\begin{align}
    DT^{\nu,\delta}_{\leq \beta} \coloneq DT^{\nu,\delta}(\X)_{\leq \beta} &\coloneq \sum_{\beta' \leq \beta} \sum_{c \in N_{0}(\X)} DT^{\nu,\delta}_{(\beta',c)}z^{\beta'}q^{c}, \\
    J^{\nu}(\delta)_{\leq \beta} &\coloneq \sum_{\beta' \leq \beta} \sum_{\substack{c \in N_0(\X) \\ \nu(\beta',c) = \delta}} J^{\nu}_{(\beta',c)} z^{\beta'} q^c \label{eqn:DefOfJDelta}.
\end{align}
By Lemma~\ref{lem:Walls_for_Delta_Pairs}, the notion of $(\T_{\nu,\delta},\F_{\nu,\delta})$-pair of class $(-1,\beta',c)$ with $\beta' \leq \beta$ can only change when $\delta \in W_{\beta}$.
Consequently, the same holds for the above series.
In particular the walls $W_{\beta}$ are discrete, and one can write the wall-crossing formula comparing $DT^{\nu,\delta}(\X)_{\le \beta}$ to $DT^{\nu, \infty}(\X)_{\le \beta}$ as a countably infinite product.

Given that we restrict our attention to $DT^{\nu,\delta}(\X)_{\le \beta}$, we now introduce a suitable truncation $\Q[N^{\eff}(\X)]_{\le \beta}$ of the Poisson torus $\Q[N(\X)]$.
\begin{defn}\label{Quantum_Torus_of_A}
  Define $\Q[N^{\eff}(\X)] \subset \Q[N(\X)]$ as the vector space with $\Q$-basis
  \begin{equation*}
    \Bigl\{z^\beta q^c t^{-k[\hO_{\X}]} \in \Q[N(\X)] \Bigm| \beta \in N^{\eff}_1(\X),\, c \in N_0(\X),\, k \in \Z_{\geq 0}\Bigr\}
  \end{equation*}
  The subspace $\Q[N^{\eff}(\X)]$ is a Poisson subalgebra of $\Q[N(\X)]$.
\end{defn}
Consider the ideal $I_\beta \subset \Q[N^{\eff}(\X)]$ generated by $\{z^{\beta'}, t^{-2[\hO_{\X}]} \mid \beta' \not\leq \beta\}$, and let $\Q[N^{\eff}(\X)]_{\le \beta} = \Q[N^{\eff}(\X)]/I_{\beta}$ denote the quotient.
It is again a Poisson algebra.
\begin{prop}\label{prop:WCusingTFpairs}
  Let $\beta \in N_1(\X)$, let $\delta \in W_\beta$, and let $0 < \epsilon < \frac{1}{l(\beta)!}$.
  The identity
  \begin{align*}
    DT^{\nu,\delta + \epsilon}_{\leq \beta}t^{-[\hO_\X]}  = \exp(\{J^{\nu}(\delta)_{\leq \beta},- \})
    DT^{\nu,\delta - \epsilon}_{\leq \beta}t^{-[\hO_\X]}
  \end{align*}
  holds in $\Q[N^{\eff}(\X)]_{\le \beta}$, where the term $J^{\nu}(\delta)_{\le \beta}$ is defined in the proof.
\end{prop}
\begin{proof}
  The assumption on $\epsilon$ implies that for any $E \in \Coh_{\le 1}(\X)$ with $\beta_{E} \le \beta$ and $\nu(E) \in [\delta - \epsilon, \delta + \epsilon)$, we have $\nu(E) = \delta$.
  In particular, this implies that for $\beta' \le \beta$, we have $\M_{\nu}([\delta - \epsilon, \delta + \epsilon)) = \M_{\nu}^{\ss}(\delta)$.
  After projecting to $\Q[N^{\eff}(\X)]_{\le \beta}$, we then have
  \begin{align*}
    I((\LL-1)\log([\uM_{\nu}([\delta-\epsilon, \delta + \epsilon))]) & = I((\LL-1)\log([\uM^{\ss}_{\nu}(\delta)]) \\
                                                                   & = \sum_{\beta' \le \beta}\sum_{\substack{c \in N_0(\X) \\ \nu(\beta',c) = \delta}} J^{\nu}_{(\beta',c)}z^{\beta'}q^{c}
  \end{align*}
  which is equal to $J^{\nu}(\delta)_{\le \beta}$.
  The torsion triple $(\T_{\nu,\delta - \epsilon},\M_\nu([\delta - \epsilon, \delta + \epsilon)),\F_{\nu,\delta + \epsilon})$ is wall-crossing material by Proposition \ref{cor:WallCrossingMaterial_Delta_Pairs}.
  Projecting equation~\eqref{eq:NumericalWallCrossingFormula} of Theorem~\ref{thm:WallCrossingFormula} to $\Q[N^{\eff}(\X)]_{\leq \beta}$, we obtain the identity
  \begin{equation*}
    I \Bigl((\LL-1)[\uP_{\nu, \delta + \epsilon}]\Bigr) = \exp(\{J^{\nu}(\delta)_{\leq \beta},-\}) I \Bigl((\LL-1)[\uP_{\nu, \delta - \epsilon}]\Bigr).
  \end{equation*}
  Evaluating the integrals by equation~\eqref{eq:DTRank1Definition} completes the proof.
\end{proof}

We now prove the rationality of the generating series of stable pair invariants.
\begin{thrm}\label{thm:PTIsRational}
  For each class $\beta \in N_{1}(\X)$, there exists a unique rational function $f_\beta(q)$ such that the series 
  \begin{equation}
    PT(\X)_{\beta} = \sum_{c \in N_{0}(\X)}PT(\X)_{(\beta,c)}q^c
  \end{equation}
  is the expansion in $\Q[N_{0}(\X)]_{\deg}$ of $f_\beta(q)$.

  More precisely, we can write $f_\beta(q)$ as a sum of functions $g_{D}/h_{D}$, where $D$ is a decomposition $\beta = \sum_{i=1}^{r}\beta_{i}$ into effective classes, where $g_{D} \in \Z[N_{0}(\X)]$, and where 
  \begin{equation}
    h_D = \prod_{i=1}^{r}(1-\prod_{j=1}^{i}q^{2\beta_{j}\cdot A})^{2i}.
  \end{equation}
\end{thrm}

\begin{proof}
  Let $\delta_{0} \in \R$.
  Iterating the wall-crossing formula from Proposition~\ref{prop:WCusingTFpairs},
  \begin{align*}
    DT^{\nu,\infty}_{\le \beta}t^{-[\hO_{\X}]} = \prod_{\delta \in W_{\beta} \cap [\delta_{0}, \infty)} \exp(\{J(\delta)_{\le \beta},-\})(DT^{\nu, \delta_{0}}_{\le \beta}t^{-[\hO_{\X}]}),
  \end{align*}
  where the product is taken in increasing order of $\delta$.
  Substituting in the definition of $J^{\nu}(\delta)_{\leq \beta}$ in equation~\eqref{eqn:DefOfJDelta} and expanding the exponential, the $z^{\beta}t^{-[\hO_{\X}]}$-coefficient of the right hand side becomes an infinite sum.
  The terms of this sum are described as follows.
  Fix an integer $r \in \Z_{\geq 1}$, a sequence $(\alpha_{i})_{i=1}^{r} = (\beta_{i},c_{i})_{i= 1}^{r} \subset N_{\le 1}(\X)$, and a class $\alpha' = (\beta',c') \in N_{\le 1}(\X)$, satisfying
  \begin{itemize}
  \item $\beta = \beta' + \sum \beta_{i}$,
  \item $\delta_{0} \le \nu(\alpha_{1}) \le \nu(\alpha_{2}) \le \cdots \le \nu(\alpha_{r})$,
  \item $J^{\nu}_{\alpha_{i}} \not= 0$ for all $1 \leq i \leq r$,
  \item $DT^{\nu,\delta}_{\alpha'} \not= 0$.
  \end{itemize}
  The non-zero term in the infinite sum associated with this data is
  \begin{align*}\label{eqn:Contribution}
    T((\alpha_{i}), \alpha')z^{\beta}q^{c'+\sum c_{i}}t^{-[\hO_{\X}]} &= A_{(\alpha_{i})}\{J^{\nu}_{\alpha_{r}}z^{\beta_{r}}q^{c_{r}},-\}\circ \cdots \circ \{J^{\nu}_{\alpha_{1}}z^{\beta_{1}}q^{c_{1}}, -\}(DT^{\nu, \delta_{0}}_{\alpha'}z^{\beta'}q^{c'}t^{-[\hO_{\X}]}),
  \end{align*}
  where $A_{(\alpha_i)}$ is a factor arising from the exponential:
  \begin{equation*}
    A_{(\alpha_{i})} \coloneq \prod_{\delta \in W_{\beta}}\frac{1}{|\{i \mid \nu(\alpha_{i})= \delta\}|!}.
  \end{equation*}
  Putting all these terms together gives
  \begin{equation*}
    DT^{\nu,\infty}_{\beta} = \sum_{(\alpha_{i}),\alpha'} T((\alpha_{i}),\alpha')z^{\beta}q^{c' + \sum c_{i}}.
  \end{equation*}

  We now claim that this is the expansion of a rational function with respect to $\deg$.
  To see this, we write out the formula for the Poisson bracket.
  This yields
  \begin{align*}
    T((\alpha_{i}),\alpha') =  A_{(\alpha_{i})} B_{(\alpha_{i}),\alpha'} \left(\prod_{i=1}^{r} J^{\nu}_{\alpha_{i}}\right) DT^{0}_{\alpha'}
  \end{align*}
  where
  \begin{equation}\label{eqn:ExpressionForB}
    B_{(\alpha_{i}),\alpha'} = \sigma^{\sum_{i < j}\chi(\alpha_{j},\alpha_{i}) + \sum_{i} \chi(\alpha_{i}, \alpha' - [\hO_{\X}])}\prod_{i = 1}^{r}\chi(\alpha_{i}, -[\hO_{\X}] + \alpha' + \sum_{j=1}^{i-1}\alpha_{j}).
  \end{equation}
  We emphasize that for the proof of rationality, the precise formula for $B_{(\alpha_{i}),\alpha'}$ is only important in that it depends quasi-polynomially on the classes $\alpha_{i}$.

  We partition these $T$-terms in groups as follows.
  A \emph{group} consists of the data of a class $\alpha' = (\beta',c')$, a sequence $(\beta_i)_{i=1}^{r}$, a sequence $(\kappa_i)_{i=1}^r$ where $\kappa_i \in N_{0}(\X)/\Z (A \cdot \beta_{i})$, and a subset $E \subseteq \{1, \ldots, r-1\}$.
  This data is required to satisfy the conditions
  \begin{itemize}
  \item $\beta = \beta' + \sum_{i=1}^{r}\beta_{i}$, and
  \item $J^{\nu}_{(\beta_{i},c_{i})} \not= 0$ for $c_i \in \kappa_i$ and $i = 1,2, \ldots, r$.
  \end{itemize}
  Note that for any class $(\beta'',c'') \in N_{\le 1}(\X)$, tensoring by $A$ induces an isomorphism $\uM^{\ss}_{\nu}(\beta'',c'') \cong \uM^{\ss}_{\nu}(\beta'', c'' + A \cdot \beta'')$.
  The invariant $J^{\nu}_{(\beta_i,c_i)}$ is thus independent of the choice of representative $c_i \in \kappa_i$, and we may write $J^{\nu}_{(\beta_i,\kappa_i)} \coloneq J^{\nu}_{(\beta_i,c_i)}$.

  Collecting all terms belonging to the group $(\alpha',(\beta_i),(\kappa_i),E)$, we obtain
  \begin{align*}
    C(\alpha',(\beta_{i}),(\kappa_{i}), E) = \sum_{c_{i}} T(r,((\beta_{i},c_{i})),\alpha')q^{c' + \sum c_{i}},
  \end{align*}
  where the sum is over all $c_{i} \in N_{0}(\X)$ such that
  \begin{align}\label{eq:SlopeConstraintsSigma}
    &c_i \in \kappa_i, \\
    &\label{eqn:SlopeConstraintInequalitySigma}
      \delta_{0} \le \nu(\beta_{1},c_{1}) \le \nu(\beta_{2},c_{2}) \le \cdots \le \nu(\beta_{r},c_{r}), \\
    &\label{eqn:SlopeConstraintEqualitySigma}
      \nu(\beta_{i},c_{i}) = \nu(\beta_{i+1},c_{i+1}) \Leftrightarrow i \in E.
  \end{align}
  Note that for such a choice of $c_i$, the factor $A_{((\beta_i,c_i))}$ defined above depends only on $E$.
  Indeed, set $\{n_i\} = \{1, \ldots, r\} \setminus E$ with $n_1 < n_2 < \ldots < n_{r - |E|}$.
  Then
  \begin{equation*}
    A_E \coloneq \prod \frac{1}{(n_{i}-n_{i-1})!} = A_{((\beta_i,c_i))}.
  \end{equation*}
  We find that the contribution of the group $(\alpha',(\beta_i),(\kappa_i),E)$ is
  \begin{equation*}
    C(\alpha',(\beta_i),(\kappa_i), E) = A_{E} \prod_{i=1}^{r} J^{\nu}_{\beta_i,\kappa_i} DT^{\nu,\delta_{0}}_{\alpha'} 
    \left(\sum_{c_i} B_{(\beta_i,c_i),\alpha'} q^{c' + \sum c_i}\right)
  \end{equation*}
  where the sum runs over the $c_i \in N_0(\X)$ satisfying the above conditions.
  
  Now, for every choice of $(\beta_i)$, $(\kappa_i)$, and $E$, there exists a sequence $(c_{i}^{0})$ with $c_{i}^{0} \in \kappa_{i}$ which is \emph{minimal} in the sense that replacing any $c_{i}^{0}$ with $c_{i}^{0} - A \cdot \beta_{i}$ would violate one of (\ref{eqn:SlopeConstraintInequalitySigma}) and (\ref{eqn:SlopeConstraintEqualitySigma}).
  We find
  \begin{equation*}
    C(\alpha',(\beta_i),(\kappa_i), E) = A_{E}  \prod_{i=1}^{r} J^{\nu}_{\beta_{i},\kappa_{i}} DT^{0}_{\alpha'} \left(\sum_{a_i} B_{(\beta_i ,c^0_i + a_i \beta_i \cdot A),\alpha'} q^{c'+\sum c_i^0+a_i\beta_i \cdot A}\right)
  \end{equation*}
  where the sum is over the set $S_E = \{0 \leq a_1 \leq a_2 \leq \ldots \leq a_r \mid a_i \in \Z,\, a_i = a_{i+1} \Leftrightarrow i \in E\}$.

  Since the Euler form is bilinear, we conclude by equation~\eqref{eqn:ExpressionForB} that $B$ depends quasi-polynomially on the $a_i$ with quasi-period $2$ (because of $\sigma$).
  Lemma \ref{thm:SumDefinedByInequalitiesIsARationalFunction} shows
  \begin{equation}\label{eq:Group_C_resummed_as_a_Rational_Function}
    C(\alpha',(\beta_i),(\kappa_i),E) = \frac{p}{\prod_{i\in [r] \setminus E} (1-\prod_{j = 1}^{i}q^{2\beta_{j}\cdot A})^{2i}}
  \end{equation}
  holds in $\Q[N_0(\X)]_{\deg}$ for some Laurent polynomial $p \in \Q[N_0(\X)]$.
  Moreover, it shows that the exponent is $2i$ because $\sum_{j \geq r-i+1} \deg_{a_j} B \leq 2i-1$.

  Finally, we claim that there are only finitely many such \emph{groups}, \ie, there are only finitely many non-trivial choices for the data of $(\alpha',(\beta_i), (\kappa_i),E)$.
  The sum of those rational functions is then $f_\beta(q)$.
  For the choice of $(\beta_i)$ and $E$, this is obvious.
  The claim for $\alpha'$ follows from part~(\ref*{enum:PNuFinitelyManyC}) of Proposition~\ref{thm:BoundednessOfRank1SigmaSemistables}, and the claim for $(\kappa_i)$ follows from part~(\ref*{enum:NuStableFinitelyManyC}) of Proposition~\ref{thm:SigmaDeltaHygiene}.
\end{proof}

\subsection{Duality properties of \texorpdfstring{$PT(\X)$}{stable pair invariants}}
\label{sec:PTDuality}
We establish a symmetry of $PT(\X)$ induced by the derived dualising functor $\D(-) = R\lHom(-,\hO_{\X})[2]$, analogous to the $q~\leftrightarrow~q^{-1}$ symmetry of $PT(Z)_{\gamma}(q)$ if $Z$ is a variety.
We fix a curve class $\beta \in N_1(\X)$.
Recall that if $F \in \Coh_{\leq 1}(\X)$ is of class $[F] = \beta$, then
\begin{equation*}
  p_{F}(k) = l(F)k + \deg(F)
\end{equation*}
denotes its modified Hilbert polynomial, where $l(F) = l(\beta)$ since $l(N_0(\X)) = 0$.

\begin{lem}\label{thm:DeltaPairsAreDualStablePairsAtMinusInfinity}
  Let $(\beta,c) \in N_{\leq 1}(\X)$ and $\delta \le 0$.
  If $\delta \le \deg(\beta,c) + M^{-}_{\le \beta}$, then we have $DT^{\nu,\delta}_{(\beta,c)} = PT(\X)_{\D(\beta,c)}$.
\end{lem}
\begin{proof}
  Combining Lemmas \ref{thm:DualisingPreservesSigmaStables} and \ref{thm:sigmaInfinityIsPT} yields
  \[
    DT^{\nu,\delta}_{(\beta,c)} = DT^{\nu,-\delta}_{\D(\beta,c)} = PT(\X)_{\D(\beta,c)}
  \]
  as required.
\end{proof}
As a consequence, we may define $DT^{\nu,-\infty}_{(\beta,c)} = DT^{\nu,\delta}_{(\beta, c)}$ for $\delta \ll 0$.

We now put a grading on $\Q[N_{0}(\X)]$ by $\deg(q^{c}) = \deg(c)$, and extend this to a grading on $\Q(N_{0}(\X))$ as in Section~\ref{sec:QuasiPolynomials}.
\begin{lem}\label{thm:degreeOfTheCTerm}
  Let $\beta \in N_1(\X)$ and assume that $\delta \leq -l(\beta)$.
  Then
  \begin{equation*}
    \deg\left(DT_{\beta}^{\nu,\delta} - DT_{\beta}^{\nu,\infty}\right) \le M^{+}_{\le \beta} + \delta l(\beta) + l(\beta)^{2}.
  \end{equation*}
  \end{lem}
\begin{proof}
  We use the notation as in the proof of Theorem~\ref{thm:PTIsRational}, so $\delta = \delta_{0}$.
  The function $DT_{\beta}^{\nu,\delta_{0}} - DT_{\beta}^{\nu,\infty}$ is a finite sum of rational functions of the form of equation~\eqref{eq:Group_C_resummed_as_a_Rational_Function}.
  By Lemma~\ref{thm:degreeClaimForRationalFunction}, the degree of these rational functions is less than $\deg(\beta,c' + \sum_{i=1}^{r}c_{i}^{0})$.

  By Lemma~\ref{thm:vanishingOfUP}, it thus suffices to show that
  \begin{equation*}
    \sum_{i=1}^{r} \deg(\beta_{i}, c_{i}^{0}) < \delta_{0} l(\beta) + l(\beta)^{2}.
  \end{equation*}
  We argue as follows.
  By minimality of the classes $c_{i}^{0} \in N_0(\X)$ satisfying condition \eqref{eq:SlopeConstraintsSigma}, \eqref{eqn:SlopeConstraintInequalitySigma}, \eqref{eqn:SlopeConstraintEqualitySigma}, we have inequalities
  \begin{equation*}
  \begin{split}
    \nu(\beta_{1},c_{1}^{0}) &< \delta_{0} + 1 \\
    \nu(\beta_{i},c_{i}^{0}) &\le \nu(\beta_{i-1},c_{i-1}^{0}) + 1 \quad \text{ for }i > 1,
  \end{split}
  \end{equation*}
  for if not one could replace some $c_{i}^{0}$ by $c_{i}^{0} - \beta_{i}\cdot A$.
  Thus $\nu(\beta_{i},c_{i}^{0}) \le \delta_{0} + i$, and so
  \[
    \deg(\beta_{i},c_{i}^{0}) < (\delta_{0} + i)l(\beta_{i}) \le (\delta_{0} + r)l(\beta_{i})
      \le \delta_{0} + r
  \]
  since $\delta_{0} + r < 0$.
  As $\beta = \beta_1 + \ldots + \beta_r$ implies that $r \le l(\beta)$, the claim follows.
\end{proof}

\begin{lem}
  \label{thm:DTMinusInfinityIsDTInfinity}
  For each class $\beta \in N_{1}(\X)$, the series
  \[
    DT^{\nu,-\infty}_{\beta} = \sum_{c \in N_{0}(\X)} DT^{\nu,-\infty}_{\beta,c}q^{c}
  \]
  is the expansion in $\Q[N_{0}(\X)]_{-\deg}$ of the rational function $PT(\X)_{\beta}$.
\end{lem}
\begin{proof}
  Lemma~\ref{thm:DeltaPairsAreDualStablePairsAtMinusInfinity} together with Theorem \ref{thm:PTIsRational} shows that $DT^{\nu,-\infty}(\X)_{\beta}$ is the expansion in $\Q[N_{0}(\X)]_{-\deg}$ of a rational function, and so it suffices to show that $DT^{\nu,\infty}_{\beta} - DT^{\nu,-\infty}_{\beta} = 0$ in $\Q(N_{0}(\X))$.

  Now let $\delta \ll 0$.
  On the one hand, Lemma~\ref{thm:DeltaPairsAreDualStablePairsAtMinusInfinity} shows that $DT^{\nu,\delta}_{(\beta, c)} = DT^{\nu,-\infty}_{(\beta,c)}$ if $\deg(\beta,c) \ge - M_{\le \beta}^{-} + \delta$.
  It follows that
  \[
    \deg\left(DT_{\beta}^{\nu,\delta} - DT_{\beta}^{\nu,-\infty}\right) \le - M_{\le \beta}^{-} + \delta - \deg(\beta,0).
  \]
  On the other hand, by taking $\delta < - l(\beta)$ smaller if necessary, Lemma~\ref{thm:degreeOfTheCTerm} yields
  \[
    \deg\left(DT_{\beta}^{\nu,\delta} - DT_{\beta}^{\nu,\infty}\right) \le M^{+}_{\le \beta} + \delta l(\beta) + l(\beta)^{2}.
  \]
  Combining these two bounds, we obtain
  \[
    \deg\left(DT_{\beta}^{\nu,\infty}- DT_{\beta}^{\nu,-\infty}\right) \le \max\{-M_{\le \beta}^{-} + \delta - \deg(\beta, 0), M^{+}_{\le \beta} + \delta l(\beta) + l(\beta)^{2}\}.
  \]
  Letting $\delta \to -\infty$ now implies that $DT_{\beta}^{\nu,\infty} = DT_{\beta}^{\nu,-\infty}$ in $\Q(N_0(\X))$.
\end{proof}

\begin{prop}
  \label{thm:dualityResultForPT}
  We have an equality of rational functions
  \[
    \D(z^{\beta}f_{\beta}(q)) = z^{\D(\beta)}f_{\D(\beta)}(q).
  \]
  Equivalently, the function
  \[
    PT(\X) = \sum_{\beta \in N_{1}(\X)} \sum_{c \in N_{0}(\X)} PT(\X)_{(\beta,c)} z^{\beta}q^{c}
  \]
  is invariant under the involution $\D(-)$, when the $q$-parts of the series are thought of as the rational functions $f_{\beta}(q)$.
\end{prop}
\begin{proof}
  By Lemma~\ref{thm:DeltaPairsAreDualStablePairsAtMinusInfinity}, we have $\D(PT(\X)) = DT^{\nu,-\infty}$, and by Lemma~\ref{thm:DTMinusInfinityIsDTInfinity} the equality of rational functions $DT^{\nu,\infty}_{\beta} = DT^{\nu,-\infty}_{\beta}$ holds for every $\beta \in N_{1}(\X)$.
\end{proof}



\section{The crepant resolution conjecture}
\label{sec:zetaStability}
In the final two sections, we prove the crepant resolution conjecture, in the form stated in Theorem~\ref{thm:MainTheorem}.
Recall that $\X$ is a CY3 orbifold in the sense of Section~\ref{suborbifolds}, and that $f \colon Y \to X$ denotes the distinguished crepant resolution of its coarse moduli space.
From this point on, we moreover assume that $\X$ satisfies the hard Lefschetz condition of Section~\ref{sec:hardLefschetz}.

We define a stability condition $\zeta$ on $\Coh_{\le 1}(\X)$ whose stability function takes values in $(-\infty,\infty]^{2}$.
Associated with $\zeta$, we obtain a two-parameter family of torsion pairs $(\Tz, \Fz)$ with $\gamma \in \R_{> 0}$ and $\eta \in \R$.
The associated $(\Tz,\Fz)$-pairs interpolate between stable pairs on $\X$ for $\gamma \gg 0$, and Bryan--Steinberg pairs relative to $f$ for $0 < \gamma \ll 1$.

Let $\beta \in N_{1}(\X)$.
The series $DT^{{\zeta, (\gamma, \eta)}}_{\le \beta}$ has a wall-crossing behaviour as described in Section~\ref{Intro_PT_vs_BS} and illustrated by its diagram.
The key step consists in controlling the series as a $\gamma$-wall is crossed.
This involves crossing countably infinitely many $\eta$-walls, and relates $DT_{\beta}^{\zeta, (\gamma,\infty)}$ to $DT_{\beta}^{\zeta, (\gamma,-\infty)}$ via the wall-crossing formula in a way which is similar to the passage from $DT^{\nu, \infty}$ to $DT^{\nu, -\infty}$ in Section~\ref{sec:rationalityOfStablePairs}.
However, the required boundedness results are more subtle as compared to $DT^{\nu, \delta}$.
In particular, the series $DT^{{\zeta, (\gamma, \eta)}}_{\le \beta}$ is in general not a Laurent polynomial.

We argue as follows.
Given a wall $\gamma \in V_{\beta}$, we show that there is a unique curve class $\beta_{\gamma} \le \beta$ for which $J^{\zeta}_{(\beta_{\gamma},c)}$ can contribute to the wall-crossing formula relating $DT_{\le \beta}^{\zeta, (\gamma, \infty)}$ to $DT_{\le \beta}^{\zeta, (\gamma, -\infty)}$.
Setting $c_{\gamma} = \beta_{\gamma} \cdot A$, we organise the wall-crossing in sub-series $DT^{\zeta, (\gamma,\eta)}_{\beta, c + \Z c_{\gamma}}$ of $DT^{\zeta, (\gamma, \eta)}_{\beta}$ for each class $c \in N_0(\X) / \Z c_{\gamma}$, defined by
\[
  DT_{\beta, c + \Z c_{\gamma}}^{\zeta, (\gamma,\eta)} = \sum_{k \in \Z} DT_{(\beta, c + kc_{\gamma})}^{\zeta, (\gamma,\eta)} z^{\beta}q^{c + kc_{\gamma}}.
\]
Through various boundedness results, we then show that $DT_{\beta, c + \Z c_{\gamma}}^{\zeta, (\gamma,\eta)}$ is a Laurent polynomial for any $\eta \in \R$.
An argument similar to that of Section~\ref{sec:rationalityOfStablePairs} shows that the limits $DT_{\beta, c + \Z c_{\gamma}}^{\zeta, (\gamma,\infty)}$ and $DT_{\beta, c + \Z c_{\gamma}}^{\zeta, (\gamma,-\infty)}$ exist and are equal as rational functions.
An application of Lemma~\ref{thm:QuasipolynomialDifferenceIsANewExpansion} then shows $DT_{\beta}^{\zeta, (\gamma,\infty)}$ and $DT_{\beta}^{\zeta, (\gamma,\infty)}$ are also equal as rational functions.
Finally, we slide off the $\gamma$-wall and show that
\[
DT_{\beta}^{\zeta,(\gamma,\pm \infty)} = DT_{\beta}^{\zeta,(\gamma \pm \epsilon,\eta)}
\]
for $0 < \epsilon \ll 1$ and $\eta \in \R$, thus completing the $\gamma$-wall-crossing.


\subsection{\texorpdfstring{$\zeta$}{Zeta}-stability}
For the remainder of the paper, we fix a generic ample class $\omega \in N^{1}(Y)_{\R}$.
Recall that $\Psi \colon D(\X) \to D(Y)$ is the inverse of the McKay correspondence.
To avoid confusion, we write $\deg_Y$ for the usual degree of zero-cycles on $Y$.


We now define a stability condition $\zeta$ on $\Coh_{\le 1}(\X)$.
\begin{defn}\label{def:Zeta_Stability}
  Define $\zeta \colon N_{1}^{\eff}(\X) \setminus \{0\} \to (-\infty,+\infty]^2$ by
  \begin{equation}\label{eq:Zeta_Stability_Condition}
    \zeta(\beta,c) = \left(-\dfrac{\deg_{Y}(\ch_{2}(\Psi(A \cdot \beta)) \cdot \omega)}{\deg(A \cdot \beta)},\nu(\beta,c)\right) \in (-\infty,+\infty]^2
  \end{equation}
  if $\beta \not= 0$, and let
  \[
    \zeta(0,c) = (\infty,\infty).
  \]
  We think of $(\infty,\infty]^2$ as a totally ordered set via the lexicographical ordering, so $(a,b) \le (a',b')$ if $a < a'$, or if $a = a'$ and $b \leq b'$.
\end{defn}
\begin{lem}
  The category $\Coh_{\leq 1}(\X)$ is $\zeta$-Artinian.
\end{lem}
\begin{proof}
  Let $F \in \Coh_{\le 1}(\X)$, and assume for a contradiction that $F = F_{0} \supset F_{1} \supset \ldots$ is an infinite chain of subobjects with $\zeta(F_{i}) \ge \zeta(F_{i-1})$ for all $i$.
  Now, as $\beta_{F_{i}} \le \beta_{F_{i-1}}$ and the set $\{ \beta \mid 0 \le \beta \le \beta_{F}\}$ is finite, we may reduce to the case where $\beta_{F_{i}} = \beta_{F}$ for all $i$.
  But then $\zeta(F_{i}) \ge \zeta(F_{i-1})$ implies that $\nu(F_{i}) \ge \nu(F_{i-1})$ for all $i$, which is impossible by the existence of Harder--Narasimhan filtrations for $\nu$.
\end{proof}
It is easy to see that $\zeta$ satisfies the see-saw property, and so we deduce
\begin{cor}\label{cor:Zeta_Is_A_Stability_Conditions}
  The function $\zeta$ defines a stability condition on $\Coh_{\leq 1}(\X)$.
\end{cor}

Let $(\gamma,\eta) \in \R_{> 0} \times \R$.
We obtain a family of torsion pairs on $\Coh_{\leq 1}(\X)$ by collapsing the Harder--Narasimhan filtration of $\zeta$-stability:
\begin{equation}
  \begin{split}
    \Tz &\coloneq \{T \in \Coh_{\leq 1}(\X) \,|\, T \onto Q \neq 0
    \Rightarrow \zeta(Q) \geq (\gamma,\eta) \} \\
    \Fz &\coloneq \{F \in \Coh_{\leq 1}(\X) \,|\, 0 \neq S \into F
    \Rightarrow \zeta_{\gamma,\eta}(F) < (\gamma,\eta) \}.
  \end{split}
\end{equation}
We write $\jP_{\zeta,(\gamma,\eta)} \subset \A$ for the full subcategory of $(\Tz,\Fz)$-pairs in the sense of Definition~\ref{def:TFpair}.

For any $\gamma > 0$, we define the linear function $L_{\gamma} \colon N_{0}(\X) \to \R$ by
\begin{equation}\label{deq:Definition_Lgamma}
  L_{\gamma}(c) = \deg(c) + \gamma^{-1}\deg_{Y}(\ch_{2}(\Psi(c)) \cdot \omega).
\end{equation}

\begin{rmk}
  The function $L_{\gamma}$ controls the series expansion, in the sense of Definition \ref{def:Expansion_Rational_Function_wrt_L}, of the rational functions $f_{\beta}(q) \in \Q(N_{0}(\X))$ of Theorem~\ref{thm:PTIsRational}.
  Roughly speaking, this means that a class $c \in N_{0}(\X)$ is thought of as ``effective'' in the expansion of $f_{\beta}(q)$ at $(\gamma,\eta)$ if $L_{\gamma}(c) > 0$.

  It is easy to see that when $\gamma \gg 0$, which corresponds to stable pairs on $\X$, we have $L_{\gamma}(c) > 0$ if $c$ is effective.
  Similarly, when $0 < \gamma \ll 1$, which corresponds to Bryan--Steinberg pairs on $Y$, we have $L_{\gamma}(c) > 0$ if $\Psi(c)$ is effective.
\end{rmk}

\begin{rmk}
  For $\beta \not= 0$, if $\zeta(\beta,c) = (\zeta_{1}(\beta),\nu(\beta,c))$, then $\gamma - \zeta_{1}(\beta)$ has the same sign as $L_{\gamma}(\beta \cdot A)$.
  So the torsion pair $(\Tz, \Fz)$ can equivalently be described as
  \begin{equation}
    \begin{split}
      \Tz \coloneq &\{T \in \Coh_{\leq 1}(\X) \,|\, T \onto Q \neq 0
          \Rightarrow L_{\gamma}(Q \cdot A) < 0\} \\
          = &\{T \in \Coh_{\leq 1}(\X) \,|\, T \onto Q \neq 0
          \Rightarrow L_{\gamma}(Q \cdot A) = 0,\, \nu(Q) \ge \eta\} \\
      \Fz \coloneq &\{F \in \Coh_{\leq 1}(\X) \,|\, 0 \neq S \into F
          \Rightarrow L_{\gamma}(F \cdot A) > 0\} \\
          = &\{F \in \Coh_{\leq 1}(\X) \,|\, 0 \neq S \into F
          \Rightarrow L_{\gamma}(F \cdot A) = 0,\, \nu(F) < \eta\}.
    \end{split}
  \end{equation}

\end{rmk}

\subsection{Openness of \texorpdfstring{$(\Tz,\Fz)$}{the Zeta torsion pair}}
Let $(\gamma,\eta) \in \R_{\geq 0} \times \R$.
In this section we prove that the torsion pair $(\Tz,\Fz)$ is open.

\begin{defn}\label{def:Theta_Stability_Function}
  Define a stability function $\theta \colon \N^{\eff}_{0}(\X) \setminus \{0\} \to \R$ by setting
  \begin{equation}
    \theta(c) = - \dfrac{\deg_{Y}(\ch_{2}(\Psi(c)) \cdot \omega)}{\deg(c)}
  \end{equation}
\end{defn}
The function $\theta$ satisfies the see-saw property, and hence defines a stability condition on $\Coh_{0}(\X)$, since this category is Artinian.
In particular, objects in $\Coh_0(\X)$ have Harder--Narasimhan filtrations with respect to $\theta$.

Thus we may define a torsion pair $(\Tt,\Ft)$ on $\Coh_0(\X)$ by setting
\begin{equation}\label{eq:Theta_Torsion_Pair}
  \begin{split}
    \Tt &\coloneq \{T \in \Coh_0(\X) \mid T \onto Q \neq 0 \Rightarrow \theta(Q) \geq \gamma\} \\
    \Ft &\coloneq \{F \in \Coh_0(\X) \mid 0 \neq S \into F \Rightarrow \theta(S) < \gamma\}.
  \end{split}
\end{equation}

\begin{lem}
  The torsion pair $(\Tt, \Ft)$ is open.
\end{lem}
\begin{proof}
  We must show that the substacks $\uT_{\tg}$ and $\uF_{\tg}$, parametrising objects in $\Tt$ and $\Ft$ respectively, are open in $\cCoh_{\X,0}$.
  This follows from the arguments of \cite[Thm.~2.3.1]{MR2665168}, since there are at most finitely many classes of potentially destabilising quotients.
  In turn, this follows because the set $\{0 \leq c' \leq c \mid c' \in N_{0}(\X)\}$ is finite for every $c \in N_0(\X)$, since $\Coh_0(\X)$ is Artinian.
\end{proof}

\begin{defn}\label{def:Good_Pencil}
  Let $E \in \Coh_{\leq 1}(\X)$, and let $n \in \Z_{>0}$.
  We say that a pencil $L = \P^{1} \subset |nA|$ is a \emph{good pencil for $E$} if the following conditions hold:
  \begin{enumerate}
  \item the base locus of $L$ intersects neither $\supp(E)$ nor the singular locus of $X$,
  \item no member of $L$ contains a 1-dimensional component of $\supp(E)$.
  \end{enumerate}
  For a point $p \in L$ we denote the associated divisor substack by $D_p \into \X$.
  Let $\X_L$ denote the blow-up of $\X$ in the base locus of $L$, and let $b \colon \X_L \to L$ denote the natural morphism.
\end{defn}

By Bertini's theorem, there exists a good pencil for every $E \in \Coh_{\leq 1}(\X)$.
\begin{lem}\label{thm:ConditionsForSemistability}
  Let $E \in \Coh_{\leq 1}(\X)$, and let $L$ be a good pencil for $E$.
  Then $E \in \Tz$ if and only if it satisfies conditions \textbf{T1} and \textbf{T2}:
  \begin{enumerate}
  \item[\textbf{(T1)}] There exists a $p \in L$ such that the restriction $E|_{D_p}$ lies in $\cat T_{\theta, \gamma}$.
  \item[\textbf{(T2)}] The sheaf $E$ admits no quotient sheaf $Q$ with
    \begin{equation}
      L_{\gamma}(A \cdot \beta_Q) = 0
    \end{equation}
    and $\nu(Q) < \eta$.
  \end{enumerate}
  We have $E \in \Fz$ if and only if it satisfies conditions \textbf{F1} and \textbf{F2}:
  \begin{enumerate}
  \item[\textbf{(F1)}] There exists a $p \in L$ such that the restriction $E|_{D_p}$ lies in $\Ft$
  \item[\textbf{(F2)}] The sheaf $E$ admits no subsheaf $S$ with
    \begin{equation}
      L_{\gamma}(A \cdot \beta_S) = 0
    \end{equation}
    and $\nu(S) \geq \eta$.
  \end{enumerate}
\end{lem}
\begin{proof}
  We only treat the characterisation of membership of $\Tz$ since the arguments for membership of $\Fz$ are similar.

  A sheaf $E$ fails to lie in $\Tz$ if and only if there is a surjection $E \onto Q$ with $L_{\gamma}(A \cdot \beta_{Q}) > 0$ or with $Q$ as in \textbf{T2}.
  Thus it suffices to show that $E$ violates condition \textbf{T1} if and only if there exists a surjection $E \onto Q$ with $L_{\gamma}(A \cdot \beta_Q) > 0$.
  
  First assume that such a quotient $E \onto Q$ exists.
  For a general point $p \in L$, its restriction $Q|_{D_p}$ is a quotient of $E|_{D_p}$ since $L$ is a good pencil for $E$. 
  This shows that $E|_{D_p} \not\in \T_{\theta, \gamma}$.

  Conversely, suppose that condition \textbf{T1} does not hold.
  Since the support of $E$ is disjoint from the base locus of $L$, we may think of $E$ as a sheaf on the blow-up $b \colon \X_L \to L$.
  There exists an open subset $U \subseteq L$ such that $E|_{U}$ is flat over $U$.

  An easy modification of the argument of \cite[Thm.~2.3.2]{MR2665168} shows that there exists a filtration of $E|_{U} \in \Coh(\X_{L}|_{U})$,
  \[
    0 \subset E_{1} \subset \cdots \subset E_{n} = E|_{U},
  \]
  such that for a generic point $p \in U$, the induced filtration of $E|_{p}$ is the $\theta$-HN-filtration.
  In particular, since \textbf{T1} fails, we have $\theta((E_{n}/E_{n-1})|_{p}) < \gamma$.
  Let $j \colon \X_{L}|_{U} \to \X$ denote the natural map, and consider the composition
  \[
    E \to j^{*}j_{*}(E) = j_{*}(E|_{U}) \to j_{*}(E|_{U}/E_{n}).
  \]
  Letting $Q$ be the image of $E$ under this map, we obtain a surjection $E \to Q$.
  For a general $p \in U$, we have $\theta(Q|_{D_{p}}) < \gamma$, and so $L_{\gamma}(\beta_{Q} \cdot A) > 0$ as required.
\end{proof}
We now prove that the torsion pair $(\Tz,\Fz)$ is open.
\begin{lem}\label{thm:ConditionAIsOpen} \label{thm:ConditionBIsOpen}
  Conditions \textbf{T1}, \textbf{T2}, \textbf{F1} and \textbf{F2} are open in flat families in $\Coh_{\leq 1}(\X)$.
\end{lem}
\begin{proof}
  We first prove openness of \textbf{T1}.
  Let $S$ be the base scheme of a flat family of sheaves in $\Coh_{\leq 1}(\X)$, and let $E_s$ be the sheaf corresponding to some point $s \in S$.
  There exists a good pencil $L \subset |nA|$ for $E_s$. 
  Suppose that $E_s$ satisfies condition \textbf{T1}, and let $p \in L$ be a point for which the restriction $(E_s)|_{D_p}$ lies in $\Tt$.
  Picking a suitable open neighbourhood $s \in U \subset S$, the pencil $L$ remains good for all sheaves in the neighbourhood.
  Since $\Tt$ is open, $(E_u)|_{D_{p}}$ lies in $\Tt$ for all $u \in U$.
  Openness of condition \textbf{F1} follows by the same argument.

  Openness of \textbf{T2} is shown in the same way as the openness part of Theorem~\ref{thm:NironisBoundednessTheorem}: Given a family of sheaves $E$ over a finite type base scheme $S$, then for any $s \in S$ and surjection $E_{s} \onto F$ with $\nu(F) < \eta$, we must have $(\beta_{F}, c_{F}) \subset \{(\beta_{i}, c_{i})\}_{i=1}^r$. 
  Imposing the extra condition $L_{\gamma}(\beta' \cdot A) = 0$, we are left with a finite set of classes $\{(\beta'_{i},c'_{i})\}$, and the set of $s \in S$ where \textbf{T2} holds is the complement of the image of
  \[
    \bigcup_{i} \Quot(E, (\beta_{i}', c_{i}'))
  \]
  in $S$.
  But this set is closed since the $\Quot$ scheme is projective over $S$.
  This proves openness of \textbf{F1}.
  The case of \textbf{F2} is similar.
\end{proof}

\begin{cor}\label{cor:ZetaGammaEta_TorsionPair_Open}
  The torsion pair $(\Tz,\Fz)$ is open for all $(\gamma,\eta) \in \R_{>0} \times \R$.
\end{cor}
Proposition \ref{thm:TFOpenImpliesPOpen} then gives the following result.
\begin{cor}\label{cor:ZetaGammaEta_Pairs_Open_TorsionPair_Open_Stack}
  The category $\jP_{\zeta,(\gamma,\eta)}$ is open for all $(\gamma, \eta) \in \R_{>0} \times \R$.
\end{cor}

\subsection{Boundedness results}
We now prove a number of boundedness properties of the moduli stacks of $\zeta$-semistable sheaves and of $(\Tz,\Fz)$-pairs.

Consider the following `limit' subcategories of $\Coh_{\leq 1}(\X)$.
\begin{equation}\label{eq:TorsionPair_Zeta_Infinity}
  \T_{\zeta, (\gamma,-\infty)} = \bigcup_{\eta \in \R} \T_{{\zeta, (\gamma, \eta)}}
  \quad \text{ and } \quad 
  \F_{\zeta, (\gamma,\infty)} = \bigcup_{\eta \in \R} \F_{{\zeta, (\gamma, \eta)}}.
\end{equation}

\begin{defn}
  Let $S \subset N_{0}(\X)$, and let $L \colon N_{0}(\X) \to \R$ be a homomorphism.
  We say that $S$ is \emph{weakly $L$-bounded} if the image of $S$ in $N_{0}(\X)/\ker L$ is $L$-bounded in the sense of Definition \ref{def:Lbounded}.
\end{defn}
\begin{lem}
  \label{thm:ZetaBoundedness1DimensionalSheaves}
  Let $\gamma \in \R_{> 0}$, let $\eta,\eta_{1},\eta_{2} \in \R$, and let $\beta \in N_{1}(\X)$.
  The sets
  \begin{align}
    \label{eqn:setNumberOne} \{c_{F} \in N_{0}(\X) &\mid \exists F \in \T_{\nu, \eta_1}
                                                     \cap \F_{\zeta, (\gamma,\eta_2)} \text{ with } \beta_{F} \leq \beta \}, \\
    \label{eqn:setNumberTwo} \{c_{F} \in N_{0}(\X) &\mid \exists F \in \F_{\nu, \eta_1}
                                                     \cap \T_{\zeta, (\gamma,\eta_2)} \text{ with } \beta_{F} \leq \beta\}
  \end{align}
  are each $L_{\gamma}$-bounded.
  The sets
  \begin{align}
    \label{eqn:setNumberThree} \{c_{F} \in N_{0}(\X) &\mid \exists F \in \T_{\nu, \eta} \cap \F_{\zeta, (\gamma,\infty)} \text{ with } \beta_{F} \leq \beta \}, \\ 
    \label{eqn:setNumberFour} \{c_{F} \in N_{0}(\X) &\mid \exists F \in \F_{\nu, \eta} \cap \T_{\zeta, (\gamma,-\infty)} \text{ with } \beta_{F} \leq \beta\}
  \end{align}
  are each weakly $L_{\gamma}$-bounded.
\end{lem}
\begin{proof}
  We only prove the claims for the sets in equations~\eqref{eqn:setNumberOne} and \eqref{eqn:setNumberThree}, as the other two sets can be dealt with by a similar argument.

  Define
  \[
    S \coloneq \{F \in \Coh_{\leq 1}(\X) \mid F \in \T_{\nu, \eta_1} \cap \F_{\zeta, (\gamma,\eta_2)} \text{ with } \beta_{F} \leq \beta\},
  \]
  and let $x \in \R$.
  We have to prove that $c(S) \cap \{c \in N_0(\X) \mid L_\gamma(c) \leq x\}$ is a finite set.

  We first assume that $\eta_{1} = \eta_{2} = \eta$.
  Let $F \in S$.
  Then $F$ is pure 1-dimensional, and so there exists a $k \in \Z_{\ge 0}$ such that we may coarsen the $\nu$-HN-filtration of $F$ to get
  \begin{equation}
    F_{[\eta,\eta+1)}, F_{[\eta+1,\eta+2)}, \ldots, F_{[\eta+k, \eta+k+1)},
  \end{equation}
  with each $F_{I} \in \M_{\nu}(I)$; note that some $F_{I}$ may be zero.
  
  Let $S_{m} \subseteq S$ denote the set of $F$ with at most $m$ non-zero pieces in this coarse filtration, and let $S_{m}' \subseteq S_{m}$ be the subset of those for which $F_{[\eta, \eta+1)} \not= 0$.
  Note that $S_1'$ is a subset of
  \begin{equation}
    R \coloneq \{F \in \Coh_{1}(\X) \mid \beta_{F} \le \beta,\, F \in \M_{\nu}([\eta,\eta+1))\}.
  \end{equation}
  Hence, it follows that $c(S_{1}') \subseteq c(R)$ is finite by Theorem~\ref{thm:NironisBoundednessTheorem}.

  The set $Q = \Z_{>0}\{\beta' \cdot A \mid \beta' \leq \beta \text{ and }L_{\gamma}(\beta' \cdot A) > 0\}$ is $L_{\gamma}$-bounded.
  Twisting by $A$, we find $c(S_{m}) \subseteq c(S_{m}') + Q$, and so $c(S_{m})$ is $L_{\gamma}$-bounded if $c(S'_{m})$ is.

  Now take an object $F \in S'_m$, and decompose $F$ as
  \[
    0 \to F_{[\eta + 1, \infty)} \to F \to F_{[\eta, \eta + 1)} \to 0
  \]
  with $F_{I} \in \M_{\nu}(I)$.
  Then $F_{[\eta,\eta+1)} \in R$, and since $\cat F_{\zeta, (\gamma, \eta)}$ is closed under subobjects, $F_{[\eta+1,\infty)} \in S_{m-1}$.
  Hence $c(S_{m}') \subseteq c(S_{m-1}) + c(R)$, and so $c(S'_{m})$ is $L_{\gamma}$-bounded if $c(S_{m-1})$ is.
  Since $S_{m} = S$ when $m \ge l(\beta)$, a finite induction then gives the claim.

  Now let $\eta_1,\eta_2 \in \R$ be arbitrary.
  Without loss of generality we may assume $\eta_1 < \eta_2$ for otherwise we reduce to the known claim.
  For $F \in S$, consider the $\nu$-HN filtration
  \begin{equation}
    0 \to F_{[\eta_2, \infty)} \to F \to F_{[\eta_1,\eta_2)} \to 0.
  \end{equation}
  The set of possible values for $c_{F_{[\eta_{1},\eta_{2})}}$ is finite by Theorem~\ref{thm:NironisBoundednessTheorem}.
  Since $\F_{\zeta, (\gamma,\eta_2)}$ is closed under subobjects, we have $F_{[\eta_{2}, \infty)} \in \T_{\nu, \eta_2} \cap \F_{\zeta, (\gamma,\eta_2)}$.
  Thus by the previously treated case of $\eta_{1} = \eta_{2}$, the set of possible classes for $F_{[\eta_{2}, \infty)}$ is $L_{\gamma}$-bounded.
  This completes the claim for the first mentioned set.

  For the set (\ref{eqn:setNumberThree}), define $Q = \Z_{>0}\{\beta' \cdot A \mid \beta' \le \beta \text{ and }L_{\gamma}(\beta\cdot A) \ge 0\}$ and note that this set is only weakly $L_{\gamma}$-bounded.
  We conclude by the same argument.
\end{proof}

Let $\M^{\ss}_{\zeta}(a,b) \subset \Coh_{\leq 1}(\X)$ denote the full subcategory of $\zeta$-semistable sheaves of slopes $(a,b) \in \R^2$.
\begin{prop}\label{prop:ZetaGammaEta_Semistables_DecompositionallyFinite}
  Let $(\beta,c) \in N_{\leq 1}(\X)$ be a class and let $(a,b) \in \R^2$ be a slope.
  \begin{enumerate}
  \item The moduli stack $\uM^{\ss}_{\zeta}(a,b) \subset \cCoh_{\leq 1,\X}$ is open and, in particular, it is an algebraic stack locally of finite type.
  \item The set
    \[
      \{c \in N_{0}(\X) \mid \uM^{\ss}_{\zeta}(a,b) \not= \varnothing\}
    \]
    is $L_{\gamma}$-bounded.
  \item The category $\M^{\ss}_{\zeta}(a,b)$ is log-able as in Definition~\ref{dfn:DecompositionallyFinite}.
  \end{enumerate}
\end{prop}
\begin{proof}
  For the first claim, note that for any $(\beta, c) \in N_{\le 1}(\X)$, there exists an $\epsilon > 0$ such that if $[F] = (\beta,c)$, then
  \[
    F \in \M^{\ss}_{\zeta}(a,b) \Leftrightarrow F \in \cat T_{\gamma, (a, b)} \cap \cat F_{\gamma, (a, b+\epsilon)}.
  \]
  Openness of $\M^{\ss}_{\zeta}(a,b)$ then follows by Corollary \ref{cor:ZetaGammaEta_TorsionPair_Open}.
  
  For the second part, the category $\M^{\ss}_{\zeta}(a,b)$ is obviously closed under direct sums and direct summands.
  Let $E \in \M^{\ss}_{\zeta}(a,b)$, and decompose $E$ with respect to the $\nu$-HN filtration
  \begin{equation}
    0 \to E_{\geq \eta} \to E \to E_{< \eta} \to 0
  \end{equation}
  where $E_{\geq \eta} \in \T_{\nu, \eta}$ and $E_{< \eta} \in \F_{\nu, \eta}$.
  We deduce $E_{\geq \eta} \in \F_{\gamma, (a, b + \epsilon)}$ and $E_{< \eta} \in \cat T_{\gamma, (a, b)}$.
  Thus by Lemma \ref{thm:ZetaBoundedness1DimensionalSheaves}, the set of possible values for $c_{E_{\geq \eta}}$ and $c_{E_{< \eta}}$ are each $L_{\gamma}$-bounded.
  Since $c_{E_{\geq \eta}} + c_{E_{< \eta}} = c_{E}$, it follows that there are finitely many choices for each one.

  Applying Corollary \ref{thm:BoundingDegreeAndNuMinMaxGivesFiniteType}, we find that the moduli of possibilities for $E_{\geq \eta}$ and $E_{< \eta}$ are of finite type.
  By Proposition \ref{thm:extensionStackIsFiniteType}, then, the stack $\M^{\ss}_{\zeta}(a,b)$ is of finite type.

  The decomposition $c_{E} = c_{E_{\geq \eta}} + c_{E_{< \eta}}$ also shows that the set of possible values for $c_{E}$ is $L_{\gamma}$-bounded, which implies that $(\beta_{E}, c_{E})$ can be written as a sum of classes $(\beta_{E_{i}}, c_{E_{i}})$ with $E_{i} \in \M^{\ss}_{\zeta}(a,b)$ in at most finitely many ways. 
\end{proof}

\begin{prop}\label{thm:BoundednessOfRank1ObjectsForZetaStability}
  For any $(\gamma,\eta) \in \R_{>0} \times \R$, the set
  \begin{equation}\label{eq:Boundedness_of_classes_c_ZetaGammaEta_Pairs}
    \{c \in N_{0}(X) \mid \jP_{\zeta,(\gamma,\eta)}(\beta, c) \not= \emptyset\}
  \end{equation}
  is $L_{\gamma}$-bounded.
  Moreover, the stack $\uP_{\zeta,(\gamma,\eta)}(\beta,c)$ is of finite type.
\end{prop}
\begin{proof}
  Let $E \in \jP_{\zeta,(\gamma,\eta)}$.
  By Proposition~\ref{prop:InducedTorsionTriple}, it has a three-term filtration induced by the torsion triple $(\T_{\nu,\eta},V(\T_{\nu,\eta},\F_{\nu,\eta}),\F_{\nu,\eta})$ on $\A$.
  Thus we have an inclusion and a surjection
  \begin{equation}
    E_{\geq \eta} \into E, \quad \text{ and } \quad E \onto E_{<\eta},
  \end{equation}
  where $E_{\geq \eta} \in \T_{\nu, \eta}$, $E_{<\eta} \in \F_{\nu, \eta}$, and the middle filtration quotient
  \[
  E_{\jP_{\nu,\eta}} \coloneq \ker(E \onto E_{<\eta})/E_{\geq \eta} \in \jP_{\nu, \eta}.
  \]

  Since $E \in \jP_{\zeta,(\gamma,\eta)}$, we have $E_{\geq \eta} \in \F_{\gamma,\eta}$ and $E_{<\eta} \in \T_{\gamma,\eta}$.
  By Lemma~\ref{thm:ZetaBoundedness1DimensionalSheaves}, it follows that the sets of possible values for $c(E_{\geq \eta})$ and $c(E_{<\eta})$ are both $L_{\gamma}$-bounded.
  Moreover, the set of possible values for $c(E_{\jP_{\nu,\eta}})$ is finite, by part~\ref*{enum:PNuFinitelyManyC} of Lemma~\ref{thm:BoundednessOfRank1SigmaSemistables}.
  Thus, the set (\ref{eq:Boundedness_of_classes_c_ZetaGammaEta_Pairs}) is $L_{\gamma}$-bounded.

  Arguing as in the proof of Proposition \ref{thm:ZetaBoundedness1DimensionalSheaves}, the moduli for $E_{\geq \eta}, E_{< \eta}$, and $E_{\jP_{\nu,\eta}}$ are each of finite type, by Corollary~\ref{thm:BoundingDegreeAndNuMinMaxGivesFiniteType} and Lemma~\ref{thm:BoundednessOfRank1SigmaSemistables}.
  Hence $\uP_{\zge}(\beta_{E},c_{E})$ is of finite type.
\end{proof}
\begin{cor}
  Let $(\gamma,\eta) \in \R_{>0} \times \R$.
  The category $\jP_{\zge}$ defines an element in $H_{\gr}(\cC)$.
\end{cor}


\subsection{The walls for \texorpdfstring{$(\Tz,\Fz)$}{zeta}-pairs}
\label{Chap5_CRC_CountingGammaEtaPairs}
Let $\beta \in N_1(\X)$ be a class.
First, we locate the walls for $(\gamma,\eta) \in \R_{>0} \times \R$ where the notion of $(\Tz,\Fz)$-pair of class $(-1,\beta',c')$ with $\beta' \le \beta$ could change.
Recall that $W_{\beta} = (1/l(\beta)!)\Z \subset \R$.

Let $\Coh_{\le 1}(\X)_{\le \beta} \subset \Coh_{\le 1}(\X)$ be the subcategory consisting of sheaves $F$ with $\beta_{F} \le \beta$.
\begin{lem}
  \label{thm:ZetaStabilityUnchangedAwayFromBigWalls}
  \label{cor:Walls_For_ZetaGammaEta_Pairs_On_Gamma_Wall_are_NironiWalls}
  Let $\beta \in N_1(\X)$.
  The categories $\Tz \cap \Coh_{\le 1}(\X)_{\le \beta}$ and $\Fz \cap \Coh_{\le 1}(\X)_{\le \beta}$ are constant on the connected components of $(\R_{> 0} \times \R) \setminus (V_{\beta} \times \R)$, where 
  \begin{equation}
    V_{\beta} = \left\{-\dfrac{\deg_{Y}(\ch_{2}(\Psi(A \cdot \beta')) \cdot \omega)}{\deg(A \cdot \beta')} \,\colon\, 0 < \beta' \leq \beta \right\} \cap \R_{>0}.
  \end{equation}
  Moreover, for each $\gamma \in V_{\beta}$, the parts $\Tz \cap \Coh_{\le 1}(\X)_{\le \beta}$ and $\Fz~\cap~\Coh_{\le 1}(\X)_{\le \beta}$ of the torsion pair $(\Tz,\Fz)$ are locally constant on $\{\gamma\} \times \R \setminus W_{\beta}$.
\end{lem}
\begin{proof}
  Argue as in Lemma \ref{lem:Walls_for_Delta_Pairs}.
\end{proof}

\subsection{Counting invariants for \texorpdfstring{$\zeta$}{zeta}-stability}
Second, we define DT-type invariants virtually counting $\zeta$-semistable sheaves and $(\Tz,\Fz)$-pairs.

\subsubsection{Rank $0$}
Let $(a,b) \in \R^2$.
Consider the subcategory $\M^{\ss}_{\zeta}(a,b) \subset \Coh_{\leq 1}(\X)$ of $\zeta$-semistable sheaves of slope $(a,b)$.
By Proposition~\ref{prop:ZetaGammaEta_Semistables_DecompositionallyFinite}, it defines a log-able element $[\uM^{\ss}_{\zeta}(a,b)]$ in $H_{\gr}(\cC)$.
Thus we obtain a regular element
\begin{equation*}
  \eta_{\zeta, (a,b)} = (\LL-1)\log [\uM^{\ss}_{\zeta}(a,b)] \in H_{\gr,\reg}(\cC)
\end{equation*}
by Theorem~\ref{thm:NoPoles}.
Projecting this element to the semi-classical quotient $H_{\gr,\Sc}(\cC)$ and applying the integration morphism, we define DT-type invariants $J^{\zeta}_{(\beta,c)} \in \Q$ by the formula
\begin{align*}
  \sum_{\zeta(\beta,c)=(a,b)} J^{\zeta}_{(\beta,c)} z^\beta q^c \coloneq I\left(\eta_{\zeta, (a,b)} \right) \in \Q\{N(\X)\}.
\end{align*}

\subsubsection{Rank $-1$}
Let $(\beta, c) \in N_{\le 1}(\X)$, and let $(\gamma, \eta) \in \R_{>0} \times \R$ be away from any wall.
By Lemmas~\ref{thm:TFPairsHaveSimpleAutomorphisms} and \ref{thm:BoundednessOfRank1ObjectsForZetaStability}, we obtain an element $(\LL - 1)[\uP_{\zge}(\beta,c)] \in H_{\reg}(\cC)$.
Again, projecting to the semi-classical quotient $H_{\Sc}(\cC)$ and applying the integration morphism, we define integer DT-type invariants
\begin{equation*}
  DT^{\zeta, (\gamma,\eta)}_{(\beta,c)}z^{\beta}q^ct^{-[\hO_{\X}]} \coloneq I\bigl((\LL-1)[\uP_{\zge}(\beta,c)]\bigr).
\end{equation*}

Finally, we assemble these invariants into generating series
\begin{align}\label{eq:ZetaPair_Ninvariants_Generating_Series}  
  DT^{\zge}_{\beta} &\coloneq \sum_{c \in N_0(\X)} DT^{\zge}_{(\beta,c)} q^c, \\
  J^\zeta(a,b)_{\beta} &\coloneq \sum_{\substack{c \in N_0(\X) \\ \zeta(\beta,c) = (a,b)}} J^\zeta_{(\beta,c)} q^c.
\end{align}
These series are elements in smaller subrings of $\Q\{N(\X)\}$.
\begin{lem}\label{lem:ZetaPair_Ninvariants_Generating_Series_InLGammaCompletion}
  We have $DT^{\zeta, (\gamma, \eta)}_{\beta} \in \Z[N_{0}(\X)]_{L_{\gamma}}$ and $J^\zeta(\gamma,\eta)_{\beta} \in \Q[N_{0}(\X)]_{L_{\gamma}}$.
\end{lem}
\begin{proof}
  The first claim follows from Proposition~\ref{thm:BoundednessOfRank1ObjectsForZetaStability}.
  The second claim follows from part 2 of Proposition~\ref{prop:ZetaGammaEta_Semistables_DecompositionallyFinite}.
\end{proof}

\subsection{The limit as \texorpdfstring{$\gamma \to \infty$}{gamma goes to infinity}}
We describe the limit invariants as $\gamma$ becomes large.
\begin{lem}\label{thm:ZetaEpsilonIsSigmaInfinity}
  Let $\beta \in N_{1}(\X)$, and let $\gamma > \max_{\gamma' \in V_{\beta}} \gamma'$.
  An object $E \in \A$ of class $(-1,\beta,c)$ is a $(\Tz,\Fz)$-pair if and only if it is a stable pair.
  In particular,
  \begin{equation}
    DT^{\zeta,(M,\eta)}_{(\beta,c)} = PT(\X)_{(\beta,c)}
  \end{equation}
  for all $\eta \in \R$ and for all $M$ large enough.
\end{lem}
\begin{proof}
  Recall that $\T_{PT} = \Coh_0(\X)$, that $\F_{PT} = \Coh_1(\X)$, and that $(\T_{PT},\F_{PT})$-pairs are precisely stable pairs.
  If $\gamma > \max_{\gamma' \in V_{\beta}} \gamma'$ and $E \in \Coh_{\le 1}(\X)$ with $\beta_{E} \le \beta$, then $\zeta(E) \ge (\gamma,\eta)$ if and only if $E \in \Coh_{0}(T)$.
  This implies that $\Tz \cap \Coh_{\le 1}(\X)_{\le \beta} = \T_{PT}$ and $\Fz \cap \Coh_{\le 1}(\X)_{\le \beta} = \F_{PT} \cap \Coh_{\le 1}(\X)_{\le \beta}$, which gives the claim.
\end{proof}

\subsection{Crossing the \texorpdfstring{$\gamma$}{gamma}-wall}
Let $\beta \in N_1(\X)$ be a curve class.
We analyse what happens to the generating series $DT^{\zeta, (\gamma, \eta)}_{\leq \beta} \in \Z[N(\X)]_{L_{\gamma}}$ of $(\Tz,\Fz)$-pair invariants when $\gamma$ crosses a wall in $V_\beta$.

To describe this wall-crossing, we first show that by choosing the ample classes $(A,\omega) \in \N^{1}(X) \times \N^{1}(Y)_{\R}$ to be sufficiently general, there is a unique curve class $\beta_{\gamma}$ for which the invariants $J^{\zeta}_{(\beta_{\gamma},c)}$ contribute to the wall-crossing formula.

Recall the function $L_{\gamma}(c) = \deg(c) + \gamma^{-1}\deg_{Y}(\ch_{2}(\Psi(c)) \cdot \omega)$ for $c \in N_0(\X)$.
\begin{lem}
  \label{thm:OmegaCanBeChosenGenerically}
  If $A \in N^{1}(X)$ is general and $\omega \in N^{1}(Y)_{\R}$ is very general, then for each $\gamma \in V_{\beta}$ there is, up to scaling, a unique class $\beta_{\gamma} \in N_{1} (\X)$ with $0 < \beta_{\gamma} \le \beta$ such that
  \[
    L_{\gamma}(A\cdot \beta_{\gamma}) = 0.
  \]
  The class $c_{\gamma} \coloneq \beta_{\gamma} \cdot A \in N_{0}(\X)$ is, up to scaling, the unique class such that $L_{\gamma}(c_{\gamma}) = 0$.
\end{lem}
\begin{proof}
  We first prove that if $\omega$ is very general, then there is up to scaling at most one class $0 \not= c \in N_0(\X)$ such that $L_{\gamma}(c) = 0$.
  Note that if $L_{\gamma}(c) = 0$, then $c \in T := \{c \in N_{0}(\X) \mid \Psi(c) \not\in N_{0}(Y)\text{ and }\deg(c) \not= 0\}$.
  For if $\Psi(c) \in N_{0}(Y)$, then $c$ is some multiple of the class of an unstacky point, and so $L_{\gamma}(c) = \deg(c)$ is $0$ if and only if $c = 0$.
  If $\deg(c) = 0$ and $\Psi(c) \not\in N_{0}(Y)$, then $L_{\gamma}(c) = \gamma^{-1}\deg_{Y}(\ch_{2}(\Psi(c)) \cap \omega) \not= 0$, since $\omega$ is very general.

  Define for any $c \in T$ the number
  \[
    \gamma_c(\omega) = -\frac{\deg(c)}{\deg_{Y}(\ch_{2}(\Psi(c)) \cdot \omega)},
  \]
  which is the unique number such that $L_{\gamma_{c}(\omega)}(c) = 0$.
  If $c, c' \in T$ are not proportional, then after rescaling we may assume $\deg(c) = \deg(c')$ and $\ch_{2}(\Psi(c)) \not= \ch_{2}(\Psi(c'))$.
  It then follows that the $\omega$ satisfying $\gamma_{c}(\omega) \not= \gamma_{c'}(\omega)$ form a non-empty Zariski open subset of $N^{1}(Y)_{\R}$.
  Taking $\omega$ to lie in the intersection of the countably many such subsets gives the claim.

  We now prove the uniqueness of $\beta_{\gamma}$.
  If $\beta' \in N_{1,\mr}(\X) \setminus 0$, then $A\cdot \beta' \in \Psi^{-1}(N_{0}(Y))$, and as shown above, then $L_\gamma(A \cdot \beta') \not= 0$.
  Thus the candidates for $\beta_{\gamma}$ are the elements of the set $S = \{\beta' \mid 0 < \beta' \le \beta, \beta' \not\in N_{1,\mr}(\X)\}$.
  For $\beta' \in S$, define
  \[
    \gamma_{\beta'}(A,\omega) = \gamma_{\beta' \cdot A}(\omega) = -\dfrac{\deg(\beta'\cdot A)}{\deg_{Y}(\ch_{2}(\Psi(\beta')\cdot A) \cdot \omega)} \in \R.
  \]
  Given a pair of classes $(A,\omega)$, the number $\gamma_{\beta'}(A,\omega)$ is the unique number for which $L_{\gamma_{\beta'}}(\beta'\cdot A) = 0$.

  Take $\beta', \beta'' \in S$ and assume that $\beta'$ is not proportional to $\beta''$.
  Fix a very general $\omega \in N^{1}(Y)_{\R}$.
  The locus of $A \in N^{1}(X)$ for which $\gamma_{\beta'}(A,\omega) \not= \gamma_{\beta''}(A,\omega)$ is an open algebraic subset of $N^{1}(X)$.
  We claim that it is non-empty, and then since $S$ is a finite set, taking $A$ to be an ample class in the intersection of these finitely many open subsets we get uniqueness of $\beta_{\gamma}$.

  So assume for a contradiction that $\gamma_{\beta'}(A,\omega) = \gamma_{\beta''}(A,\omega)$ for all $(A,\omega)$.  
  If $c_{1}(\Psi(\beta'))$ is proportional to $c_{1}(\Psi(\beta''))$, then rescaling $\beta'$ (which does not change $\gamma_{\beta'}$) we may assume that $c_{1}(\Psi(\beta')) = c_{1}(\Psi(\beta''))$, and so $\beta' - \beta'' \in N_{1,\mr}(\X) \setminus 0$.
  For a general $A$, by Lemma \ref{thm:PushforwardInjectiveOnOneDimClasses}, we then have $(\beta' - \beta'') \cdot A \in \Psi^{-1}(N_{0}(Y)) \setminus 0$.
  Then as argued above
  \[
    \deg((\beta' - \beta'') \cdot A) \not= 0,
  \]
  which shows that $\gamma_{\beta'}(A,\omega) \not= \gamma_{\beta''}(A,\omega)$.

  If $c_{1}(\Psi(\beta'))$ is not proportional to $c_{1}(\Psi(\beta''))$, then by Lemma \ref{thm:DivisorsCapAAreLinearlyIndependent} we may find $A$ such that $\ch_{2}(\Psi(\beta')\cdot A)$ is not proportional to $\ch_{2}(\Psi(\beta'')\cdot A)$, which implies that $\beta' \cdot A$ and $\beta'' \cdot A$ are not proportional.
  The first part of the proof then shows that $\gamma_{\beta'}(A,\omega) = \gamma_{\beta' \cdot A}(\omega) \not= \gamma_{\beta''\cdot A}(\omega) = \gamma_{\beta''}(A,\omega)$.
\end{proof}
\begin{rmk}
  Henceforth, we write $\beta_{\gamma} \in N_{1}(\X)$ for the \emph{minimal} effective class satisfying the conditions of Lemma~\ref{thm:OmegaCanBeChosenGenerically}.
\end{rmk}

\subsection{Vanishing and changing bounds}
We collect further boundedness results allowing us to apply the wall-crossing formula.
Let $\gamma \in V_{\beta}$, let $\eta \in \R$, and let $x \in \R$.
By Proposition~\ref{thm:BoundednessOfRank1ObjectsForZetaStability}, the set
  \[
    S = \bigcup_{\beta' \le \beta} \{(\beta',c) \mid L_{\gamma}(c) \le x \text{ and }\uP_{\zge}(\beta', c) \not= \varnothing \}
  \]
is finite.
Thus we may define
\[
  M^{+}_{\beta, \gamma, x} = \max_{(\beta',c) \in S} \deg(\beta',c) \qquad \text{and} \qquad
  M^{-}_{\beta, \gamma, x} = \min_{(\beta',c) \in S} \deg(\beta',c).
\] 

\begin{lem}
  \label{thm:vanishingOfUPZeta}
  Let $\beta' \le \beta$, and let $c \in N_{0}(\X)$.
  \begin{enumerate}
    \item If $\eta \le 0$ and $\deg(\beta',c) > M^{+}_{\beta, \gamma, L_{\gamma}(c)}$, then $\uP_{\zge}(\beta',c) = \varnothing$.
    \item If $\eta \ge 0$ and $\deg(\beta',c) < M^{-}_{\beta, \gamma, L_{\gamma}(c)}$, then $\uP_{\zge}(\beta',c) = \varnothing$.
  \end{enumerate}
\end{lem}
\begin{proof}
  This follows from the argument of Lemma~\ref{thm:vanishingOfUP}.
\end{proof}

\begin{lem}
  \label{thm:BoundednessOfZetaSemistable1DimensionalSheaves}
  Let $\gamma \in V_{\beta}$.
  The set
  \[
    \bigcup_{\substack{d \ge 1 \\ d \beta_{\gamma} \le \beta}} \{c \in N_{0}(\X) \mid \uM^{\ss}_{\zeta}(d\beta_{\gamma}, c) \not= \varnothing\}
  \]
  is weakly $L_{\gamma}$-bounded.
\end{lem}
\begin{proof}
  Let $F \in \M^{\ss}_{\zeta}(d\beta_{\gamma},c)$.
  We have $\zeta(F) = (\gamma, \nu(F))$ and so $F \in \F_{\zeta, (\gamma,\infty)}$.
  Note that $\zeta(F(n)) = (\gamma,\nu(F) + n)$, thus for $n \gg 0$ we have $F(n) \in \T_0 \cap \F_{\zeta, (\gamma,\infty)}$.
  An application of Lemma~\ref{thm:ZetaBoundedness1DimensionalSheaves} completes the proof.
\end{proof}

As a consequence of Lemma~\ref{thm:BoundednessOfZetaSemistable1DimensionalSheaves}, we may define the number
\[
  K_{\gamma} = \min \{L_{\gamma}(c) \mid \exists d \text{ such that }d\beta_{\gamma} \le \beta \text{ and }\M_{\zeta}^{\ss}(d\beta_{\gamma}, c)\not= \varnothing\}
\]
because $A \cdot d \beta_{\gamma} = dc_{\gamma}$ and so $L_{\gamma}(c_{\gamma}) = 0$.
This allows us to slide off the $\gamma$-wall.
\begin{lem}
  \label{prop:SendingEta_to_Infinity_Sliding_Off_Gamma_Wall}
  \label{thm:EtaStableIsEventuallyInfinityStable}
  Let $\gamma \in V_{\beta}$, let $E \in \A$ be of class $(-1,\beta, c)$, and let
  \begin{gather*}
    \eta_{\gamma, (\beta,c)}^{+} = \max\{0, \deg(\beta,c) - M^{-}_{\beta, \gamma, L_{\gamma}(c) - K_{\gamma}}\} \\
    \eta_{\gamma, (\beta,c)}^{-} = \min\{0, \deg(\beta,c) - M^{+}_{\beta, \gamma, L_{\gamma}(c) - K_{\gamma}}\}
  \end{gather*}

  If $\eta > \eta_{\gamma, (\beta,c)}^{+}$, then the following are equivalent:
  \begin{enumerate}
  \item\label{enum:deltaBig} $E$ is a $(\cat T_{\zeta, (\gamma,\eta)},\cat F_{\zeta, (\gamma,\eta)})$-pair
  \item\label{enum:PT_Big} $E$ is a $(\cat T_{\zeta, (\gamma + \epsilon, \eta)}, \cat F_{\zeta, (\gamma + \epsilon, \eta)})$-pair.
  \end{enumerate}

  If $\eta < \eta_{\gamma, (\beta,c)}^{-}$, then the following are equivalent:
  \begin{enumerate}
  \item\label{enum:deltaSmall} $E$ is a $(\cat T_{\zeta, (\gamma,\eta)},\cat F_{\zeta, (\gamma,\eta)})$-pair
  \item\label{enum:PT_Small} $E$ is a $(\cat T_{\zeta, (\gamma - \epsilon, \eta)}, \cat F_{\zeta, (\gamma - \epsilon, \eta)})$-pair.
  \end{enumerate}
\end{lem}
\begin{proof}
  We treat the claim for $\eta > \eta_{\gamma,(\beta,c)}^{+}$, the other one is handled similarly.
  
  We first show that condition (\ref*{enum:deltaBig}) is independent of the precise value of $\eta > \eta^{+}_{\gamma,(\beta,c)}$.
  The argument is a slightly refined version of that of Lemma~\ref{lem:Delta_Pairs_are_Stable_Pairs_at_Infinity}.
  By Lemma~\ref{thm:ZetaStabilityUnchangedAwayFromBigWalls}, it is enough to show that for any $\eta \in W_{\beta} \cap (\eta_{(\beta,c)},\infty)$ we have
  \[
    \uP_{\zeta,(\gamma, \eta-\epsilon)}(\beta,c) = \uP_{\zeta,(\gamma,\eta + \epsilon)}(\beta,c).
  \]
  By the wall-crossing formula, these two moduli stacks can only differ if there exists a decomposition $(\beta,c) = (\beta',c') + (\beta'', c'')$ with $\zeta(\beta'',c'') = (\gamma, \eta)$ and
  \begin{equation*}
    \uP_{\zge}(\beta',c') \neq \varnothing \neq \uM^{\ss}_{\zeta}(\beta'',c'').
  \end{equation*}
  Suppose for a contradiction that there exists such a decomposition of $(\beta,c)$.
  It follows that $L_{\gamma}(c'') \ge K_{\gamma}$, hence $L_{\gamma}(c') \le L_{\gamma}(c) - K_{\gamma}$, and so $\deg(\beta',c') \ge M_{\beta, \gamma, L_{\gamma}(c) - K_{\gamma}}^{-}$.
  But then
  \begin{align*}
    \deg(\beta,c) &= \deg(\beta',c') + \deg(\beta'',c'') \\
                  &\ge M_{\beta, \gamma, L_{\gamma}(c) - K_{\gamma}}^{-} + l(\beta'')\nu((\beta'',c'')) \\
                  &> M_{\beta, \gamma, L_{\gamma}(c) - K_{\gamma}}^{-} + \eta^{+}_{\gamma, (\beta,c)} = \deg(\beta,c)
  \end{align*}
  which is a contradiction.

  Assume that $E$ is a $(\T_{\zeta,(\gamma+ \epsilon, \eta)}, \F_{\zeta,(\gamma + \epsilon, \eta)})$-pair.
  If $S$ is a subobject of $E$, then $L_{\gamma + \epsilon}(S \cdot A) > 0$, and hence by continuity $L_{\gamma}(S \cdot A) \ge 0$.
  Taking $\eta \ge \nu(S)$, it follows that $S \in \cat F_{\zeta, (\gamma,\eta)}$.
  If $Q$ is a quotient object, then $L_{\gamma + \epsilon}(Q \cdot A) \le 0$, so
  \[
    \deg_{Y}(\ch_{2}(\Psi(Q \cdot A) \cdot \omega)) \le -(\gamma + \epsilon) \deg(Q \cdot A).
  \]
  Now either $\deg(Q \cdot A) = 0$, in which case $Q \cdot A = 0$ and so $Q \in \Coh_{0}(\X)$, or else $\deg(Q \cdot A) > 0$, which implies $L_{\gamma}(Q \cdot A) < 0$.
  In either case $Q \in \Tz$, and so $E$ is a $(\Tz, \Fz)$-pair as claimed.

  Conversely, suppose that $E$ is a $(\T_{{\zeta, (\gamma, \eta)}}, \F_{{\zeta, (\gamma, \eta)}})$-pair and let $S \in \Coh_{\leq 1}(\X)$ be a subobject of $E$.
  Then either $L_{\gamma}(A \cdot S) > 0$, in which case $L_{\gamma + \epsilon}(A \cdot S) > 0$ as well, or $L_{\gamma}(A \cdot S) = 0$ and so, as $\deg(S \cdot A) > 0$, we find
  \[
    \deg_{Y}(\ch_{2}(\Psi(S) \cdot A) \cdot \omega) < 0,
  \]
  which implies $L_{\gamma + \epsilon}(S \cdot A) > 0$.
  If $Q \in \Coh_{\le 1}(\X)$ is a quotient object of $E$, then either $Q \in \Coh_{0}(\X) \subset \T_{\zeta, (\gamma + \epsilon, \eta)}$, or else taking $\eta > \nu(Q)$, we find $L_{\gamma}(A \cdot Q) < 0$, and so $L_{\gamma + \epsilon}(A \cdot Q) < 0$.
  We conclude that $E$ is a $(\T_{\zeta,(\gamma + \epsilon, \eta)}, \F_{\zeta,(\gamma + \epsilon, \eta)})$-pair.
\end{proof}

With all the relevant lemmas in place, we obtain a wall-crossing formula relating the DT-type invariants of $(\gamma,\eta) \in \R_{>0} \times \R$ before and after an $\eta$-wall on a $\gamma$-wall. 
\begin{prop}
\label{thm:WCFormulaForZetaPairs}
  Let $\beta \in N_1(\X)$, let $\gamma \in V_{\beta}$, and let $\eta \in W_\beta$.
  The identity
  \begin{align}
    DT^{\zeta,(\gamma,\eta + \epsilon)}_{\leq \beta}t^{[-\hO_{\X}]}  = \exp(\{J^{\zeta}(\gamma,\eta)_{\leq \beta},- \}) DT^{\zeta,(\gamma,\eta - \epsilon)}_{\leq \beta}t^{-[\hO_{\X}]}
  \end{align}
  holds in $\Q[N^{\eff}_{1}(\X)]_{\le \beta}$.
\end{prop}
\begin{proof}
This follows exactly as Proposition~\ref{prop:WCusingTFpairs}.
\end{proof}

Now, we establish rationality as a $\gamma$-wall is crossed.
Define the series
\[
  DT^{{\zeta, (\gamma, \eta)}}_{\beta, c_{0} + \Z c_{\gamma}} \coloneq \sum_{k \in \Z} DT^{{\zeta, (\gamma, \eta)}}_{\beta,c_{0} + kc_{\gamma}} q^{c_{0} + kc_{\gamma}}.
\]
As in Section~\ref{sec:rationalityOfStablePairs}, we equip $\Q[N_{0}(\X)]$ with a grading by setting $\deg(q^{c}) = \deg(c)$.
This extends naturally to a grading on $\Q(N_{0}(\X))$.
\begin{lem}
  \label{thm:ZetaWallCrossingIsQuasiPolynomial}
  Let $\beta \in N_{1}(\X)$, let $c_{0} \in N_{0}(\X)$, let $\gamma \in V_{\beta}$, and let $\eta_{0} \le -l(\beta)$.
  Then the series
  \[
    DT^{{\zeta, (\gamma, \eta_{0})}}_{(\beta,c_{0}+\Z c_{\gamma})} - DT^{\zeta, (\gamma,\infty)}_{(\beta,c_{0}+ \Z c_{\gamma})}
  \]
  is rational of degree $< \deg(\beta,0) + M^{+}_{\beta, \gamma, L_{\gamma}(c_{0})} + \eta_{0} l(\beta) + l(\beta)^{2}$.
\end{lem}

\begin{proof}
  Since $\gamma \in V_\beta$, the walls for $(\Tz,\Fz)$-pairs of class $(-1,\beta',c)$ with $\beta' \leq \beta$ are given by the $\eta$-walls $W_\beta$ for Nironi stability by Corollary~\ref{cor:Walls_For_ZetaGammaEta_Pairs_On_Gamma_Wall_are_NironiWalls}.
  By Proposition~\ref{thm:WCFormulaForZetaPairs}, we have a wall-crossing formula, which when iterated yields
  \begin{equation*}
    DT^{\zeta,(\gamma,\infty)}_{\le \beta}t^{-[\hO_{\X}]} = \prod_{\eta \in W_{\beta} \cap [\eta_{0}, \infty)} \exp(\{J^{\zeta}_{\beta_{\gamma}}(\eta),-\}) DT^{{\zeta, (\gamma, \eta_{0})}}_{\leq \beta}t^{-[\hO_{\X}]},
  \end{equation*}
  where
  \begin{equation*}
    J_{\beta_{\gamma}}^{\zeta}(\eta) = \sum_{\substack{d \in \Z_{\ge 1}, c \in N_{0}(\X) \\ d\beta_{\gamma} \le \beta \\ \nu(d\beta_{\gamma},c) = \eta}} J^{\zeta}_{(d\beta_{\gamma},c)}z^{d\beta_{\gamma}}q^{c}.
  \end{equation*}

  Expanding the exponential, substituting the expression from equation~\eqref{eq:ZetaPair_Ninvariants_Generating_Series}, we collect all terms contributing to the coefficient of $z^{\beta}q^{c_0+kc_{\gamma}}t^{-[\hO_{\X}]}$ on the right hand side.
  The terms of this sum are described as follows.
  Fix an $r \geq 1$, a sequence $(d_i)_{i=1}^r$ in $\Z_{\geq 1}$, a sequence $(c_i)_{i=1}^r$ in $N_{0}(\X)$, and a class $\alpha' = (\beta',c') \in N_{\leq 1}(\X)$, satisfying
  \begin{itemize}
  \item $\beta = \beta' + \sum_{i=1}^r d_{i}\beta_{\gamma}$,
  \item $c' + \sum_{i=1}^r c_{i} \equiv c_{0} \pmod{c_{\gamma}}$,
  \item $\eta_{0} \le \nu(d_{1}\beta_{\gamma},c_{1}) \le \cdots \le \nu(d_{r}\beta_{\gamma},c_{r})$,
  \item $J^{\zeta}_{(d_{i}\beta_{\gamma},c_{i})} \not= 0$ for all $1 \leq i \leq r$,
  \item $DT^{\zeta, (\gamma, \eta_{0})}_{\alpha'} \not= 0$.
  \end{itemize}
  The non-zero term in the coefficients of $z^{\beta}q^{c_0+kc_\gamma}t^{-[\hO_{\X}]}$ associated with this data is
  \begin{align*}
    \label{eqn:Contribution}
    T(r,(d_{i}),(c_{i}),\alpha')z^{\beta}q^{c'+\sum c_{i}}s &= A_{(d_{i}),(c_{i})}\{J^{\zeta}_{(d_{r}\beta_{\gamma},c_{r})}z^{d_{r}\beta_{\gamma}}q^{c_{r}},-\}\circ \\
                                                            &\cdots \circ \{J^{\zeta}_{(d_{1}\beta_{\gamma},c_{1})}z^{d_{1}\beta_{\gamma}}q^{c_{1}}, -\}(DT^{{\zeta, (\gamma, \eta_{0})}}_{(\beta',c')}z^{\beta'}q^{c'}t^{-[\hO_{\X}]})
  \end{align*}
  where $A_{(d_i),(c_i)}$ is a factor arising from the exponential:
  \begin{equation*}
    A_{(d_i),(c_i)} = \prod_{\nu \in W_{\beta}} \frac{1}{|\{i \mid \nu(d_{i}\beta_{\gamma},c_{i}) = \nu\}|!}.
  \end{equation*}
  Putting all these terms together, we have
  \begin{equation*}
    \sum_{k \in \Z} \left(DT^{\zge}_{(\beta,c_0+kc_{\gamma})} - DT^{\zeta,(\gamma,-\infty)}_{(\beta,c_0+kc_{\gamma})}\right)q^{c_0+kc_{\gamma}} = \sum_{r,(d_{i}),(c_{i}),\alpha'} T(r,(d_{i}),(c_{i}),\alpha')q^{c' + \sum c_{i}}.
  \end{equation*}
  We now analyse the $T$-terms.
  Expanding the Poisson brackets yields
  \begin{align*}
    T(r,(d_{i}),(c_{i}),\alpha') = A_{(d_{i}),(c_{i})} B_{(d_{i}),(c_{i}),\alpha'}
    \prod_{i=1}^{r} J^{\zeta}_{(d_{i}\beta_{\gamma},c_{i})}DT^{{\zeta, (\gamma, \eta_{0})}}_{\alpha'}
  \end{align*}
  where, letting $\alpha_{i} = (d_{i}\beta_{\gamma},c_{i})$ and $\alpha' = (\beta',c')$, we have
  \begin{equation*}
    B_{(d_{i}),(c_{i}),\alpha'} = \sigma^{\sum_{i < j}\chi(\alpha_{j},\alpha_{i}) + \sum_{i} \chi(\alpha_{i}, \alpha' - [\hO_{\X}])}\prod_{i = 1}^{r}\chi(\alpha_{i}, -[\hO_{\X}] + \alpha' + \sum_{j=1}^{i-1}\alpha_{j})
  \end{equation*}
  where $\sigma \in \{\pm 1\}$ depending on the integration morphism chosen; see Section~\ref{Section5_IntegrationMap}.

  We partition these $T$-terms in groups as follows.
  A \emph{group} consists of the data of a class $\alpha' = (\beta',c') \in N_{\leq 1}(\X)$ and a sequence of positive integers $(d_{i})_{i=1}^r$ satisfying the same conditions as above, a sequence $(\kappa_{i})_{i=1}^r \in N_{0}(\X)/\Z (d_{i}\beta_{\gamma} \cdot A)$, and a subset $E \subseteq \{1, \ldots, r-1\}$.
  The $\kappa_{i}$ are required to satisfy
  \begin{equation}
    J^{\zeta}_{(d_{i}\beta_{\gamma}, c_{i})} \not= 0 \text{ for }c_{i} \in \kappa_{i}
    \text{ and all } i = 1,2,\ldots,r.
  \end{equation}
  Tensoring by the line bundle $A$ induces an isomorphism $\uM^{\ss}_{\zeta}(\gamma,d) \cong \uM^{\ss}_{\zeta}(\gamma,d+\gamma\cdot A)$.
  It follows that the invariant $J^{\zeta}_{(d_i\beta_\gamma,c_i)}$ is independent of the choice of representative $c_i \in \kappa_i$, thus we may define
  \[
  J^{\zeta}_{(d_{i}\beta_{\gamma},\kappa_{i})} \coloneq J^{\zeta}_{(d_{i}\beta_{\gamma},c_{i})}.
  \]

  Collecting all terms belonging to the group $(\alpha',(d_i),(\kappa_i),E)$, we obtain
  \begin{align*}
    C(\alpha',(d_{i}),(\kappa_{i}),E) = \sum_{(c_{i})} T(r,(d_{i}),(c_{i}),\alpha')q^{c' + \sum c_{i}},
  \end{align*}
  where the sum is over all $c_{i} \in N_{0}(\X)$ such that
  \begin{align}
    &c_i \in \kappa_i, \\
    \label{eqn:SlopeConstraintInequality}
    &\eta_{0} \le \nu(d_{1}\beta_{\gamma},c_{1}) \le \cdots \le \nu(d_{r}\beta_{\gamma},c_{r}), \\
    \label{eqn:SlopeConstraintEquality}
    &\nu(d_{i}\beta_{\gamma},c_{i}) = \nu(d_{i+1}\beta_{\gamma},c_{i+1}) \Leftrightarrow i \in E.
  \end{align}
  Note that for such a choice of $c_{i}$, the factor $A_{(d_{i}),(c_{i})}$ defined above depends only on $E$.
  Indeed, set $\{n_{i}\} = \{1, \ldots, r\} \setminus E$ with $n_{1}< \cdots < n_{r - |E|}$.
  Then
  \begin{equation*}
    A_{E} := \prod \frac{1}{(n_{i}-n_{i-1})!} = A_{(d_{i}),(c_{i})}.
  \end{equation*}
  We find that the contribution of the group $(\alpha',(d_i),(\kappa_i),E)$ is
  \begin{equation*}
    C(\alpha',(d_{i}),(\kappa_{i}),E) = A_E \prod_{i=1}^r J^\zeta_{(d_i\beta_\gamma,\kappa_i)}
    DT^{{\zeta, (\gamma, \eta_{0})}}_{\alpha'} \left( \sum_{c_i} B_{(d_{i}),(c_{i}),\alpha'} q^{c'+\sum c_i} \right)
  \end{equation*}
  where the sum runs over all $c_i \in N_0(\X)$ as above.
  Now, for every choice of $(d_{i})$, $(\kappa_{i})$, and $E$, there exists a sequence $(c_{i}^{0})$ with $c_{i}^{0} \in \kappa_{i}$ which is \emph{minimal} in the sense that replacing any $c_{i}^{0}$ with $c_{i}^{0} - d_{i}c_{\gamma}$  would violate one of \eqref{eqn:SlopeConstraintInequality} and \eqref{eqn:SlopeConstraintEquality}.
  We find
  \begin{equation*}
    C(\alpha',(d_i),(\kappa_i), E) = A_{E}  \prod_{i=1}^{r} J^{\zeta}_{(d_i \beta_\gamma,\kappa_{i})} DT^{\zeta, (\gamma, \eta_{0})}_{\alpha'} \left(\sum_{a_i} B_{(d_i),(c^0_i+a_i d_i c_\gamma),\alpha'} q^{c'+\sum c_i^0+a_id_i c_\gamma}\right)
  \end{equation*}
  where the sum is over the set $S_E = \{0 \leq a_1 \leq a_2 \leq \ldots \leq a_r \mid a_i \in \Z,\, a_i = a_{i+1} \Leftrightarrow i \in E\}$.
  Note that the coefficients of this expression depend quasi-polynomially on the $a_i$.

  Now observe by Lemma \ref{thm:SumDefinedByInequalitiesIsARationalFunction} that this $C$ is a rational function.
  Moreover if $\deg(C)$ denotes the degree of $C$ as a rational function, Lemma~\ref{thm:degreeClaimForRationalFunction} also implies that
  \[
    \deg(C) < \deg(\beta,0) + \deg(c' + \sum c_{i}^{0}).
  \]
  We have $\deg(c') \le M^{+}_{\beta, \gamma, L_{\gamma}(c')}$, and arguing as in Section \ref{sec:PTDuality} we find that
  \[
    \deg(\sum c_{i}^{0}) \le l(\beta) \eta_{0} + l(\beta)^{2},
  \]
  and thus
  \[
    \deg(C) < \deg(\beta,0) + M^{+}_{\le \beta, L_{\gamma}(c')} + l(\beta)\eta_{0} + l(\beta)^{2}.
  \]
  This completes the proof.
\end{proof}

Combining Lemmas~\ref{thm:EtaStableIsEventuallyInfinityStable} and \ref{thm:ZetaWallCrossingIsQuasiPolynomial}, and letting $\eta_0$ go to $-\infty$, we obtain
\begin{cor}\label{thm:DTDeltaIsQuasipolynomial}
  Let $\beta \in N_{1}(\X)$, let $c_{0} \in N_{0}(\X)$, and let $\gamma \in V_{\beta}$ be a wall for $\beta$.
  The functions $DT^{\zeta,(\gamma,\infty)}_{\beta, c_{0} + \Z c_{\gamma}}$ and $DT^{\zeta,(\gamma,-\infty)}_{(\beta,c_{0} + \Z c_{\gamma})}$ are equal as rational functions.
\end{cor}
\begin{proof}
  Argue as in Lemma \ref{thm:DTMinusInfinityIsDTInfinity}, using Lemmas \ref{prop:SendingEta_to_Infinity_Sliding_Off_Gamma_Wall} and \ref{thm:ZetaWallCrossingIsQuasiPolynomial}.
\end{proof}

\begin{thrm}
  \label{thm:CrossingBigWallPreservesRationalFunction}
  Let $\beta \in N_{1}(\X)$, let $\gamma \in \R_{> 0} \setminus V_{\beta}$, and let $\eta \in \R$.
  Then $DT^{\zeta, (\gamma, \eta)}_\beta$ is the expansion of the rational function $f_{\beta}(q)$ in $\Z[N_{0}(\X)]_{L_{\gamma}}$.
\end{thrm}
\begin{proof}
  The set of walls is finite by Lemma~\ref{thm:ZetaStabilityUnchangedAwayFromBigWalls}, so we may write $V_{\beta}=\{\gamma_{i}\}_{i=1}^{n}$ where
  \begin{equation}
    0 < \gamma_{1} < \gamma_{2} < \ldots < \gamma_{n}.
  \end{equation}
  By Lemma~\ref{lem:ZetaPair_Ninvariants_Generating_Series_InLGammaCompletion}, we know that $DT^{\zeta, (\gamma, \eta)}_{\beta} \in \Z[N_{0}(\X)]_{L_{\gamma}}$.
  What remains to be shown is that it is an expansion of the rational function $f_{\beta}(q)$ of Theorem~\ref{thm:PTIsRational}.

  Lemmas~\ref{cor:Delta_Pairs_are_Stable_Pairs_at_Infinity}, \ref{thm:ZetaStabilityUnchangedAwayFromBigWalls}, and \ref{thm:ZetaEpsilonIsSigmaInfinity} prove the claim when $\gamma > \gamma_{n}$.
  Moreover, by Lemma~\ref{thm:ZetaStabilityUnchangedAwayFromBigWalls} it suffices to prove the claim for $\gamma_{-} = \gamma_{i} - \epsilon$ under the assumption that the claim is true for $\gamma_{+} = \gamma_{i} + \epsilon$.
  By Lemma~\ref{thm:OmegaCanBeChosenGenerically} then, there is up to scale a unique class $c_{\gamma_{i}} \in N_{0}(\X)$ such that $L_{\gamma_{i}}(c_{\gamma_{i}}) = 0$, $L_{\gamma_{-}}(c_{\gamma_i}) > 0$, and $L_{\gamma_{+}}(c_{\gamma_i}) < 0$.

  By induction, the series $DT^{\zeta,(\gamma_{+},\eta)}_{\beta}$ is the expansion of $f_{\beta}(q)$ in $\Z[N_0(\X)]_{L_{\gamma_{+}}}$.
  By Proposition~\ref{prop:SendingEta_to_Infinity_Sliding_Off_Gamma_Wall}, we have the equality of coefficients
  \begin{equation}
    DT^{\zeta, (\gamma_{\pm}, \eta)}_{(\beta,c_0+kc_{\gamma_i})} = DT^{\zeta, (\gamma_i,\pm \infty)}_{(\beta,c_0+kc_{\gamma_i})}.
  \end{equation}
  Thus it follows by Corollary~\ref{thm:DTDeltaIsQuasipolynomial} that the difference
  \begin{equation}
    DT^{\zeta, (\gamma_{+}, \eta)}_{(\beta,c_{0} + kc_{\gamma_i})} - DT^{\zeta, (\gamma_{-}, \eta)}_{(\beta,c_{0} + kc_{\gamma_i})}
  \end{equation}
  is quasi-polynomial in $k$. 
  Finally, by Lemma~\ref{thm:QuasipolynomialDifferenceIsANewExpansion} we may conclude that the series $DT^{\zeta, (\gamma_{-}, \eta)}_{\beta}$ is the re-expansion of $DT^{\zeta, (\gamma_{+}, \eta)}_{\beta}$ in $\Z[N_0(\X)]_{L_{\gamma_{-}}}$.
\end{proof}

\section{Recovering Bryan--Steinberg invariants}
\label{sec:comparisonBS}

In this section, we relate the end product of the $\gamma$-wall-crossing, the notion of $(\Tzge,\Fzge)$-pair as $\gamma \to 0$, to stable pairs relative the crepant resolution $f \colon Y \to X$, which we briefly recall.
By Theorem~\ref{thm:CrossingBigWallPreservesRationalFunction}, this completes the proof of the crepant resolution conjecture.

As before $\X$ denotes a CY3 orbifold that satisfies the hard Lefschetz condition.
By the McKay correspondence, its coarse moduli space $X$ has a distinguished crepant resolution $f \colon Y \to X$, and $\dim f^{-1}(x) \leq 1$ for all $x \in X$; see Section~\ref{subsubsec:CrepantResolutions}.

We denote the inverse to the McKay equivalence $\Phi \colon D(Y) \to D(\X)$ by $\Psi$.
It commutes with the natural pushforwards $g_* \circ \Phi = Rf_*$, where $g \colon \X \to X$, and sends $\Phi(\hO_Y) = \hO_{\X}$.
Recall that $\Psi(\Coh(\X)) = \Per(Y/X)$ by \cite[Thm~1.4]{MR3419955}, see Section~\ref{subsubsec:CrepantResolutions}.
Finally, we write $\Per(Y) \coloneq \Per(Y/X)$ and $\Per_{\leq 1}(Y) \coloneq \Psi(\Coh_{\leq 1}(\X))$.

\subsection{Bryan--Steinberg pairs}
We recall the notion of f-stable pair from \cite{MR3449219}.
Let $(\T_f,\F_f)$ be the torsion pair on $\Coh_{\leq 1}(Y)$ defined by
\begin{equation}
    \T_{f} = \{F \in \Coh_{\leq 1}(Y) \mid \R f_{*} F \in \Coh_{0}(X)\}
\end{equation}
where $\F_f = \T_{f}^{\perp}$ denotes the orthogonal complement.
\begin{defn}\label{def:BSpairs}
  An \emph{$f$-stable pair} or \emph{Bryan--Steinberg pair} $(F,s)$ consists of $F \in \Coh_{\leq 1}(Y)$ and a section $s \in H^0(Y,F)$.
  This data satisfies two stability requirements:
  \begin{enumerate}
    \item[(i)] $\coker(s)$ pushes down to a zero-dimensional sheaf, \ie, $\coker(s) \in \T_f$, and
    \item[(ii)] $G$ admits no maps from such sheaves, \ie, $\Hom(\T_f,F) = 0$,
  \end{enumerate}
  Two f-stable pairs $(F,s)$ and $(F',s')$ are isomorphic if there exists an isomorphism $\phi \colon F \to F'$ such that $\phi \circ s = s'$.
\end{defn}

By Proposition \ref{prop:CohomCrit}, the $f$-stable pair $(F,s)$ corresponds to the $(\T_f,\F_f)$-pair $(s \colon \hO_Y \to F)$ in the abelian category $\A_Y = \langle \hO_Y[1], \Coh_{\leq 1}(Y) \rangle_{\text{ex}}$.
For a numerical class $(\beta, n) \in N_{\le 1}(Y) = N_1(Y) \oplus \Z$, let $\uP_{BS}(\beta,n)$ denote the moduli stack of $f$-stable pairs of class $(-1,\beta,n) \in \Z \oplus N_{\leq 1}(Y)$.
It is a $\C^*$-gerbe over its coarse space by Propositions \ref{thm:TFOpenImpliesPOpen} and \ref{thm:TFPairsHaveSimpleAutomorphisms}, and of finite type by the boundedness results of \cite{MR3449219}.
The BS invariant of class $(\beta,n)$ is then defined as 
\begin{equation*}
  BS(Y/X)_{(\beta,n)} = e_B(\uP_{BS}(\beta,n))
\end{equation*}
and this definition agrees with that of the original invariants of \cite{MR3449219}.

The McKay equivalence sends f-stable pairs into the full subcategory $\A \subset D(\X)$.
\begin{lem}\label{lem:BSpairs_in_A}
  If $E = (\hO_Y \xrightarrow{s} F)$ is a $(\T_f,\F_f)$-pair in $\A_Y$, then $\Phi(E) \in \A$.  
\end{lem}
\begin{proof}
    The pair $E$ fits into an exact triangle $\hO_Y \xrightarrow{s} F \to E \to \hO_Y[1]$ in $D(Y)$, which $\Phi$ sends to an exact triangle $\hO_\X \to \Phi(F) \to \Phi(E) \to \hO_\X[1]$ in $D(\X)$.
    The claim follows by extension-closure of $\A$, since $\Phi(F) \in \Coh_{\leq 1}(\X)$ by \cite[Prop.~18]{MR3449219}.
\end{proof}

\subsection{The crepant resolution conjecture}
The remainder of this section concerns the proof of the following comparison result, and thus establishes Theorem~\ref{thm:MainTheorem}.
\begin{prop}\label{bspairsarezetapairs}
  Let $(\beta, n) \in N_{\leq 1}(Y)$.
    Restriction induces an isomorphism
    \begin{equation}
      \Phi \colon \uP_{BS}(\beta,n) \cong \uP_{\zge}(\Phi(\beta,n))
    \end{equation}
    provided that $0 < \gamma < \min_{\gamma' \in V_{\beta}} \gamma'$.
\end{prop}
We begin by introducing a torsion pair on $\Coh_{\leq 1}(\X)$ which is the limit of the torsion pairs $(\Tzge,\Fzge)$ as $\gamma \to 0^{+}$.
\begin{defn}\label{def:Tinfinity_Finfinity}
    Let $\Ti \subset \Coh_{\leq 1}(\X)$ denote the subcategory of sheaves $T$ such that if $T \onto Q$ is a surjection in $\Coh_{\leq 1}(\X)$, then either $Q \in \Coh_0(\X)$ or
    \[
    \deg_{Y}( \ch_{2}(\Psi (Q) \cdot A) \cdot \omega) < 0.
    \]
    Let $\Fi \subset \Coh_{\leq 1}(\X)$ denote the subcategory of sheaves $F$ such that if $S \into F$ is an injection in $\Coh_{\leq 1}(\X)$, then $S$ is pure of dimension 1 and $\deg_{Y}(\ch_{2}(\Psi (S) \cdot A) \cdot \omega) \geq 0$.
\end{defn}
\begin{lem}\label{lem:TFinfinity_Torsion_Pair}
  The pair $(\Ti,\Fi)$ defines a torsion pair on $\Coh_{\leq 1}(\X)$.
\end{lem}
\begin{proof}
  It is easy to see that $\Ti$ is closed under extensions and quotients.
  Since $\Coh_{\leq 1}(\X)$ is noetherian, the claim follows by Lemma \ref{lem:TodaTorsion}.
\end{proof}
The following straightforward result shows that $(\T_{\zeta,0},\F_{\zeta,0})$ is indeed a limit.
\begin{lem}
  Let $\beta \in N_{1}(\X)$.
  If $0 < \gamma < \min_{\gamma' \in V_{\beta}} \gamma'$, then an object $E \in \A$ of class $(-1,\beta,c)$ is a $(\Ti, \Fi)$-pair if and only if it is a $(\Tz, \Fz)$-pair.
\end{lem}
\begin{lem}\label{thm:StandardExactSequenceIsPerverse}
    Let $G \in \Coh(\X)$.
    The sequence
    \begin{equation}\label{eqn:CohomologyOfPhiInverseE}
      0 \to \Phi(H^{-1}(\Psi G)[1]) \to G \to \Phi(H^{0}(\Psi G)) \to 0
    \end{equation}
    is exact in $\Coh(\X)$.
\end{lem}
\begin{proof}
  The abelian category $\Per(Y)$ has an induced torsion pair
  \[
  (\Coh(Y)[1] \cap \Per(Y), \Coh(Y) \cap \Per(Y)).
  \]
  The sequence is the image under $\Phi$ of the torsion pair decomposition of $\Psi(G)$.
\end{proof}

The following result relates the torsion pairs $(\Ti,\Fi)$ on $\Coh_{\leq 1}(\X)$ and $(\T_f,\F_f)$ on $\Coh_{\leq 1}(Y)$ under the McKay equivalence.
\begin{lem}\label{thm:MasterLemma}
    We have
    \begin{equation}\label{eqn:DescriptionOfTf}
      \T_f = \Psi(\Coh_{0}(\X)) \cap \Coh(Y).
    \end{equation}
    Moreover, we have
    \begin{align}
      \label{eqn:DescriptionOfTinfty}
      \Ti &= \Big\langle \Phi\bigl(\Per_{\leq 1}(Y)\cap\Coh(Y)[1]\bigr),\Phi(\T_f)\Big\rangle_{\eex} \\
    \label{eqn:DescriptionOfFinfty}
      \Fi &= \Phi\bigl(\Per_{\leq 1}(Y) \cap \Coh(Y) \cap \T_f^{\perp}\bigr).
    \end{align}
\end{lem}
\begin{proof}
  The inclusion $\T_f \supseteq \Psi(\Coh_{0}(\X)) \cap \Coh(Y)$ follows from the definition of $\T_f$.
  For the reverse inclusion, let $T \in \T_{f}$.
  Since $\cat T_{f} \subset \Per(Y)$, we have $\Phi(T) \in \Coh(\X)$.
  Every 1-dimensional component of the support of $T$ is contracted to a point by $f$, since otherwise its image under $f$ would be a $1$-dimensional component in the support of $\R f_{*}(T)$.
  It follows that $T$ is supported over finitely many points of $X$, and so $\Phi(T) \in \Coh_{0}(\X)$.
  This proves (\ref{eqn:DescriptionOfTf}).

  Let $T \in \Per_{\le 1}(Y) \cap (\Coh(Y)[1])$.
  If $T \onto T'$ is a surjection in $\Per_{\le 1}(Y)$, then $T' \in \Coh(Y)[1]$.
  Hence $\deg_{Y}(\ch_{2}(T' \cdot A) \cdot \omega) \le 0$, and equality occurs if and only if $T' \in \Coh_{\le 1}(Y)$.
  In that case, since $T' \in \Per_{\leq 1}(Y)$, the support of $f_{*}(T')$ must be $0$-dimensional, and so $\Phi(T') \in \Coh_{0}(\X)$.
  This proves that $T \in \Ti$, and so we find $\Phi(\Per(Y) \cap \Coh(Y)[1]) \subset \Ti$.
  Thus $\langle \Phi(\Per(Y) \cap \Coh(Y)[1]),\Phi(\cat T_{f})\rangle\subseteq\Ti$.
  
  For the reverse inclusion, let $G \in \Ti$, and let $F_{0} = H^{0}(\Psi G)$ and $F_{-1} = H^{-1}(\Psi G)$.
  By Lemma~\ref{thm:StandardExactSequenceIsPerverse}, we have the short exact sequence
  \begin{equation}
    \label{eqn:DecompositionOfG}
    0 \to \Phi F_{-1}[1] \to G \to \Phi F_{0} \to 0
  \end{equation}
  in $\Coh(\X)$.

  The surjection $G \onto \Phi F_{0}$ implies that $\Phi F_{0} \in \Ti$.
  Thus either $\Phi F_{0} \in \Coh_0(\X)$ or $\deg_{Y}(\ch_1(F_{0}) \cdot A \cdot \omega) < 0$.
  But $\ch_1(F_{0})$ is effective, forcing $\deg_{Y}(\ch_{1}(F_{0}) \cdot A \cdot \omega) \ge 0$, and thus $\Phi F_{0} \in \Coh_0(\X)$.
  By \eqref{eqn:DescriptionOfTf} then $F_{0} \in \cat T_{f}$, and so the decomposition of $G$ in \eqref{eqn:DecompositionOfG} is an extension as in (\ref{eqn:DescriptionOfTinfty}).

  For \eqref{eqn:DescriptionOfFinfty}, let $G \in \Fi$ and write $F = H^{-1}(\Psi G)$.
  By definition of $\Fi$, we have $\Phi F \in \Coh_1(\X)[-1]$ and $\ch_1(F) \cdot \omega \cdot A \leq 0$.
  But $\ch_1(F)$ is effective, hence $F \in \Coh_{\leq 1}(Y)$.
  Again it follows that $F$ is contracted, hence $\Phi F \in \Coh_{0}(\X)[-1]$.
  But then as $F[1] \into G$ and $G \in \Fi$, we get $F = 0$, and thus $\Psi G \in \Coh(Y)$.
  Finally, equation \eqref{eqn:DescriptionOfTf} implies that $\Hom(\T_f,F) = \Hom(\Phi(\T_f),\Phi(F)) = 0$ and so $G \in \Phi(\Per(Y) \cap \Coh(Y) \cap \cat T_{f}^{\perp})$.

  Conversely, if $G \in \Per_{\le 1}(Y) \cap \Coh(Y) \cap \cat T_{f}^{\perp}$, then by (\ref{eqn:DescriptionOfTinfty})
  \[
    G \in \Per_{\le 1}(Y) \cap (\Per_{\le 1}(Y) \cap \Coh(Y)[1])^{\perp} \cap \T_{f}^{\perp} = \Per_{\le 1}(Y) \cap \Psi(\Ti)^{\perp},
  \]
  and so $\Phi G \in \Ti^{\perp} \cap \Coh_{\le 1}(\X) = \Fi$ as was to be shown.
\end{proof}

Finally, we identify \emph{multi-regular} $(\Ti,\Fi)$-pairs on $\X$ with $f$-stable pairs on $Y$ under the McKay equivalence.
We prove each implication separately.
\begin{lem}
  If $E$ is a $(\Ti, \Fi)$-pair with $\beta_E \in N_{1,\mr}(\X)$, then $\Psi(E)$ is an $f$-stable pair.
\end{lem}
\begin{proof}
  Writing $E$ as an iterated extension of objects $\hO_{\X}[1]$ and $E_{1},\ldots, E_{n}$ with $E_{i} \in \Coh_{\le 1}(\X)$ shows that $H^{-1}(\Psi E)$ has rank one, that $H^{0}(\Psi E) \in \Coh_{\le 1}(Y)$, and that all other cohomology sheaves of $\Psi E$ vanish.

  We claim that $H^{-1}(\Psi E)$ is torsion free.
  Let $T \into H^{-1}(\Psi E)$ be the torsion part of the sheaf $H^{-1}(\Psi E)$, then $T[1] \in \Per(Y)$.
  By Lemma \ref{thm:MasterLemma}, we find $\Phi(T[1]) \in \Ti$.
  Since $E$ is a $(\Ti,\Fi)$-pair, it satisfies $\Hom(\Phi T[1], E) = 0$ by definition.
  But we have a chain of inclusions
  \begin{equation}
    \Hom(T,T) \into \Hom(T,H^{-1}(\Psi E)) \into \Hom(T[1],\Psi E) = \Hom(\Phi T[1], E) = 0
  \end{equation}
  forcing $T = 0$.
  We conclude that $H^{-1}(\Psi E)$ is torsion free.

  It follows that $H^{-1}(\Psi E)$ is of the form $I_C(D)$ for some 1-dimensional scheme $C \subset Y$ and some divisor $D$.
  But since $\beta_E$ is multi-regular we have $c_{1}(\Psi E) = 0$, and so $c_{1}(H^{0}(\Psi E)) = [D]$.

  We have $H^{0}(\Psi E) \in \Per_{\le 1}(Y)$.
  Since $E$ is a $(\Ti,\Fi)$-pair, we must have $\Phi(H^{0}(\Psi E)) \in \Ti$.
  By Lemma \ref{thm:MasterLemma}, this implies $H^{0}(\Psi E)) \in \T_{f}$, and so in particular $D = 0$ and $H^{>0}(Y,H^{0}(\Psi E)) = 0$.
  The criterion of Lemma \ref{prop:CohomCrit} then implies that $\Psi E$ has the form $(\hO_{Y} \to F)$ for some $1$-dimensional sheaf $F$ on $Y$.
  For any $T \in \T_f$,
  \[
  \Hom(T,F) = \Hom(T,\Psi E) = \Hom(\Phi T,E) = 0,
  \]
  using Lemma \ref{thm:MasterLemma}, and so $F \in \F_f$.
  This proves that $\Psi E$ is an $f$-stable pair.
\end{proof}

\begin{lem}
  If $E = (\hO_{Y} \to F)$ is an $f$-stable pair, then $\Phi E$ is a $(\Ti,\Fi)$-pair.
\end{lem}
\begin{proof}
  By the proof of Lemma~\ref{lem:BSpairs_in_A}, $\Phi E$ is the cone of the map $\hO_{\X} \to \Phi F$ where $\Phi F \in \Coh_{\le 1}(\X)$.
  By Lemma \ref{thm:MasterLemma}, if $T \in \Ti$, then $H^{0}(\Psi T) \in \cat T_{f}$.
  By Lemma \ref{lem:SubobjectFactors}
  \[
    \Hom(T, \Phi E) = \Hom(T, \Phi F) = \Hom(\Psi T, F) = \Hom(H^{0}(\Psi T),F) = 0,
  \]
  because $F \in \F_{f}$.
  
  Let $G \in \Fi$, then by Lemma \ref{thm:MasterLemma}, we have $\Psi G \in \Coh(Y) \cap \cat T_{f}^{\perp}$.
  This implies that
  \[
    \Hom(\Phi E, G) = \Hom(E, \Psi G) = \Hom(H^{0}(E), \Psi G) = 0,
  \]
  because $H^0(E) \in T_f$.
  We conclude that $\Phi E$ is a $(\T_{\zeta,0},\F_{\zeta,0})$-pair.
\end{proof}
\begin{proof}[Proof of Proposition~\ref{bspairsarezetapairs}]
  Collecting the above results yields the result.
\end{proof}
As a consequence, we may define the generating function of $f$-stable pair invariants of class $\beta \in N_{1,\mr}(\X)$ as the generating function of $(\Ti,\Fi)$-pairs of class $\beta$.
In turn, this is nothing but the generating function of $\zeta$-pairs of class $\beta$ for which $(\gamma,\eta) \in \R_{>0} \times \R$ satisfy $0 < \gamma < \min_{\gamma' \in V_{\beta}} \gamma'$.

Collecting our previous results, we prove the crepant resolution conjecture.
\begin{thrm}\label{thm:MainTheorem}
  Let $\X$ be a 3-dimensional Calabi--Yau orbifold satisfying the hard Lefschetz condition with projective coarse moduli space, and let $\beta \in N_{1,\mr}(\X)$ be a multi-regular curve class.
  Then the equality
  \[
  PT(\X)_{\beta} = BS(Y/X)_{\beta}
  \]
  holds as rational functions in $\Q(N_0(\X))$.

  More precisely, there exists a unique rational function $f_\beta \in \Q(N_0(\X))$ such that
  \begin{enumerate}
    \item the Laurent expansion of $f_\beta$ with respect to $\deg$ is the series $PT(\X)_\beta$, 
    \item the Laurent expansion of $f_\beta$ with respect to $L_{\gamma}$ is the series $BS(Y/X)_\beta$,
  \end{enumerate}
  where $0 < \gamma < \min_{\gamma' \in V_{\beta}} \gamma'$.
\end{thrm}
\begin{proof}
  This follows from Theorem~\ref{thm:CrossingBigWallPreservesRationalFunction} and Proposition~\ref{bspairsarezetapairs}.
\end{proof}


\appendix

\newpage

\section{An example: rational functions vs. generating series}
\label{sec:Appendix_ExampleCRC}
We present a simple example of a CY3 orbifold $\X$ for which the equality
\begin{equation}\label{eq:CRC_as_Series}
  \frac{DT_{\mr}(\X)}{DT_{0}(\X)} = \frac{DT(Y)}{DT_{\exc}(Y)}
\end{equation}
holds as an equality of rationality functions, but fails to hold as an equality of generating series.
This mismatch occurs when considering the DT invariants associated to a curve class on $\X$ that, under the McKay equivalence $\Phi \colon D(Y) \to D(\X)$, does not correspond to the class of a quotient of $\hO_Y$ on $Y$.
\begin{rmk}
We emphasise that the computation carried out here is a subset of the calculations of both \cite{MR2854183} and \cite[Thm.~2.2]{MR3595887}.
We include it not to claim fault in these works, but to demonstrate that even for the very simplest of examples, the equalities in the crepant resolution conjecture must be interpreted as being between rational functions rather than generating series.
\end{rmk}

The geometry of our example is a quasi-projective CY3 orbifold with transverse $A_1$-singularities, which we now describe.
Consider the vector bundle $Z = \hO_{\P^{1}}(-1)^{\oplus 2} \to \P^{1}$, and let $\Z/2$ act on $\Z$ via multiplication by $-1$ in each fibre.
Let $\X = [Z/(\Z/2)]$.
Since $\X$ is non-compact, we work with compactly supported $K$-theory, and we write $N(\X)$ for the numerical compactly supported $K$-group.

Let $\rho$ be the non-trivial character of $\Z/2$, and let $\hO_{\X}(\rho)$ be the corresponding line bundle on $\X$.
Let $C_{0} \subset Z$ be the zero-section.
If $F$ is the push-forward of a line bundle on $C_{0}$ to $X$, there are two lifts $F^{\pm}$ of $F$ to $\X$, that is to say there are two lifts of $F$ to $\Z/2$-equivariant sheaves on $X$.
These satisfy $F^{\pm} \cong \hO(\rho) \otimes F^{\mp}$, and we label them so that $F^{+}$ locally has $\Z/2$-invariant sections and $F^{-}$ does not.
Similarly, if $p \in C_{0}$, then the skyscraper sheaf $\hO_{p}$ has two lifts $\hO_{p}^{\pm}$ to sheaves on $\X$.

A natural basis for the numerical K-group of multi-regular classes $N_{\mr}(\X)$, defined above equation~\eqref{tikz:NumericalDiagram}, is then given by the three classes $\beta = [\hO_{C}^{-}(-1)] + [\hO^{+}_{C}(-1)], [\hO^{+}_{p}], [\hO_{p}^{-}]$.
We use the shorthand
\[
  (d, (m,n)) = d\beta + m[\hO_{p}^{+}] + n[\hO_{p}^{-}] \in N_{\mr}(\X).
\]
The elements of $\Z[N_{\le 1}(\X)]$ corresponding to $\beta, [\hO_{p}^{+}], [\hO_{p}^{-}]$ are denoted by $z, q_{+}, q_{-}$, respectively.

The coarse moduli space $X$ of $\X$ is a family of singular quadric cones over $\P^{1}$.
The distinguished crepant resolution $Y$, given by the McKay correspondence, is the blow up of $X$ in the singular locus $\P^{1}$, so that $Y$ is the total space of the line bundle $\hO_{\P^{1} \times \P^{1}}(-2,-2) \to \P^{1} \times \P^{1}$.

We let $C_{h}$ and $C_{v}$ be orthogonal lines in $\P^{1} \times \P^{1} \subset Y$ such that $C_{v}$ is contracted by $f \colon Y \to X$ and $C_{h}$ is not.
We denote their numerical classes by $\beta_{h} = [\hO_{C_{h}}(-1)]$, $\beta_{v} = [\hO_{C_{v}}(-1)]$ and we let $[p] \in N_{0}(Y)$ be the class of a point.
Under the McKay equivalence, we have $[C_{h}] = (1, (0,1))$, $[C_{v}] = (0, (0,1))$ and $[p] = (0, (1,1))$. 

In the following result, we collect some DT invariants of $\X$ and $Y$.
\begin{prop}\label{prop:AppendixA_computation}
  We have
  \begin{itemize}
  \item $DT(Y)_{(0, (m, n))} = 0$ if $n \le 0$.
  \item $\sum_{m \in \Z} DT(\X)_{(0,(m,0))}q_{+}^{m} = \sum_{m \ge 0} (-1)^{m}(m + 1)q_{+}^m = (1+ q_{+})^{-2}$.
  \item $DT(Y)_{(2\beta, (m,n))} = DT(\X)_{(2\beta, (m,n))} = 0$ if $n < 4$.
  \end{itemize}
  Setting $n = 4$, we have
\begin{align*}
\sum_{m \in \Z} DT(Y)_{(2\beta,(m,4))}z^{2}q_{+}^{m}q_{-}^{4} &= \sum_{m \le 2} (-1)^{m+1}(3m-9)z^{2}q_{+}^{m}q_{-}^{4} = \frac{3z^{2}q_{+}^{4}q_{-}^{4}}{(1+ q_{+})^{2}} \\
\sum_{m \in \Z} DT(\X)_{(2\beta,(m,4))}z^{2}q_{+}^{m}q_{-}^{4} &= 3\sum_{m \ge 4} \binom{m+3}{3} z^{2}q_{+}^{m}q_{-}^{4} = \frac{3z^{2}q_{+}^{4}q_{-}^{4}}{(1+ q_{+})^{4}}
\end{align*}
\end{prop}
\begin{proof}
  There is a natural action of $T = (\C^{*})^{3}$ on both $\X$ and $Y$, and the computation of DT invariants thus reduces to counting the $T$-fixed points with their Behrend weights.
  By \cite[Thm.~3.4]{MR2407118}, these are computable as
  \[
    \nu_M(p) = (-1)^{\dim T_{M, p}}.
  \]

  We omit the details of this fixed-point counting.
  However, we remark that the Behrend weight calculation is particularly simple in the cases we consider since the moduli spaces involved are in fact smooth, and so $\nu_M \equiv (-1)^{\dim M}$.

  To see smoothness on the side of the resolution $Y$, one shows that objects in $\Quot(Y, \hO_{Y})_{(2\beta,(m,4))}$ correspond to divisors on $\P^{1}\times \P^{1}$ of class $2C_{h} + (2-m)C_{v}$.
  In particular,
  \[
  \Quot(Y)_{(2\beta,(m,4)} \cong \P^{2} \times \P^{2-m}.
  \]
  To see smoothness on the side of the orbifold $\X$, we remark that $\Quot(\hO_{X})_{0,(m,0)} \cong \Sym^{m}\P^{1} = \P^{m}$, while one can show that the kernel $\hI_{C'}$ of an object $(\hO_{\X} \onto \hO_{C'}) \in \Quot(\X, \hO_{\X})_{(2\beta,(m,4))}$ must satisfy $\hI_{C}^{3} \subset \hI_{C'} \subset \hI_{C}^{2}$.
  This induces isomorphisms
  \begin{align*}
    \Quot(\X, \hO_{\X})_{(2\beta,(m,4))} &\cong \Quot(\X, \hI_{C}^{2}/\hI_{C}^{3})_{2(\beta,(m,4))} \\
        &\cong \Quot(\P^{1}, \hO_{\P^{1}}(2)^{\oplus 3})_{[\hO_{\P^{1}}(m-2)]},
  \end{align*}
  and the latter scheme is non-singular.
\end{proof}

We obtain the following DT generating series, contradicting equation~\eqref{eq:CRC_as_Series}.
\begin{cor}\label{cor:Mismatch_As_Generating_Series}
  The $z^{2}q_{+}^{m}q_{-}^{4}$-coefficient of $DT(\X)/DT_{0}(\X)$ is $0$ if $m \le 3$, and otherwise is $(-1)^{m}(3m-9)$.

  The $z^{2}q_{+}^{m}q_{-}^{4}$-coefficient of $DT(Y)/DT_{\exc}(Y)$ is $0$ if $m \ge 3$, and otherwise is $(-1)^{m+1}(3m-9)$.
\end{cor}
\begin{rmk}
  Note that the difference between the generating series collecting, for all $m \in \Z$, the $z^2q_{+}^mq_{-}^4$-terms on $\X$ and the series collecting those terms on $Y$, is
  \[
  z^2q_{-}^4 \sum_{m \in \Z} a_m q_{+}^m = z^2q_{-}^4\sum_{m \in \Z}(-1)^m(3m-9) q_{+}^m .
  \]
  In particular, the coefficient of the difference $a_m = (-1)^m(3m-9)$ is quasi-polynomial.
  By Lemma~\ref{thm:QuasipolynomialDifferenceIsANewExpansion}, these two generating series are the expansions at $q_{+} = 0$ and $q_{+}^{-1} = 0$, respectively, of the same rational function; this follows directly from Proposition~\ref{prop:AppendixA_computation}.
\end{rmk}


\section{Open hearts give exact linear algebraic stacks}
\label{sec:BRaxioms}
All our schemes are defined over $\C$.
Let $Y$ be a smooth, projective variety (these hypotheses are likely far stronger than necessary).
Recall that $\Mum_{Y}$ is the stack such that objects of $\Mum_{Y}(S)$ correspond to perfect objects $\hE \in D^{b}(S \times Y)$ which are universally gluable, \ie~for any scheme morphism $f \colon S' \to S$, we have $\Hom(f^{*}\hE, f^{*}\hE[i]) = 0$ if $i \le -1$.
Here, we write $f^*$ for the more correct $(f \times \ID_Y)^*$.

We want to show that a reasonable substack $\A \into \Mum_{Y}$ satisfies the axioms used in \cite{1612.00372} to analyse its Hall algebra.
Our definition of reasonable is that $\A$ is open as a substack, and as a category is closed under extensions, sums and summands, with all negative degree Ext-groups vanishing.
More precisely:
\begin{ass}
  \label{categoryAssumption}
  Assume that $\A \into \Mum_{Y}$ is an open substack such that for any scheme $S$, the full subcategory of $D^{b}(S \times Y)$ whose objects are $\A(S)$ is closed under extensions, sums and summands, and if $\hE_{1}, \hE_{2} \in \A(S)$, then $\Hom(\hE_{1}, \hE_{2}[i]) = 0$ for $i \le -1$.
\end{ass}

\begin{rmk}\label{redundancyRemark}
  There is some redundancy in the assumptions above.
  In fact, it is enough to require that $\A(S)$ be closed under extensions and summands.
  Being closed under sums then follows from being closed under extensions, by the trivial extension $\hE_{1} \to \hE_{1} \oplus \hE_{2} \to \hE_{2}$.
  Moreover, $\A(S)$ being closed under sums implies that the sum of any two objects of $\A(S)$ is universally gluable, which implies the condition of vanishing of Ext-groups of negative degree.
\end{rmk}

\begin{rmk}
  Assumption \ref{categoryAssumption} implies that direct sum defines a morphism $\oplus \colon \A \times \A \to \A$.
  There is no such morphism on the bigger stack $\Mum_{Y}$, because the sum of two universally gluable complexes is not in general universally gluable.
\end{rmk}

We wish to apply the results of \cite{1612.00372} to the Hall algebra of $\A$.
In order to do so, we must show that we can give $\A$ the structure of an ``exact linear algebraic stack''.
For precise definitions of this and related terms we refer to \cite{1612.00372}.
Roughly speaking, the data of a linear algebraic stack consists of the following:
\begin{itemize}
\item a stack $\A$, which is algebraic locally of finite type,
\item a stack
  \[
    \Hom(-,-) \to \A \times \A,
  \]
  along with certain composition maps defining an $\hO_{S}$-linear category with the same objects as $\A(S)$, and whose underlying groupoid is $\A(S)$.
  We require that locally on $\A \times \A$ the functor $\Hom(-,-)$ is \emph{coherent representable}, i.e., equal to the kernel of a homomorphism of finite rank locally free sheaves.
\end{itemize}
To say that $\A$ is moreover \emph{exact} linear algebraic \cite[Sec.~3]{1612.00372} means that we have a substack $\A^{(2)}$ of the stack of all 3-term sequences in $\A$, with morphisms $a_{1}, a_{2}, a_{3} \colon \A^{(2)} \to \A$ taking a sequence to its first, second and third term, such that
\begin{itemize}
\item for every scheme $S$, the objects of $\A^{(2)}(S)$ define the structure of Quillen exact category on $\A(S)$,
\item the stack $\A^{(2)}$ is linear algebraic,
\item the morphism $a_{2} \colon \A^{(2)} \to \A$ is representable as a morphism of algebroids in the sense of \cite[Def.~1.49]{1612.00372},
\item the morphism $a_{1} \times a_{3} \colon \A^{(2)} \to \A \times \A$ is of finite type, and
\item for every scheme $S$, the category $\A(S)$ is Karoubian.
\end{itemize}

The main result of this appendix is the following.
\begin{prop}\label{thm:AppendixProposition}
If $\A$ satisfies Assumption \ref{categoryAssumption}, then the data of Hom-spaces and exact triangles in the full subcategory $\A(S) \subset D^{b}(S \times Y)$ give $\A$ the structure of an exact linear algebraic stack.
\end{prop}
The proof is the combination of Lemmas \ref{thm:AIsLinearAlgebraic}, \ref{thm:LinearAlgebraicMorphismStacks}, \ref{thm:Karoubian}, \ref{thm:representable}, \ref{thm:finiteType} and Corollary \ref{thm:QuillenExact}.
These results generalise those for the category $\A(\C) = \Coh(Y)$, in which case the verification of the hard parts can be found in \cite{MR2854172}.

\subsection{Preliminaries on Hom-spaces}
If $f \colon S \to \Mum_{Y}$ is a morphism from a scheme $S$, we will write $f^{*}\hU$ for the corresponding object in $D^{b}(S \times Y)$.
To avoid getting into the definition of derived categories on general algebraic stacks, we think of this as strictly formal notation, and avoid positing the existence of a universal object $\hU$ in some derived category of $\Mum_{Y} \times Y$.

Let $S$ be a scheme, and let $\hE_{1}, \hE_{2}$ be perfect objects in $D^{b}(S \times Y)$.
Let $\pi \colon S \times Y \to S$ be the projection, and define the complex
\[
\CHom(\hE_{1}, \hE_{2}) \coloneq R\pi_{*}\lRHom(\hE_{1}, \hE_{2}),
\]
which lies in $D^b(S)$.
\begin{lem}
  Let $S$ be a scheme, and let $\hE_{1}, \hE_{2} \in D^{b}(S \times Y)$ be perfect.
  The object $\CHom(\hE_{1},\hE_{2})$ is perfect.
  For any scheme morphism $f \colon S' \to S$, we have
  \[
    \CHom(f^{*}\hE_{1}, f^{*}\hE_{2}) \cong f^{*}\CHom(\hE_{1}, \hE_{2})
  \]
  and
  \[
    \Hom(f^{*}\hE_{1}, f^{*}\hE_{2}[i]) = H^{i}(S', f^{*}\CHom(\hE_{1},\hE_{2}))
  \]
\end{lem}
\begin{proof}
  Since $\hE_{1}$ and $\hE_{2}$ are perfect, so is $\lRHom(\hE_{1}, \hE_{2})$.
  The claims now follow from \cite[\href{https://stacks.math.columbia.edu/tag/0DJT}{Lemma 0DJT}]{stacks-project}.
\end{proof}


\begin{lem}
  Let $S$ be a scheme, and let $E \in D^{b}(S)$ be a perfect complex.
  The following conditions are equivalent:
  \begin{enumerate}
  \item\label{locallyFree} Locally on $S$, we can find a finite locally free complex $K^{\bullet}$ representing $E$ with $K^{i} = 0$ if $i \le -1$.
  \item\label{totalVanishing} For every scheme morphism $f \colon S' \to S$, we have $H^{i}(f^{*}E) = 0$ if $i \le -1$.
  \item\label{affineVanishing} For every affine scheme morphism $f \colon S' \to S$, we have $H^{i}(f^{*}E) = 0$ if $i \le -1$.
  \item For every field $k$ and every morphism $f \colon \Spec k \to S$, we have $H^{i}(f^{*}E) = 0$ if $i \le -1$.
  \item\label{points} For every $x \in S$ with quotient field $k(x)$ and inclusion $i_{x} \colon \Spec k(x) \to S$, we have $H^{i}(i_{x}^{*}E) = 0$ if $i \le -1$.
  \end{enumerate}

  If $S$ is affine, then these conditions are equivalent to
  \begin{enumerate}
  \item[(\ref*{points}*)] For every closed point $x \in S$ with quotient field $k(x)$ and corresponding inclusion $i_{x} \colon \Spec k(x) \to S$, we have $H^{i}(i_{x}^{*}E) = 0$ if $i \le -1$.
  \end{enumerate}
\end{lem}
\begin{proof}
  The implications $(i) \Rightarrow (i+1)$ are all obvious, so we need to show $(\ref*{points}) \Rightarrow (\ref*{locallyFree})$.
  This follows from \cite[\href{https://stacks.math.columbia.edu/tag/0BCD}{Lem.~0BCD}]{stacks-project}, which shows that if $x \in S$ satisfies (\ref*{points}), then $E$ satisfies (\ref*{locallyFree}) in an open neighbourhood of $x$.

  Finally, if $S$ is affine, then the closure of any point $x$ contains a closed point $y$.
  If condition (\ref*{points}) holds for $y$, then by \cite[\href{https://stacks.math.columbia.edu/tag/0BDI}{Lem.~0BDI}]{stacks-project} it also holds for $x$, and so (\ref*{points}*) implies (\ref*{points}).
  This completes the proof.
\end{proof}
If any (hence all) of the conditions above are satisfied, we say that \emph{$E$ has tor-amplitude in $[0,\infty)$}.

\begin{lem}
  Let $S$ be a scheme, and let $\hE \in D^{b}(S \times Y)$ be perfect.
  The object $\hE$ is universally gluable if and only if $\CHom(\hE, \hE)$ has tor-amplitude in $[0,\infty)$.
  There exists an open subset $U \subset S$ such that for any scheme morphism $f \colon S' \to S$, the object $f^{*}\hE$ is universally gluable if and only if $f$ factors through $U$.
\end{lem}
\begin{proof}
  The object $\hE$ being universally gluable obviously implies condition (\ref*{affineVanishing}).
  Conversely, condition (\ref*{totalVanishing}) implies that $\hE$ is universally gluable.
  
  Let $V \subset S$ be an open affine.
  Then $\CHom(\hE,\hE)$ can be represented by a finite complex of finite rank projectives on $V$.
  The set $V_{j}$ of points $x \in V$ for which $H^{j}(i_{x}^{*}\CHom(\hE,\hE)) = 0$ is open for any $i$, by \cite[\href{https://stacks.math.columbia.edu/tag/0BDI}{Lem.~0BDI}]{stacks-project}.
  The set of points $x \in V$ for which (\ref*{points}) holds, is the intersection of finitely many $V_{i}$, hence is itself open.
  Thus the set of points in $S$ for which (\ref*{points}) holds is an open subscheme $U \subset S$, and $f^{*}\hE$ is universally gluable if and only if $f$ factors through $U$.
\end{proof}

\subsection{Checking the assumption on closed points}
Before giving the proof of Proposition~\ref{thm:AppendixProposition}, let us show that Assumption \ref{categoryAssumption} can be checked on $\C$-points.

\begin{lem}\label{lem:Assumptions_on_Cpoints}
  Let $\A \into \Mum_{Y}$ be an open substack such that the full subcategory of $D^{b}(Y)$ whose objects are $\A(\C)$ is closed under sums, summands and extension, and moreover $\Hom(E_{1}, E_{2}[i]) = 0$ if $E_{1}, E_{2} \in \A(\C)$ and $i \le -1$.

  Then $\A$ satisfies Assumption \ref{categoryAssumption}.
\end{lem}
\begin{proof}
  Let $S$ be a scheme, and let $\hE_{1}, \hE_{2} \in D^{b}(S \times Y)$ be perfect complexes.
  By Remark~\ref{redundancyRemark}, it suffices to show the following two conditions:
  \begin{enumerate}
  \item\label{sums} if $\hE_{1} \oplus \hE_{2} \in \A(S)$, then $\hE_{1} \in \A(S) \text{ and }\hE_{2} \in \A(S)$, and
  \item\label{extensions} if $\hE_{1} \to \hE \to \hE_{2}$ is an exact triangle and $\hE_{1}, \hE_{2} \in \A(S)$, then $\hE \in \A(S)$.
  \end{enumerate}

  First, assume that $S$ is a finite type $\C$-scheme.
  Since the property of being universally gluable is open and $\A$ is open in $\Mum_{Y}$, an object $\hE \in D^{b}(S \times Y)$ lies in $\A(S)$ if and only if $i_{x}^{*}\hE \in \A(\C)$ for every $\C$-point $x \in S$.
  We thus obtain property (\ref*{sums}) by observing that for every $\C$-point $x \in S$, we have
  \[
    \hE_{1} \oplus \hE_{2} \in \A(S) \Rightarrow i_{x}^{*}\hE_{1} \oplus i_{x}^{*}\hE_{2} \in \A(\C) \Rightarrow i_{x}^{*}\hE_{1}, i_{x}^{*}\hE_{2} \in \A(\C).
  \]
  As for property (\ref*{extensions}), observe that $i_{x}^{*}\hE \in \A(\C)$ for all $x \in S$ by extension-closure of $\A(\C)$.
  It then follows that $\hE \in \A(S)$.
  
  Second, assume that $S$ is any affine scheme.
  To prove (\ref*{sums}) we argue as follows.
  Since $\hE_{1} \oplus \hE_{2}$ is universally gluable, so are $\hE_{1}$ and $\hE_{2}$.
  So let $f_{1}, f_{2} \colon S \to \Mum_{Y}$ be such that $\hE_{i} \cong f_{i}^{*}\hU$.
  By {\cite[\href{https://stacks.math.columbia.edu/tag/0CMY}{Prop.~0CMY}]{stacks-project}} and the fact that $\Mum_{Y}$ is locally of finite type, there exists a finite type affine $S'$ such that the $f_i$ factor as
  \[
  S \stackrel{g}{\to} S' \stackrel{f_{i}'}{\to} \Mum_Y.
  \]
  Let $U \subset S'$ be the open subset on which $(f_{1}')^{*}\hU \oplus (f_{2}')^{*}\hU$ lies in $\A(U)$, which exists by openness of universal gluability and $\A$.
  Then $g$ factors through $U$.
  By the finite type case above, we then have $(f_{1}')^{*}\hU|_{U}, (f_{2}')^{*}\hU|_{U} \in \A(U)$, and so $f_{1}^{*}\hU, f_{2}^{*}\hU \in \A(S)$.

  To prove (\ref*{extensions}), note that the exact triangle corresponds to a class in
  \[
  \Hom(\hE_{2}, \hE_{1}[1]) = H^{1}(S, \CHom(\hE_{1},\hE_{2})).
  \]
  Since $\CHom(\hE_{2}, \hE_{1})$ is perfect, it is representable by a finite complex of finitely generated projectives $K^{\bullet}$, and the extension class defining $\hE$ comes from an element of $K^{1}$.
  We can find a finitely generated subring $R' \subset R = \Gamma(S,\hO_{S})$ such that $\hE_{1}$, $\hE_{2}$ and this Ext-class are all defined over $R'$, which means that the triangle is pulled back from $S' = \Spec R'$.
  Applying property (\ref*{extensions}) from the finite type case above then proves the claim.

  Finally, the openness of universal gluability and of $\A$ shows that a complex $\hE \in D^{b}(S \times Y)$ lies in $\A(S)$ if and only $i_{x}^{*}\hE \in \A(\Spec k(x))$ for every $x \in S$.
  We may then argue as in the finite type case, replacing $\C$ by the fields $k(x)$, where we have just shown that (\ref*{sums}) and (\ref*{extensions}) hold.
\end{proof}

\begin{cor}
  If $\A$ is an open substack such that the objects of $\A(\C)$ define the heart of a t-structure on $D^{b}(Y)$, then $\A$ satisfies the assumptions of Lemma~\ref{lem:Assumptions_on_Cpoints}.
\end{cor}


\subsection{Proof of main proposition}
We now assume that $\A$ satisfies Assumption \ref{categoryAssumption}.
Let $S$ be a scheme, and let $f_{1}, f_{2} \colon S \to \A$ be morphisms.
Let
\[
\CHom(f_{1},f_{2}) = \CHom(f_{1}^{*}\hU, f_{2}^{*}\hU).
\]


The algebraic stack $\A$ admits an obvious refinement to a linear stack, in the sense of \cite[Def.~1.9]{1612.00372}.
Explicitly, given schemes $T_{1}, T_{2}$ with morphisms $g \colon T_{1} \to T_{2}$ and $f_{1}, f_{2} \colon T_{i} \to \A$ corresponding to objects $\hE_{1}, \hE_{2}$, we set
\[
\Hom_{g}(f_{1}, f_{2}) = \Hom(\hE_{1}, g^{*}(\hE_{2})),
\]
where this Hom-space has the obvious $\hO$-linear structure.

\begin{lem}
  \label{thm:AIsLinearAlgebraic}
  The stack $\A$ is an algebraic linear stack.
\end{lem}
\begin{proof}
  The stack $\A$ is algebraic in the usual sense (as a category fibred in groupoids) because it is an open substack of an algebraic stack.

  Given two scheme morphisms $f_{1}, f_{2} \colon S \to \A$, we get a presheaf $\Hom(f_{1},f_{2})$ on $(\text{Sch}/S)$.
  This presheaf takes $g \colon S' \to S$ to $H^{0}(S', g^{*}\CHom(f_{1},f_{2}))$.
  Since $\CHom(f_{1},f_{2})$ has tor-amplitude in $[0, \infty)$, it is locally on $S$ of the form $0 \to K^{0} \stackrel{d}{\to} K^{1} \to \ldots$ with the $K^{i}$ locally free of finite rank.
  The presheaf $\Hom(f_{1},f_{2})$ is thus locally on $S$ given by $\ker d$, and hence is locally coherent representable.
  See \cite[\href{https://stacks.math.columbia.edu/tag/08JX}{Lem.~08JX}]{stacks-project}, which gives a similar argument in the case $\A = \Coh(Y)$.
\end{proof}

\subsection{Exact algebraic stack}
The linear stack $\A$ is naturally equipped with a structure of an exact algebraic stack, in the sense of \cite{1612.00372}.
We define $\Seq_{\A}$ to be the stack of all sequences $\hE' \to \hE \to \hE''$ in $\A(S)$, and we let $\A^{(2)}$ denote the substack of those sequences which are exact in $D(S \times Y)$.
Both of these stacks have natural structures of linear stacks, by letting Hom-spaces be the component-wise maps commuting with the morphisms in the sequence.

Since $\A$ is linear algebraic, there exists an algebraic stack $\Hom_{\A}(-,-) \to \A \times \A$, and this morphism is representable in affine schemes.
\begin{lem}
  \label{thm:LinearAlgebraicMorphismStacks}
  The stacks $\Seq_{\A}$ and $\A^{(2)}$ are algebraic linear stacks.
\end{lem}
\begin{proof}
  As ordinary stacks, $\Seq_{\A}$ is equivalent to $\Hom_{\A}(-,-) \times_{\A} \Hom_{\A}(-,-)$, and hence is algebraic.
  Let $S$ be a scheme and let $f_{1}, f_{2} \colon S \to \Seq_{\A}$ be morphisms.
  For $j = 1,2$, let $\hE_{j}^{\bullet} = (\hE_{j}^{1} \stackrel{i_{j}}{\to} \hE_{j}^{2} \stackrel{p_{j}}{\to} \hE_{j}^{3})$ in $D^{b}(S\times Y)$ denote the corresponding sequence.
  On a morphism $g \colon S' \to S$, the presheaf $\Hom(f_{1},f_{2})$ evaluates to
  \begin{align*}
    \Hom(g^{*}\hE_{1}^{\bullet}, g^{*}\hE_{2}^{\bullet}) = \ker \left(\bigoplus_{k = 1}^{3}\Hom(f^{*}\hE_{1}^{k}, f^{*}\hE_{2}^{k}) \stackrel{q}{\to} \Hom(f^{*}\hE_{1}^{1}, f^{*}\hE_{2}^{2}) \oplus \Hom(f^{*}\hE_{1}^{2}, f^{*}\hE_{2}^{3})\right),
  \end{align*}
  where
  \[
    q = \begin{pmatrix}
      i_{2}\circ (-) & (-) \circ -i_{1} & 0 \\
      0 & p_{2} \circ (-) & (-) \circ -p_{1}
      \end{pmatrix}
    \]
  is the map representing the commutativity constraints.
  Let $Kq$ be the corresponding map of complexes on $S$, so
  \[
    Kq \colon \ker \left(\bigoplus_{j = 1}^{3}\CHom(\hE_{1}^{j}, \hE_{2}^{j}) \stackrel{Cq}{\to} \CHom(\hE_{1}^{1}, \hE_{2}^{2}) \oplus \CHom(\hE_{1}^{2}, \hE_{2}^{3})\right)
  \]
  Since each $\CHom(\hE_{1}^{k},\hE_{2}^{k'})$ has tor-amplitude in $[0,\infty)$, so does $C(Kq)[-1]$.
  Hence
  \[
    \Hom(g^{*}\hE_{1}^{\bullet}, g^{*}\hE_{2}^{\bullet}) = H^{0}(S', g^{*}C(Kq)[-1]),
  \]
  which as in Lemma \ref{thm:AIsLinearAlgebraic} shows that the Hom-functor is locally coherent representable.

  To show that $\A^{(2)}$ is algebraic linear, it is enough to show that $\A^{(2)}$ is a locally closed substack of $\Seq_{\A}$.
  Let $S$ be a scheme and let $f \colon S \to \Seq_{\A}$ be a morphism corresponding to a sequence $\hE_{1} \stackrel{i}{\to} \hE_{2} \stackrel{p}{\to} \hE_{3}$ of objects in $D^{b}(S \times Y)$.
  The morphism $f$ factors through $\A^{(2)}$ if and only if this sequence is an exact triangle.
  This firstly requires $p \circ i = 0$, which defines a closed subscheme of $S$, since $\Hom(\hE_{1}, \hE_{3})$ is affine over $S$.
  
  Replacing $S$ by this subscheme, we may assume that $p \circ i = 0$.
  Let $h \colon \hE_{2} \to C(i)$ be the canonical map to the cone $C(i)$ of $i$.
  Since $\Hom(\hE_{1}[1], \hE_{3}) = 0$, there is a unique map $j \colon C(i) \to \hE_{3}$ such that $p = j \circ h$.
  The given sequence is exact if and only if $j$ is an isomorphism.
  Let $Z$ be the support of $C(j)$, which is a closed subset of $S \times Y$.
  If $\pi \colon S \times Y \to S$ is the projection, then $S \setminus \pi(Z)$ is open, and equals $S \times_{\Seq_{\A}} \A^{(2)}$.
  This shows that $\A^{(2)}$ is locally closed in $\Seq_{\A}$.
\end{proof}

\begin{lem}
If $\hD$ is a triangulated category and $\hC$ is a subcategory closed under extensions such that $\Hom(E,F[-1]) = 0$ for $E, F \in \hC$, then the class of exact triangles with objects in $\hC$ give $\hC$ the structure of a Quillen exact category.
\end{lem}
\begin{proof}
  We refer to \cite{buhler_exact_2010} for the axioms of a Quillen exact category.

  Using the long exact sequence of shifted Hom-spaces associated to an exact triangle and the vanishing of negative extensions, we find that the set of exact triangles define a set of cokernel-kernel pairs on $\hC$.
  Axioms [E0] and [E0\textsuperscript{op}] follow.

  For [E1], suppose $f_{1}$ and $f_{2}$ are composable admissible monics, which means that $C(f_{1})$ and $C(f_{2})$ are objects of $\hC$.
  By the octahedral axiom, $C(f_{2} \circ f_{1})$ is an extension of $C(f_{2})$ by $C(f_{1})$, and since $\hC$ is closed under extensions, it follows that $C(f_{2} \circ f_{1}) \in \hC$, hence $f_{2} \circ f_{1}$ is an admissible monic.
  Axiom [E1\textsuperscript{op}] is similar.

  For [E2\textsuperscript{op}], let $A \to B$ be an admissible epic in $\hC$, and let $B' \to B$ be an arbitrary morphism of $\hC$.
  Define $A'$ as $C(A \oplus B' \to B)[-1]$, and let $f \colon A \to B'$ be the natural map.
  Applying the octahedral axiom to the sequence of morphisms $A' \to B' \to B' \oplus A$ we find that $C(f)$ equals the cone of the map $A \to B$, and this shows both that $A' \in \hC$ and that $A' \to B'$ is an admissible epic.
  Finally, for every $E \in \hC$, we have a long exact sequence beginning with
  \[
    0 \to \Hom(E, A') \to \Hom(E, A \oplus B') \to \Hom(E, B),
  \]
  which shows that $A'$ is the pullback of $A \to B$ along $B' \to B$.
\end{proof}

\begin{cor}
  \label{thm:QuillenExact}
  For any scheme $S$, the objects of the category $\A^{(2)}(S)$ define the structure of a Quillen exact category on the category $\A(S)$.
\end{cor}

\begin{lem}
  \label{thm:Karoubian}
  For any scheme $S$, the category $\A(S)$ is Karoubian.
\end{lem}
\begin{proof}
  By Assumption \ref{categoryAssumption}, the category has finite direct sums.
  It remains to show that it is idempotent complete.
  By \cite[Sec.~3]{bokstedt_homotopy_1993}, the unbounded derived category $D_{\text{QCoh}}(S \times Y)$ is idempotent complete, so any idempotent $p \colon \hE \to \hE$ with $\hE \in \A(S)$ induces a splitting $\hE = \hE_{1} \oplus \hE_{2}$ in $D_{\text{QCoh}}(S \times Y)$, and then both $\hE_{1}$ and $\hE_{2}$ must be perfect.
  By our assumptions on $\A$, this implies that $\hE_{1}, \hE_{2} \in \A(S)$.
\end{proof}

Let $a_{i} \colon \A^{(2)} \to \A$ denote the morphism taking the sequence $\hE_{1} \to \hE_{2} \to \hE_{3}$ to $\hE_{i}$.
\begin{lem}
  \label{thm:representable}
  The morphism $a_{2} \colon \A^{(2)} \to \A$ is representable in the sense of \cite[Def.~1.49]{1612.00372}
\end{lem}
\begin{proof}
  Any endomorphism of an exact sequence $\hE_{1} \to \hE_{2} \to \hE_{3}$ is determined by the induced endomorphism of $\hE_{2}$, and the claim follows.
\end{proof}


\begin{lem}
  \label{thm:finiteType}
  The morphism $a_{1} \times a_{3} \colon \A^{(2)} \to \A \times \A$ is of finite type.
\end{lem}
\begin{proof}
  Let $S$ be an affine scheme, and let $f_{1}, f_{2} \colon S \to \A$ be morphisms.
  We must show that the fibre product stack $\X := S \times_{\A^2} \A^{(2)}$ is of finite type over $S$.
  The stack $\X$ parametrises exact sequences $f_{1}^{*}\hU \to \hE \to f_{2}^{*}\hU$, and the automorphisms are those that are the identity on $f_{1}^{*}\hU$ and $f_{2}^{*}\hU$.
  The complex $\CHom(f_{2}, f_{1})$ is represented by a finite complex of finite rank projectives $0 \to K^{0} \to K^{1} \stackrel{d_{1}}{\to} K^{2} \to \ldots$.
  Let $Y \to S$ be the affine scheme defined by $\ker d_{1}$, \ie~for any morphism $g \colon S' \to S$, we have $Y(S') = H^{0}(S', \ker (g^{*}K^{1} \stackrel{g^{*}d_{1}}{\to} g^{*}K^{2}))$.
  There is a homomorphism
  \begin{equation}
    \label{eqn:homomorphism}
    H^{0}(S', \ker (g^{*}K^{1} \stackrel{g^{*}d_{1}}{\to} g^{*}K^{2})) \to H^{1}(S', g^{*}K^{\bullet}) = \Ext^{1}(g^{*}f_{2}^{*}\hU, g^{*}f_{1}^{*}\hU),
  \end{equation}
  which defines a morphism $Y \to \X$.
  If $S'$ is affine, then (\ref{eqn:homomorphism}) is an isomorphism, which implies that the functor $Y(S) \to \X(S')$ is essentially surjective.
  Thus the morphism $Y \to \X$ is surjective, and since $Y$ is of finite type over $S$, so is $\X$.
\end{proof}

\begin{rmk}
  One can show that the stack $\X$ in the proof above is equivalent to the quotient stack $[\ker d_{1}/K^{0}]$, where $K^{0}$ acts additively on $\ker d_{1}$ via $d_{0}$.
\end{rmk}


%
%
%
%
%
\bibliographystyle{alpha}
\bibliography{biblio}

\end{document}